\documentclass[11pt]{amsart}
\usepackage{amsmath,amssymb,url,color,tikz,colortbl,stmaryrd}
\usetikzlibrary{patterns}
\usepackage[all]{xy}
\usepackage{young}
\usepackage{enumerate}
\usepackage[margin=0.88in]{geometry}
\usepackage{hyperref}
\usepackage{amsfonts}
\usepackage{mathdots}
\usepackage{xcolor}
\definecolor{amethyst}{rgb}{0.54 0.17 0.89}

\usepackage{graphicx}
\usepackage{listings}
\usepackage{tram}
\usepackage{manfnt}
\usepackage{url,color,tikz,euscript}
\usepackage{tikz-cd}
\usepackage{bbding}
\usepackage{aircraftshapes}
\usetikzlibrary{arrows,automata}
\usetikzlibrary{decorations.pathreplacing}

\usepackage[colorinlistoftodos,color=red!70]{todonotes}

\makeatletter
\def\l@subsection{\@tocline{2}{0pt}{2.5pc}{5pc}{}}
\makeatother

\theoremstyle{theorem}
\newtheorem{thm}{Theorem}[section]
\newtheorem{theorem}[thm]{Theorem}

\newtheorem{lemma}[thm]{Lemma}
\newtheorem{prop}[thm]{Proposition}
\newtheorem{proposition}[thm]{Proposition}
\newtheorem{cor}[thm]{Corollary}
\newtheorem{corollary}[thm]{Corollary}

\theoremstyle{definition}

\newtheorem{definition}[thm]{Definition}

\newtheorem{example}[thm]{Example}

\newtheorem{remark}[thm]{Remark}

\usepackage{scalerel}
\usepackage{mathtools}

\def\Fiber{\mathsf{Fiber}}
\def\Scale{\mathsf{Scale}}
\def\scale{\mathsf{Scale}}
\def\colscale{\mathsf{ColScale}}
\def\Scroll{\mathsf{Scroll}}
\def\Tape{\mathsf{Tape}}
\def\Table{\mathsf{Table}}
\def\Vector{\mathsf{Vector}}
\def\Snake{\mathsf{Snake}}
\def\Cosnake{\mathsf{CoSnake}}
\def\Ouroboros{\mathsf{Ouro}}
\def\Coouroboros{\mathsf{CoOuro}}
\def\Slither{\mathsf{Slither}}
\def\Coslither{\mathsf{CoSlither}}
\def\Head{\mathsf{Head}}
\def\Tail{\mathsf{Tail}}
\def\Live{\mathsf{Live}}
\def\Swallow{\mathsf{Swal}}
\def\Coswallow{\mathsf{CoSwal}}
\def\Affsnake{\Snake^{\text\Plane}\!}
\def\Affcosnake{\Cosnake^{\text\Plane}\!}
\def\0{\color{gray}0}
\def\1{\bf 1}

\definecolor{Red}{rgb}{.9,0,0}
\definecolor{Blue}{rgb}{0,0,.9}
\definecolor{Green}{HTML}{006600}

\def\hang{\hangindent 24pt}
\def\d@nger{\medbreak\begingroup\clubpenalty=10000
  \def\par{\endgraf\endgroup\medbreak} 
  \noindent\hang\hangafter=-2
  \hbox to0pt{\hskip-\hangindent\dbend\hfill}}
\outer\def\danger{\d@nger}

\newcommand{\p}{{\bf p}}
\newcommand{\q}{{\bf q}}
\newcommand{\ff}{\mathbb{F}}
\newcommand{\rr}{\mathbb{R}}
\newcommand{\zz}{\mathbb{Z}}

\newcommand{\cc}{\mathfrak{C}}
\newcommand{\nn}{\mathbb{N}}

\renewcommand{\ss}{\mathfrak{S}}

\newcommand{\lcm}{\operatorname{lcm}}

\newcommand{\Tog}{\operatorname{Tog}}

\definecolor{green}{HTML}{006600}

\newcommand{\dfn}[1]{\textcolor{blue}{\emph{#1}}}

\catcode`@=12 
\renewcommand{\labelitemii}{\scriptsize\raise2pt\hbox{$\!{\blacktriangleright}$}}

\newcommand{\cala}{\mathcal{A}}

\newcommand{\calc}{\mathcal{C}}

\newcommand{\cali}{\mathcal{I}}
\newcommand{\calj}{\mathcal{J}}

\newcommand{\call}{\mathcal{L}}

\newcommand{\cals}{\mathcal{S}}
\newcommand{\calt}{\mathcal{T}}

\newcommand{\calv}{\mathcal{V}}

\newcommand{\calx}{\mathcal{X}}

\renewcommand{\bar}{\overline}

\newcommand{\E}{E}

\newcommand\<{\left<}
\renewcommand\>{\right>}

\newcommand\longto{\longrightarrow}

\DeclareMathOperator{\codeg}{codeg}
\DeclareMathOperator{\row}{Row}

\newcommand\Tstrut{\rule{0pt}{2ex}}     \newcommand\Bstrut{\rule[-.9ex]{0pt}{0pt}}

\title{Torsors and Tilings from Toric Toggling}
\author[C. Defant, M. Joseph, M. Macauley, A. McDonough]{Colin Defant, Michael Joseph, Matthew Macauley, Alex McDonough}

\begin{document}

\begin{abstract}
Much of dynamical algebraic combinatorics focuses on global dynamical systems defined via maps that are compositions of local toggle operators. The second author and Roby studied such maps that result from toggling independent sets of a path graph. We investigate a ``toric'' analogue of this work by analyzing the dynamics arising from toggling independent sets of a cycle graph. %Two commuting bijections on the set of ``live'' entries in the grids generated by the dynamics define torsors for what we call the infinite \emph{snake group} and the finite \emph{ouroboros groups}. By studying covering maps between these, we deduce precise combinatorial properties of the corresponding orbits. 
Each orbit in the dynamical system can be encoded via a grid of 0s and 1s; two commuting bijections on the set of 1s in this grid produce torsors for what we call the infinite \emph{snake group} and the finite \emph{ouroboros groups}. By studying related covering maps, we deduce precise combinatorial properties of the orbits. Because the snake and ouroboros groups are abelian, they define tilings of cylinders and tori by parallelograms, which we also characterize. Many of the ideas developed here should be adaptable both to other toggle actions in combinatorics and to other cellular automata. 
\end{abstract}

\keywords{Dynamical algebraic combinatorics, independent set, toggle group, snake group, snakes on a plane, ouroboros on a torus, torsor}

\maketitle
\tableofcontents

\section{Introduction} 
\subsection{A motivating example} 

It has been said that mathematics is the study of patterns, and this paper is all about understanding and analyzing interesting patterns that arise from a simple combinatorial action. Rather than write this paper in a traditional format starting with an overview of the background followed by a list of necessary definitions, we will begin with an example that illustrates the types of curious patterns and structures that attracted us to this problem in the first place. The action we are studying is elementary enough that it can be easily understood by a student in elementary school, yet the mathematical ideas that arise come from the interplay of combinatorics, group theory, number theory, and algebraic topology. The commuting bijections that make this all possible---they define torsors and tilings from orbits of independent sets---also appear in other actions of combinatorial objects and certain cellular automata. 

An independent set of the cycle graph $\calc_n$ can be viewed as a cyclic binary string $v_1,\dots,v_n$ such that no two adjacent entries are both $1$, where ``cyclic'' means that $v_1$ and $v_n$ are adjacent. The \emph{toggle operation} at position $k$ is the function $\tau_k$ that ``attempts to flip'' the $k^{\rm th}$ bit. Specifically, if $v_k=1$, then $\tau_k$ flips it to $0$. On the other hand, if $v_k=0$, then $\tau_k$ flips it to $1$ if doing so does not introduce consecutive $1$s; otherwise, it fixes the $k^{\rm th}$ bit. In this paper, we will study the action of iteratively toggling the bits of our binary string in the order $\tau_1,\dots,\tau_n$, and we will denote this by the map $\tau=\tau_n\circ\cdots\circ\tau_1$. Given an initial cyclic binary string $x^{(0)}$, let $x^{(1)}=\tau\left(x^{(0)}\right)$, $x^{(2)}=\tau\left(x^{(1)}\right)$, and so on. Eventually, after some $m\geq 1$ number of steps, we will return to our original string. That is, $x^{(m+i)}=x^{(i)}$ for all $i$. An example of an orbit $x^{(0)},\dots,x^{(m-1)}$ of this action with $n=12$ and $m=15$ is shown in Figure~\ref{fig:motivating-example}.

\colorlet{color1}{Red}
\colorlet{color2}{Blue}
\colorlet{color3}{orange}
\colorlet{color4}{green!95!black}
\colorlet{color5}{orange}
\colorlet{color6}{green!95!black}
\colorlet{color7}{orange}
\colorlet{color8}{green!95!black}
\begin{figure}[!ht]
\setlength{\tabcolsep}{2.5pt}
\renewcommand{\arraystretch}{.4}
\begin{tabular}{ccccccccccccc}
$x$ & $v_1$ & $v_2$ & $v_3$ & $v_4$ & $v_5$ & $v_6$ &
$v_7$ & $v_8$ & $v_9$ & $v_{1\!0}$ & $v_{1\!1}$ & $v_{1\!2}$ \Bstrut
\\\hline
$x^{(0)}$ & \1 & \0 & \1 & \0 & \1 & \0 & \0 & \0 & \1 & \0 & \1 & \0 \\
$x^{(1)}$ & \0 & \0 & \0 & \0 & \0 & \1 & \0 & \0 & \0 & \0 & \0 & \1 \\
$x^{(2)}$ & \0 & \1 & \0 & \1 & \0 & \0 & \1 & \0 & \1 & \0 & \0 & \0 \\
$x^{(3)}$ & \0 & \0 & \0 & \0 & \1 & \0 & \0 & \0 & \0 & \1 & \0 & \1 \\
$x^{(4)}$ & \0 & \1 & \0 & \0 & \0 & \1 & \0 & \1 & \0 & \0 & \0 & \0 \\
$x^{(5)}$ & \0 & \0 & \1 & \0 & \0 & \0 & \0 & \0 & \1 & \0 & \1 & \0 \\
$x^{(6)}$ & \1 & \0 & \0 & \1 & \0 & \1 & \0 & \0 & \0 & \0 & \0 & \0 \\
$x^{(7)}$ & \0 & \1 & \0 & \0 & \0 & \0 & \1 & \0 & \1 & \0 & \1 & \0 \\
$x^{(8)}$ & \0 & \0 & \1 & \0 & \1 & \0 & \0 & \0 & \0 & \0 & \0 & \1 \\
$x^{(9)}$ & \0 & \0 & \0 & \0 & \0 & \1 & \0 & \1 & \0 & \1 & \0 & \0 \\
$x^{(10)}$ & \1 & \0 & \1 & \0 & \0 & \0 & \0 & \0 & \0 & \0 & \1 & \0 \\
$x^{(11)}$ & \0 & \0 & \0 & \1 & \0 & \1 & \0 & \1 & \0 & \0 & \0 & \1 \\
$x^{(12)}$ & \0 & \1 & \0 & \0 & \0 & \0 & \0 & \0 & \1 & \0 & \0 & \0 \\
$x^{(13)}$ & \0 & \0 & \1 & \0 & \1 & \0 & \1 & \0 & \0 & \1 & \0 & \1 \\
$x^{(14)}$ & \0 & \0 & \0 & \0 & \0 & \0 & \0 & \1 & \0 & \0 & \0 & \0\Bstrut \\ \hline \Tstrut
{\bf Sum:} & \bf{3} & \bf{4} & \bf{5} & \bf{3} & \bf{4} & \bf{5} & \bf{3} & \bf{4} & \bf{5} & \bf{3} & \bf{4} & \bf{5} 
\end{tabular}
\caption{An orbit of size $15$ consisting of independent sets of the cycle graph $\calc_{12}$.}\label{fig:motivating-example}
\end{figure}

Toggle actions on combinatorial objects are important in the field of \emph{dynamical algebraic combinatorics}; in the next subsection, we will provide some context, background, and motivation for studying them, and we will discuss why independent sets are of interest. For now, let us return to the example in Figure~\ref{fig:motivating-example} and point out that for this particular orbit, the cyclic string $\cdots345345\cdots$ consisting of the \emph{column sums} has period~$3$. Our computational experiments suggested that this period must be odd in any orbit. This was the first curious property that we wanted to understand. Another pattern we observed is that if one reads the \emph{orbit table} as one reads a book---across each row, from top to bottom---the resulting string (called the \emph{orbit vector}) consists of several repeating copies of a shorter string. For example, the table in Figure~\ref{fig:motivating-example} has $180$ entries, but it is easy to check that it is simply four repeated copies of the first 45 entries. For each fixed $n$, we also saw certain orbit sizes arise more often than others, and we wanted to better understand this. Finally, we saw patterns within the $1$s which held for any orbit table. Specifically, for every $1$, there is another $1$ either two positions to the right or one position diagonally down and to the right (where we allow ``wrapping around'' the end of the table and from bottom to top). Additionally, many local patterns of $1$s seemed to repeat throughout the tables. Notice how Figure~\ref{fig:motivating-example} has many examples of $10101$ substrings, as well as three consecutive diagonal $1$s. This regularity of patterns suggested that there could be some hidden algebraic structure. Indeed, there turns out to be a simply transitive action of a particular abelian group on the \emph{live entries} (the $1$s) in any orbit table. In other words, the live entries are endowed with a group structure, but there is no distinguished identity element; this is called a \emph{torsor}. This group, which is an invariant of the orbit, encodes a number of combinatorial properties of the original action. The fact that this group is abelian means that the it defines a regular tiling of an infinite cylinder by parallelograms. The combinatorial patterns that we initially saw, and much more, can be explained through this algebraic lens. In this paper, we will develop the theory of these actions, and we will prove a number of theorems about them. 

This paper is organized as follows. We will conclude the introduction in Section~\ref{subsec:toggling} by reviewing the notion of toggling and further discussing how we became interested in this problem. This can be considered ``optional,'' but it gives the rest of the paper context and should be helpful for the non-expert. In Section~\ref{sec:scrolls}, we formalize the orbits generated by toggling independent sets in two ways:\ as bi-infinite strings called \emph{ticker tapes} and as bi-infinite tables called \emph{scrolls}, which naturally live on cylinders. We introduce commuting bijections called the \emph{successor functions} and \emph{co-successor functions}, which define equivalence classes called \emph{snakes} and \emph{co-snakes}. These correspond to cosets of an infinite abelian \emph{snake group}, which acts simply transitively on the set of live entries. Studying this group and its actions---both on the scrolls and on their universal covers---helps us understand toggling dynamics. Not only does this group make the set of live entries into a torsor, but its action gives the (co-)snakes a regular ``shape'' called their \emph{(co-)slither}. These are invariants of the orbit, and they associate to the orbit tilings of the cylinder and the plane (its universal cover) by parallelograms. In Section~\ref{sec:classification}, we completely classify the dynamics of the action (Theorem~\ref{thm:2a+3b+4c=n+1}) by characterizing the possible orbits (Corollary~\ref{cor:2a+3b+4c=n+1}). In Section~\ref{sec:tables}, we take various quotients of our scrolls to obtain finite orbit tables that naturally embed into tori. The \mbox{(co-)snakes} merge into equivalence classes called \emph{(co-)ouroboroi}, inducing a homomorphism from the snake group to the finite abelian \emph{ouroboros group}. The snake group action descends to a simply transitive action on the live entries in the quotient tables, endowing them with a torsor structure for the ouroboros group. In Section~\ref{sec:sum-vectors}, we return to the original topic that drew us to this problem: the period of the so-called \emph{sum vector}. The odd periodicity  (Corollary~\ref{cor:sum-vector}) is a straightforward consequence of the theory developed in the previous sections. However, we prove a much stronger result by characterizing which odd numbers arise as the period of the sum vector of some orbit for a given $n$ (Theorem~\ref{thm:snakeconstruction}). We conclude this paper in Section~\ref{sec:conclusions} with discussions of open problems, how this framework can be broadened to other actions from dynamical algebraic combinatorics, and how it arises in certain cellular automata~\cite{david2021toggling}. The crux of why all this works is due to the existence of commuting bijections that act simply transitively on the live entries. The fact that this phenomenon appears in other problems from combinatorics and other fields such as cellular automata suggests that the theory is of general interest.

\subsection{Why toggle independent sets?}\label{subsec:toggling}

The notion of toggling has recently gained  considerable interest within the field of dynamical algebraic combinatorics; for surveys, see \cite{roby2016dynamical,striker2017dynamical}. Toggles yield an action on a collection $\call$ of subsets of a finite set. Common examples of the set $\call$ are the set of order ideals of a poset \cite{striker2012promotion}, the set of antichains of a poset \cite{joseph2019antichain}, the set of noncrossing partitions \cite{einstein2016noncrossing}, or the set of independent sets of a path graph \cite{joseph2018toggling}. In what follows, we may assume $\call$ is a collection of subsets of the set $[n]:=\{1,\ldots,n\}$. 
For $k\in[n]$, the \dfn{toggle} $\tau_k\colon\call\to\call$ is defined by
\begin{equation}\label{eqn:tau_k}
\tau_k(E)=\begin{cases}E\cup\{k\} & \text{if }k\not\in E \text{ and } E\cup\{k\}\in \call \\
E\setminus \{k\} & \text{if }k\in E \text{ and } E\setminus \{k\}\in \call \\
E & \text{otherwise}.
\end{cases}
\end{equation}

By construction, each toggle is a bijection, so this defines a group of permutations
\[
\Tog(\call)=\left<\tau_1,\dots,\tau_n\right>
\]
called the \dfn{toggle group}. Since each $\tau_k$ is an involution, $\Tog(\call)$ is a quotient of a Coxeter group. Following Coxeter theory, we will define a \dfn{Coxeter element} to be the product of all toggles in some order. Though sometimes it can be of interest to classify this group, work on toggling is usually focused elsewhere, such as understanding which classical bijections can be decomposed as products of toggles, which combinatorial statistics are invariant or \emph{homomesic} under toggling, and how the dynamics change under different toggle orders. 

The first objects to be toggled were order ideals of a poset $P$. In 1974, Brouwer studied a bijection on the set $\calj(P)$ of order ideals that sends $I$ to the order ideal generated by the minimal elements of $P\setminus I$ \cite{brouwer1974period}. In 1995, Cameron and and Fon-Der-Flaass showed that in this bijection can be constructed in graded posets by toggling each element of $P$ once, in a particular order: by rows, from top-to-bottom \cite{cameron1995orbits}. 
In 2012, Striker and Williams observed that the conjugate Coxeter element, toggling by columns, was closely related to the classic operation of \emph{promotion} on semistandard Young tableaux. This motivated them to name the aforemtioned bijection \emph{rowmotion}, denoted $\row_J$, and also to formalize and name toggles and the toggle group. 

In \cite{cameron1995orbits}, Cameron and and Fon-Der-Flaass also studied what is now known as \emph{antichain rowmotion}: the map $\row_{\cala}$ sending $A$ to the set of minimal element(s) of the complement of the order ideal $I(A):=\big\{x\in P\mid x\leq_P a\text{ for some $a\in A$}\big\}$.
Panyushev studied this map on root posets~\cite{panyushev2009orbits}; this is one of the works that was a major impetus for the development of dynamical algebraic combinatorics.
The second author of the present paper studied this bijection in the context of toggling \cite{joseph2019antichain}. Even though order ideal and antichain rowmotion are conjugate via $\row_J\circ I=I\circ\row_{\cala}$, relating their factorizations into individual toggles is surprisingly trickier. 

The aforementioned work led to the notion of toggling other combinatorial objects~\cite{striker2015toggle}, as discussed above and defined in Eq.~\eqref{eqn:tau_k}. Since every antichain is an independent set in a certain graph, toggling independent sets was a natural next step. Additionally, the second and third author (with others) considered toggling noncrossing partitions, viewed as collections of arcs \cite{einstein2016noncrossing}. This revealed factorizations of existing actions such as the Kreweras complement and Simion--Ullman involution in terms of toggles. Like with antichains, toggling noncrossing partitions is a special case of toggling independent sets. All of this motivated the third author and Roby to investigate independent set toggling on its own. This is a difficult problem in general, so they started with a path graph \cite{joseph2018toggling}. The goals of the authors of \cite{joseph2018toggling} were to study combinatorial statistics called \emph{homomesies} arising from toggling independent sets of path graphs, as well as to prove several conjectures of Propp. The actual toggle groups were later classified in \cite{numata2021action}. 

In this article, we direct our attention to toggling independent sets of a cycle graph.\footnote{This transition from path graphs to cycle graphs explains our use of the word ``toric'' in the title of this article. Indeed, the articles \cite{develin2016toric,defant2023toric} use the word ``toric'' to refer to cyclic analogues of objects that has previously been considered in ``linear'' settings. The paper \cite{defant2023toric} also uses local ``toggle operators'' to define an operator called \emph{toric promotion}, although the toggles considered there are different from the ones considered here.} When toggling independent sets of graphs, certain aspects of cycle graphs end up being more complicated than path graphs, but others end up being nicer. For example, when toggling independent sets of a path, all Coxeter elements are conjugate; this simplifies certain arguments and makes some results hold for all Coxeter elements. In contrast, when working with the independent sets of the cycle $\calc_n$, the Coxeter elements fall into $n-1$ conjugacy classes \cite{develin2016toric}. On the other hand, one nice property of cycle graphs is that they are vertex-transitive. In the end, different patterns and different questions arise over cycle graphs than over path graphs, in ways that were initially not clear to us, and the extent of which ultimately surprised us. This problem became much more algebraic in nature, and a broader mathematical theory emerged.

\section{Dynamics and actions on infinite sets}\label{sec:scrolls}
\subsection{Scrolls vs.\ ticker tapes}\label{subsec:scrolls}

Let $\calc_n$ denote the cycle graph with vertex set $V(\calc_n)=\zz_n$ (the integers modulo $n$) and edges $\{i,i+1\}$ (including $\{n,1\}$). We often identify $\zz_n$ with $[n]:=\{1,\ldots,n\}$ or $\{0,\ldots,n-1\}$ in the obvious manner. Throughout the paper, let $n\geq 2$. Though an independent set of $\calc_n$ is a subset of $[n]$, it will usually be easiest for us to denote it as a length-$n$ binary string, with the requirement that it does not contain a pair of consecutive $1$s, including those that ``wrap around'' the end of the word. We write $\mathcal{I}_n$ for the collection of independent sets of $\calc_n$, regardless of which notation we use. We will let $\ff_2=\{0,1\}$ denote the bits of the binary string, i.e., the states of the vertices. We may write a binary $n$-tuple either as a string $v_1\cdots v_n$ or as a vector $(v_1,\dots,v_n)$ in $\ff_2^n$.

For each vertex $k\in[n]$, there is a bijective \dfn{toggle operation} $\tau_k\colon\mathcal{I}_n\to\mathcal{I}_n$ that adds $k$ to an independent set $I$ if $k\not\in I$ and $I\cup\{k\}$ is an independent set, removes $k$ from $I$ if $k\in I$, and fixes $I$ otherwise.\footnote{Alternatively, this can be formalized by taking $\call=\cali_n$ in the definition of generalized toggling in Eq.~\eqref{eqn:tau_k}. Notice that the condition ``\emph{and $I\setminus\{k\}\in\call$}'' in the middle line is actually unnecessary, because removing a vertex from an independent set always leaves it independent.} Throughout this paper, we will toggle in the order $1,\dots,n$; that is, we consider the map
\[
\tau\in\Tog(\cali_n),\qquad
\tau:=\tau_n\circ\cdots\circ\tau_1.
\]
Sometimes, we will write $v_1,\dots,v_n$ rather than $1,\dots, n$ for extra emphasis, like we did in the header of Figure~\ref{fig:motivating-example}.

In the remainder of this subsection, we will describe two formats for viewing the dynamics that result from toggling independent sets of $\calc_n$. The first is a vertically bi-infinite periodic table of 0s and 1s called the \emph{scroll}, and the second is a bi-infinite periodic sequence called the \emph{ticker tape} that we get from reading the scroll like one reads a book, across the columns and downward row-by-row. Each format has its advantages and disadvantages. Certain features are more prominent in one while hidden in the other or are notationally simpler in one than the other. 

Let $x=x^{(0)}=(x_1,\dots,x_n)\in\cali_n$, and let $x^{(i)}=\tau^i(x)$ be the result of iterating $\tau$ exactly $i$ times from $x$. Since $\tau$ is a bijection on $\cali_n$, we can define this for all $i\in\zz$. Consider the table with $n$ columns, indexed by $j=1,\dots,n$, and rows indexed by $i\in\zz$, reading downward. The $(i,j)$-entry $X_{i,j}$ is the state of vertex $v_j$ in $x^{(i)}$. In other words, the $i^{\rm th}$ row is just the vector $x^{(i)}$. This infinite table is called the \dfn{scroll} of $x$, and we will denote it by $\cals=(X_{i,j})=\Scroll(x)$.
The scroll of $x=00001010000\in\ff_2^{11}$, shown in Figure~\ref{fig:example}, will be our new running example.

\colorlet{color1}{Red}
\colorlet{color2}{Blue}
\colorlet{color3}{Red}
\colorlet{color4}{Blue}
\colorlet{color5}{purple}
\colorlet{color6}{Green}
\colorlet{color7}{orange}
\colorlet{color8}{amethyst}

\begin{figure}[!ht]
\begin{tikzpicture}
\setlength{\tabcolsep}{2.5pt}
\renewcommand{\arraystretch}{.4}
\begin{scope}
\tikzstyle{every node}=[font=\small,anchor=south]
\node at (0,0) {
\begin{tabular}{cccccccccccc}
$x$ & $v_1$ & $v_2$ & $v_3$ & $v_4$ & $v_5$ & $v_6$ &
$v_7$ & $v_8$ & $v_9$ & $v_{1\!0}$ & $v_{1\!1}$ \Bstrut
\\\hline
$x^{(0)}$ & \0 & \0 & \0 & \0 & {\bf\color{color1}1} & \0 & {\bf\color{color1}1} & \0 & \0 & \0 & \0 \\
$x^{(1)}$ & {\bf\color{color2}1} & \0 & {\bf\color{color2}1} & \0 & \0 & \0 & \0 & {\bf\color{color1}1} & \0 & {\bf\color{color1}1} & \0 \\
$x^{(2)}$ & \0 & \0 & \0 & {\bf\color{color2}1} & \0 & {\bf\color{color2}1} & \0 & \0 & \0 & \0 & {\bf\color{color1}1} \\
$x^{(3)}$ & \0 & {\bf\color{color1}1} & \0 & \0 & \0 & \0 & {\bf\color{color2}1} & \0 & {\bf\color{color2}1} & \0 & \0 \\
$x^{(4)}$ & \0 & \0 & {\bf\color{color1}1} & \0 & {\bf\color{color1}1} & \0 & \0 & \0 & \0 & {\bf\color{color2}1} & \0 \\
$x^{(5)}$ & {\bf\color{color2}1} & \0 & \0 & \0 & \0 & {\bf\color{color1}1} & \0 & {\bf\color{color1}1} & \0 & \0 & \0 \\
$x^{(6)}$ & \0 & {\bf\color{color2}1} & \0 & {\bf\color{color2}1} & \0 & \0 & \0 & \0 & {\bf\color{color1}1} & \0 & {\bf\color{color1}1} \\ \vspace{-0.05 in} \\ 
\end{tabular}};
\node at (8,0) {
\begin{tabular}{cccccccccccc}
$x$ & $v_1$ & $v_2$ & $v_3$ & $v_4$ & $v_5$ & $v_6$ &
$v_7$ & $v_8$ & $v_9$ & $v_{1\!0}$ & $v_{1\!1}$ \Bstrut
\\\hline
$x^{(0)}$ & \0 & \0 & \0 & \0 & {\bf\color{color5}1} & \0 & {\bf\color{color6}1} & \0 & \0 & \0 & \0 \\
$x^{(1)}$ & {\bf\color{color3}1} & \0 & {\bf\color{color4}1} & \0 & \0 & \0 & \0 & {\bf\color{color7}1} & \0 & {\bf\color{color8}1} & \0 \\
$x^{(2)}$ & \0 & \0 & \0 & {\bf\color{color5}1} & \0 & {\bf\color{color6}1} & \0 & \0 & \0 & \0 & {\bf\color{color3}1} \\
$x^{(3)}$ & \0 & {\bf\color{color4}1} & \0 & \0 & \0 & \0 & {\bf\color{color7}1} & \0 & {\bf\color{color8}1} & \0 & \0 \\
$x^{(4)}$ & \0 & \0 & {\bf\color{color5}1} & \0 & {\bf\color{color6}1} & \0 & \0 & \0 & \0 & {\bf\color{color3}1} & \0 \\
$x^{(5)}$ & {\bf\color{color4}1} & \0 & \0 & \0 & \0 & {\bf\color{color7}1} & \0 & {\bf\color{color8}1} & \0 & \0 & \0 \\
$x^{(6)}$ & \0 & {\bf\color{color5}1} & \0 & {\bf\color{color6}1} & \0 & \0 & \0 & \0 & {\bf\color{color3}1} & \0 & {\bf\color{color4}1} \\ \vspace{-0.05 in} \\  
\end{tabular}};
\end{scope}
\end{tikzpicture}
\caption{The scroll of $x^{(0)}=(0,0,0,0,1,0,1,0,0,0,0)\in\ff_2^{11}$ consists of the seven rows $x^{(0)},\dots,x^{(6)}$ repeated indefinitely. This is shown twice, with different color schemes, to emphasize visual patterns among the \emph{live} (value of 1) entries that we will soon formalize as \emph{snakes} and \emph{co-snakes}, respectively. 
}\label{fig:example}
\end{figure}

The global dynamics of $\tau$ can be read off the scroll as one reads from a book: reading the rows from top to bottom, with each row read from left to right. This defines a bi-infinite sequence called the \dfn{ticker tape}, denoted $\calx=(X_k)=\Tape(x)$. To convert between ticker tape and scroll notation, let $X_1=X_{0,1}$, $X_2=X_{0,2}$, $X_3=X_{0,3}$, and so on, so that $X_{in+j}=X_{i,j}$. The ticker tape of the example in Figure~\ref{fig:example} is
\[ \small
\dots,
\underbrace{X_{-6},X_{-5},X_{-4},X_{-3},X_{-2},X_{-1},X_0}_{0,0,0,0,1,0,1},
\underbrace{X_1,X_2,X_3,X_4,X_5,X_6,X_7}_{0,0,0,0,1,0,1},\underbrace{X_8,X_9,X_{10},X_{11},X_{12},X_{13},X_{14}}_{0,0,0,0,1,0,1},\dots.
\]
Topologically, the ticker tape can be naturally embedded on an infinite line. In contrast, the scroll ``wraps around'' from the end of one row to the beginning of the next, so it is natural to view it as being embedded on a bi-infinite cylinder rather than on a plane. As such, it is always well-founded to speak of the entry immediately to the left or to the right of position $(i,j)$, even if it is in the first or the last column. Notationally, even though we index the columns by $j=1,\dots,n$, it will be convenient to set $X_{i,k+n}=X_{i+1,k}$ for each $k\in\mathbb Z$. 

At times, it will be more convenient to ``lift up'' the cylinder to its universal cover and work with points in the plane. Here, we think of a scroll as a map $\cals\colon\zz\times\zz_n\to\ff_2$, which naturally lifts to the \dfn{universal scroll} $\widehat\cals\colon\zz\times\zz\to\ff_2$, making the following diagram commute:
\begin{equation}\label{eqn:cd-universal}
\xymatrix{\zz\times\zz_{\color{white}n}\ar[d]_{\q}\ar@{-->}[rd]^{\widehat{\cals}} & \\ \zz\times\zz_n\ar[r]^\cals & \ff_2}\hspace{20mm}
\xymatrix{(i+k,j+kn)\ar@{|->}[d]_{\q}\ar@{|-->}[rd]^{\widehat{\cals}} & \\ (i,j)\ar@{|->}[r]^\cals & X_{i,j}}
\end{equation}
Note that the quotient map $\q$ is not quite the ``canonical'' quotient from the plane to a cylinder that reduces the second entry modulo $n$, because the row increases when we wrap around. There is no such analogue for lifting the ticker tape in this manner because its canonical domain is a subset of $\rr$, which is already simply connected.
Note that each infinite row of the universal scroll is a shifted copy of the ticker tape.
In this context, we will usually take $\zz_n=[n]=\{1,\dots,n\}$, rather than $\{0,\dots,n-1\}$, because we want to index the columns by the vertices, which are in $[n]$. A portion of the universal scroll of our running example from Figure~\ref{fig:example} is shown in Figure~\ref{fig:universal-scroll}. The shading is meant to highlight disjoint copies of orbits under the toggling map $\tau$.

\begin{figure}[!ht]
\begin{tikzpicture}[scale=1.2]
\tikzstyle{every node}=[font=\small,anchor=south,text=gray]
\colorlet{lightgrey}{black!10}
\colorlet{darkgrey}{black!90}
  \begin{scope}[shift={(-2.75,-.35)}]
 \colorlet{color1}{Red}
\colorlet{color2}{Blue}
\foreach \x in {0,...,6} {\node at (-.35,.35*\x) {$\cdots$};}
  \node at (0,2.1) {$0$}; \node at (.25,2.1) {$0$};
 \node at (.5,2.1) {$0$}; \node at (.75,2.1) {$0$};
 \node at (1,2.1) {$\color{color1}\mathbf{1}$}; \node at (1.25,2.1) {$0$};
 \node at (1.5,2.1) {$\color{color1}\mathbf{1}$}; \node at (1.75,2.1) {$0$};
 \node at (2,2.1) {$0$}; \node at (2.25,2.1) {$0$};
 \node at (2.5,2.1) {$0$};
\node at (0,1.75) {$\color{color2}\mathbf{1}$}; \node at (.25,1.75) {$0$};
 \node at (.5,1.75) {$\color{color2}\mathbf{1}$}; \node at (.75,1.75) {$0$};
 \node at (1,1.75) {$0$}; \node at (1.25,1.75) {$0$};
 \node at (1.5,1.75) {$0$}; \node at (1.75,1.75) {$\color{color1}\mathbf{1}$};
 \node at (2,1.75) {$0$}; \node at (2.25,1.75) {$\color{color1}\mathbf{1}$};
 \node at (2.5,1.75) {$0$};
 
   \node at (0,1.4) {$0$}; \node at (.25,1.4) {$0$};
 \node at (.5,1.4) {$0$}; \node at (.75,1.4) {$\color{color2}\mathbf{1}$};
 \node at (1,1.4) {$0$}; \node at (1.25,1.4) {$\color{color2}\mathbf{1}$};
 \node at (1.5,1.4) {$0$}; \node at (1.75,1.4) {$0$};
 \node at (2,1.4) {$0$}; \node at (2.25,1.4) {$0$};
 \node at (2.5,1.4) {$\color{color1}\mathbf{1}$};
 
   \node at (0,1.05) {$0$}; \node at (.25,1.05) {$\color{color1}\mathbf{1}$};
 \node at (.5,1.05) {$0$}; \node at (.75,1.05) {$0$};
 \node at (1,1.05) {$0$}; \node at (1.25,1.05) {$0$};
 \node at (1.5,1.05) {$\color{color2}\mathbf{1}$}; \node at (1.75,1.05) {$0$};
 \node at (2,1.05) {$\color{color2}\mathbf{1}$}; \node at (2.25,1.05) {$0$};
 \node at (2.5,1.05) {$0$};
 
 \node at (0,.7) {$0$}; \node at (.25,.7) {$0$};
 \node at (.5,.7) {$\color{color1}\mathbf{1}$}; \node at (.75,.7) {$0$};
 \node at (1,.7) {$\color{color1}\mathbf{1}$}; \node at (1.25,.7) {$0$};
 \node at (1.5,.7) {$0$}; \node at (1.75,.7) {$0$};
 \node at (2,.7) {$0$}; \node at (2.25,.7) {$\color{color2}\mathbf{1}$};
 \node at (2.5,.7) {$0$};
 
 \node at (0,.35) {$\color{color2}\mathbf{1}$}; \node at (.25,.35) {$0$};
 \node at (.5,.35) {$0$}; \node at (.75,.35) {$0$};
 \node at (1,.35) {$0$}; \node at (1.25,.35) {$\color{color1}\mathbf{1}$};
 \node at (1.5,.35) {$0$}; \node at (1.75,.35) {$\color{color1}\mathbf{1}$};
 \node at (2,.35) {$0$}; \node at (2.25,.35) {$0$};
 \node at (2.5,.35) {$0$};
 
  \node at (0,0) {$0$}; \node at (.25,0) {$\color{color2}\mathbf{1}$};
 \node at (.5,0) {$0$}; \node at (.75,0) {$\color{color2}\mathbf{1}$};
 \node at (1,0) {$0$}; \node at (1.25,0) {$0$};
 \node at (1.5,0) {$0$}; \node at (1.75,0) {$0$};
 \node at (2,0) {$\color{color1}\mathbf{1}$}; \node at (2.25,0) {$0$};
 \node at (2.5,0) {$\color{color1}\mathbf{1}$};
 \end{scope}
\begin{scope}[shift={(0,0)}]
\colorlet{color1}{Red}
\colorlet{color2}{Blue}
\draw [draw=lightgrey,fill=lightgrey] (-.12,.09) rectangle (2.63,2.54);
  \node at (0,2.1) {$0$}; \node at (.25,2.1) {$0$};
 \node at (.5,2.1) {$0$}; \node at (.75,2.1) {$0$};
 \node at (1,2.1) {$\color{color1}\mathbf{1}$}; \node at (1.25,2.1) {$0$};
 \node at (1.5,2.1) {$\color{color1}\mathbf{1}$}; \node at (1.75,2.1) {$0$};
 \node at (2,2.1) {$0$}; \node at (2.25,2.1) {$0$};
 \node at (2.5,2.1) {$0$};
 
   \node at (0,1.75) {$\color{color2}\mathbf{1}$}; \node at (.25,1.75) {$0$};
 \node at (.5,1.75) {$\color{color2}\mathbf{1}$}; \node at (.75,1.75) {$0$};
 \node at (1,1.75) {$0$}; \node at (1.25,1.75) {$0$};
 \node at (1.5,1.75) {$0$}; \node at (1.75,1.75) {$\color{color1}\mathbf{1}$};
 \node at (2,1.75) {$0$}; \node at (2.25,1.75) {$\color{color1}\mathbf{1}$};
 \node at (2.5,1.75) {$0$};
 
   \node at (0,1.4) {$0$}; \node at (.25,1.4) {$0$};
 \node at (.5,1.4) {$0$}; \node at (.75,1.4) {$\color{color2}\mathbf{1}$};
 \node at (1,1.4) {$0$}; \node at (1.25,1.4) {$\color{color2}\mathbf{1}$};
 \node at (1.5,1.4) {$0$}; \node at (1.75,1.4) {$0$};
 \node at (2,1.4) {$0$}; \node at (2.25,1.4) {$0$};
 \node at (2.5,1.4) {$\color{color1}\mathbf{1}$};
 
   \node at (0,1.05) {$0$}; \node at (.25,1.05) {$\color{color1}\mathbf{1}$};
 \node at (.5,1.05) {$0$}; \node at (.75,1.05) {$0$};
 \node at (1,1.05) {$0$}; \node at (1.25,1.05) {$0$};
 \node at (1.5,1.05) {$\color{color2}\mathbf{1}$}; \node at (1.75,1.05) {$0$};
 \node at (2,1.05) {$\color{color2}\mathbf{1}$}; \node at (2.25,1.05) {$0$};
 \node at (2.5,1.05) {$0$};
 
 \node at (0,.7) {$0$}; \node at (.25,.7) {$0$};
 \node at (.5,.7) {$\color{color1}\mathbf{1}$}; \node at (.75,.7) {$0$};
 \node at (1,.7) {$\color{color1}\mathbf{1}$}; \node at (1.25,.7) {$0$};
 \node at (1.5,.7) {$0$}; \node at (1.75,.7) {$0$};
 \node at (2,.7) {$0$}; \node at (2.25,.7) {$\color{color2}\mathbf{1}$};
 \node at (2.5,.7) {$0$};
 \node at (0,.35) {$\color{color2}\mathbf{1}$}; \node at (.25,.35) {$0$};
 \node at (.5,.35) {$0$}; \node at (.75,.35) {$0$};
 \node at (1,.35) {$0$}; \node at (1.25,.35) {$\color{color1}\mathbf{1}$};
 \node at (1.5,.35) {$0$}; \node at (1.75,.35) {$\color{color1}\mathbf{1}$};
 \node at (2,.35) {$0$}; \node at (2.25,.35) {$0$};
 \node at (2.5,.35) {$0$};
  \node at (0,0) {$0$}; \node at (.25,0) {$\color{color2}\mathbf{1}$};
 \node at (.5,0) {$0$}; \node at (.75,0) {$\color{color2}\mathbf{1}$};
 \node at (1,0) {$0$}; \node at (1.25,0) {$0$};
 \node at (1.5,0) {$0$}; \node at (1.75,0) {$0$};
 \node at (2,0) {$\color{color1}\mathbf{1}$}; \node at (2.25,0) {$0$};
 \node at (2.5,0) {$\color{color1}\mathbf{1}$};
 \end{scope}
 \begin{scope}[shift={(2.75,.35)}]
 \colorlet{color1}{Red}
\colorlet{color2}{Blue}
\foreach \x in {0,...,6} {\node at (2.95,.35*\x) {$\cdots$};}
  \node at (0,2.1) {$0$}; \node at (.25,2.1) {$0$};
 \node at (.5,2.1) {$0$}; \node at (.75,2.1) {$0$};
 \node at (1,2.1) {$\color{color1}\mathbf{1}$}; \node at (1.25,2.1) {$0$};
 \node at (1.5,2.1) {$\color{color1}\mathbf{1}$}; \node at (1.75,2.1) {$0$};
 \node at (2,2.1) {$0$}; \node at (2.25,2.1) {$0$};
 \node at (2.5,2.1) {$0$};
   \node at (0,1.75) {$\color{color2}\mathbf{1}$}; \node at (.25,1.75) {$0$};
 \node at (.5,1.75) {$\color{color2}\mathbf{1}$}; \node at (.75,1.75) {$0$};
 \node at (1,1.75) {$0$}; \node at (1.25,1.75) {$0$};
 \node at (1.5,1.75) {$0$}; \node at (1.75,1.75) {$\color{color1}\mathbf{1}$};
 \node at (2,1.75) {$0$}; \node at (2.25,1.75) {$\color{color1}\mathbf{1}$};
 \node at (2.5,1.75) {$0$};
   \node at (0,1.4) {$0$}; \node at (.25,1.4) {$0$};
 \node at (.5,1.4) {$0$}; \node at (.75,1.4) {$\color{color2}\mathbf{1}$};
 \node at (1,1.4) {$0$}; \node at (1.25,1.4) {$\color{color2}\mathbf{1}$};
 \node at (1.5,1.4) {$0$}; \node at (1.75,1.4) {$0$};
 \node at (2,1.4) {$0$}; \node at (2.25,1.4) {$0$};
 \node at (2.5,1.4) {$\color{color1}\mathbf{1}$};
   \node at (0,1.05) {$0$}; \node at (.25,1.05) {$\color{color1}\mathbf{1}$};
 \node at (.5,1.05) {$0$}; \node at (.75,1.05) {$0$};
 \node at (1,1.05) {$0$}; \node at (1.25,1.05) {$0$};
 \node at (1.5,1.05) {$\color{color2}\mathbf{1}$}; \node at (1.75,1.05) {$0$};
 \node at (2,1.05) {$\color{color2}\mathbf{1}$}; \node at (2.25,1.05) {$0$};
 \node at (2.5,1.05) {$0$};
 \node at (0,.7) {$0$}; \node at (.25,.7) {$0$};
 \node at (.5,.7) {$\color{color1}\mathbf{1}$}; \node at (.75,.7) {$0$};
 \node at (1,.7) {$\color{color1}\mathbf{1}$}; \node at (1.25,.7) {$0$};
 \node at (1.5,.7) {$0$}; \node at (1.75,.7) {$0$};
 \node at (2,.7) {$0$}; \node at (2.25,.7) {$\color{color2}\mathbf{1}$};
 \node at (2.5,.7) {$0$};
 \node at (0,.35) {$\color{color2}\mathbf{1}$}; \node at (.25,.35) {$0$};
 \node at (.5,.35) {$0$}; \node at (.75,.35) {$0$};
 \node at (1,.35) {$0$}; \node at (1.25,.35) {$\color{color1}\mathbf{1}$};
 \node at (1.5,.35) {$0$}; \node at (1.75,.35) {$\color{color1}\mathbf{1}$};
 \node at (2,.35) {$0$}; \node at (2.25,.35) {$0$};
 \node at (2.5,.35) {$0$};
  \node at (0,0) {$0$}; \node at (.25,0) {$\color{color2}\mathbf{1}$};
 \node at (.5,0) {$0$}; \node at (.75,0) {$\color{color2}\mathbf{1}$};
 \node at (1,0) {$0$}; \node at (1.25,0) {$0$};
 \node at (1.5,0) {$0$}; \node at (1.75,0) {$0$};
 \node at (2,0) {$\color{color1}\mathbf{1}$}; \node at (2.25,0) {$0$};
 \node at (2.5,0) {$\color{color1}\mathbf{1}$};
 \end{scope}
 \begin{scope}[shift={(-2.75,2.1)}]
 \colorlet{color2}{Red}
\colorlet{color1}{Blue}
\draw [draw=lightgrey,fill=lightgrey] (-.12,.09) rectangle (2.63,2.54);
\foreach \x in {0,...,6} {\node at (-.35,.35*\x) {$\cdots$};}
\foreach \x in {0,...,10} {\node at (.25*\x,2.45) {$\vdots$};}
\node at (-.35,2.45) {$\ddots$};
  \node at (0,2.1) {$0$}; \node at (.25,2.1) {$0$};
 \node at (.5,2.1) {$0$}; \node at (.75,2.1) {$0$};
 \node at (1,2.1) {$\color{color1}\mathbf{1}$}; \node at (1.25,2.1) {$0$};
 \node at (1.5,2.1) {$\color{color1}\mathbf{1}$}; \node at (1.75,2.1) {$0$};
 \node at (2,2.1) {$0$}; \node at (2.25,2.1) {$0$};
 \node at (2.5,2.1) {$0$};
   \node at (0,1.75) {$\color{color2}\mathbf{1}$}; \node at (.25,1.75) {$0$};
 \node at (.5,1.75) {$\color{color2}\mathbf{1}$}; \node at (.75,1.75) {$0$};
 \node at (1,1.75) {$0$}; \node at (1.25,1.75) {$0$};
 \node at (1.5,1.75) {$0$}; \node at (1.75,1.75) {$\color{color1}\mathbf{1}$};
 \node at (2,1.75) {$0$}; \node at (2.25,1.75) {$\color{color1}\mathbf{1}$};
 \node at (2.5,1.75) {$0$};
   \node at (0,1.4) {$0$}; \node at (.25,1.4) {$0$};
 \node at (.5,1.4) {$0$}; \node at (.75,1.4) {$\color{color2}\mathbf{1}$};
 \node at (1,1.4) {$0$}; \node at (1.25,1.4) {$\color{color2}\mathbf{1}$};
 \node at (1.5,1.4) {$0$}; \node at (1.75,1.4) {$0$};
 \node at (2,1.4) {$0$}; \node at (2.25,1.4) {$0$};
 \node at (2.5,1.4) {$\color{color1}\mathbf{1}$};
   \node at (0,1.05) {$0$}; \node at (.25,1.05) {$\color{color1}\mathbf{1}$};
 \node at (.5,1.05) {$0$}; \node at (.75,1.05) {$0$};
 \node at (1,1.05) {$0$}; \node at (1.25,1.05) {$0$};
 \node at (1.5,1.05) {$\color{color2}\mathbf{1}$}; \node at (1.75,1.05) {$0$};
 \node at (2,1.05) {$\color{color2}\mathbf{1}$}; \node at (2.25,1.05) {$0$};
 \node at (2.5,1.05) {$0$};
 \node at (0,.7) {$0$}; \node at (.25,.7) {$0$};
 \node at (.5,.7) {$\color{color1}\mathbf{1}$}; \node at (.75,.7) {$0$};
 \node at (1,.7) {$\color{color1}\mathbf{1}$}; \node at (1.25,.7) {$0$};
 \node at (1.5,.7) {$0$}; \node at (1.75,.7) {$0$};
 \node at (2,.7) {$0$}; \node at (2.25,.7) {$\color{color2}\mathbf{1}$};
 \node at (2.5,.7) {$0$};
 \node at (0,.35) {$\color{color2}\mathbf{1}$}; \node at (.25,.35) {$0$};
 \node at (.5,.35) {$0$}; \node at (.75,.35) {$0$};
 \node at (1,.35) {$0$}; \node at (1.25,.35) {$\color{color1}\mathbf{1}$};
 \node at (1.5,.35) {$0$}; \node at (1.75,.35) {$\color{color1}\mathbf{1}$};
 \node at (2,.35) {$0$}; \node at (2.25,.35) {$0$};
 \node at (2.5,.35) {$0$};
  \node at (0,0) {$0$}; \node at (.25,0) {$\color{color2}\mathbf{1}$};
 \node at (.5,0) {$0$}; \node at (.75,0) {$\color{color2}\mathbf{1}$};
 \node at (1,0) {$0$}; \node at (1.25,0) {$0$};
 \node at (1.5,0) {$0$}; \node at (1.75,0) {$0$};
 \node at (2,0) {$\color{color1}\mathbf{1}$}; \node at (2.25,0) {$0$};
 \node at (2.5,0) {$\color{color1}\mathbf{1}$};
 \end{scope}
\begin{scope}[shift={(0,2.45)}]
\colorlet{color2}{Red}
\colorlet{color1}{Blue}
\foreach \x in {0,...,10} {\node at (.25*\x,2.45) {$\vdots$};}
  \node at (0,2.1) {$0$}; \node at (.25,2.1) {$0$};
 \node at (.5,2.1) {$0$}; \node at (.75,2.1) {$0$};
 \node at (1,2.1) {$\color{color1}\mathbf{1}$}; \node at (1.25,2.1) {$0$};
 \node at (1.5,2.1) {$\color{color1}\mathbf{1}$}; \node at (1.75,2.1) {$0$};
 \node at (2,2.1) {$0$}; \node at (2.25,2.1) {$0$};
 \node at (2.5,2.1) {$0$};
   \node at (0,1.75) {$\color{color2}\mathbf{1}$}; \node at (.25,1.75) {$0$};
 \node at (.5,1.75) {$\color{color2}\mathbf{1}$}; \node at (.75,1.75) {$0$};
 \node at (1,1.75) {$0$}; \node at (1.25,1.75) {$0$};
 \node at (1.5,1.75) {$0$}; \node at (1.75,1.75) {$\color{color1}\mathbf{1}$};
 \node at (2,1.75) {$0$}; \node at (2.25,1.75) {$\color{color1}\mathbf{1}$};
 \node at (2.5,1.75) {$0$};
   \node at (0,1.4) {$0$}; \node at (.25,1.4) {$0$};
 \node at (.5,1.4) {$0$}; \node at (.75,1.4) {$\color{color2}\mathbf{1}$};
 \node at (1,1.4) {$0$}; \node at (1.25,1.4) {$\color{color2}\mathbf{1}$};
 \node at (1.5,1.4) {$0$}; \node at (1.75,1.4) {$0$};
 \node at (2,1.4) {$0$}; \node at (2.25,1.4) {$0$};
 \node at (2.5,1.4) {$\color{color1}\mathbf{1}$};
   \node at (0,1.05) {$0$}; \node at (.25,1.05) {$\color{color1}\mathbf{1}$};
 \node at (.5,1.05) {$0$}; \node at (.75,1.05) {$0$};
 \node at (1,1.05) {$0$}; \node at (1.25,1.05) {$0$};
 \node at (1.5,1.05) {$\color{color2}\mathbf{1}$}; \node at (1.75,1.05) {$0$};
 \node at (2,1.05) {$\color{color2}\mathbf{1}$}; \node at (2.25,1.05) {$0$};
 \node at (2.5,1.05) {$0$};
 \node at (0,.7) {$0$}; \node at (.25,.7) {$0$};
 \node at (.5,.7) {$\color{color1}\mathbf{1}$}; \node at (.75,.7) {$0$};
 \node at (1,.7) {$\color{color1}\mathbf{1}$}; \node at (1.25,.7) {$0$};
 \node at (1.5,.7) {$0$}; \node at (1.75,.7) {$0$};
 \node at (2,.7) {$0$}; \node at (2.25,.7) {$\color{color2}\mathbf{1}$};
 \node at (2.5,.7) {$0$};
 \node at (0,.35) {$\color{color2}\mathbf{1}$}; \node at (.25,.35) {$0$};
 \node at (.5,.35) {$0$}; \node at (.75,.35) {$0$};
 \node at (1,.35) {$0$}; \node at (1.25,.35) {$\color{color1}\mathbf{1}$};
 \node at (1.5,.35) {$0$}; \node at (1.75,.35) {$\color{color1}\mathbf{1}$};
 \node at (2,.35) {$0$}; \node at (2.25,.35) {$0$};
 \node at (2.5,.35) {$0$};
  \node at (0,0) {$0$}; \node at (.25,0) {$\color{color2}\mathbf{1}$};
 \node at (.5,0) {$0$}; \node at (.75,0) {$\color{color2}\mathbf{1}$};
 \node at (1,0) {$0$}; \node at (1.25,0) {$0$};
 \node at (1.5,0) {$0$}; \node at (1.75,0) {$0$};
 \node at (2,0) {$\color{color1}\mathbf{1}$}; \node at (2.25,0) {$0$};
 \node at (2.5,0) {$\color{color1}\mathbf{1}$};
 \end{scope}
 \begin{scope}[shift={(2.75,2.8)}]
 \colorlet{color2}{Red}
\colorlet{color1}{Blue}
\draw [draw=lightgrey,fill=lightgrey] (-.12,.09) rectangle (2.63,2.54);
\foreach \x in {0,...,6} {\node at (2.95,.35*\x) {$\cdots$};}
\foreach \x in {0,...,10} {\node at (.25*\x,2.45) {$\vdots$};}
\node at (2.95,2.45) {$\iddots$};
  \node at (0,2.1) {$0$}; \node at (.25,2.1) {$0$};
 \node at (.5,2.1) {$0$}; \node at (.75,2.1) {$0$};
 \node at (1,2.1) {$\color{color1}\mathbf{1}$}; \node at (1.25,2.1) {$0$};
 \node at (1.5,2.1) {$\color{color1}\mathbf{1}$}; \node at (1.75,2.1) {$0$};
 \node at (2,2.1) {$0$}; \node at (2.25,2.1) {$0$};
 \node at (2.5,2.1) {$0$};
   \node at (0,1.75) {$\color{color2}\mathbf{1}$}; \node at (.25,1.75) {$0$};
 \node at (.5,1.75) {$\color{color2}\mathbf{1}$}; \node at (.75,1.75) {$0$};
 \node at (1,1.75) {$0$}; \node at (1.25,1.75) {$0$};
 \node at (1.5,1.75) {$0$}; \node at (1.75,1.75) {$\color{color1}\mathbf{1}$};
 \node at (2,1.75) {$0$}; \node at (2.25,1.75) {$\color{color1}\mathbf{1}$};
 \node at (2.5,1.75) {$0$};
   \node at (0,1.4) {$0$}; \node at (.25,1.4) {$0$};
 \node at (.5,1.4) {$0$}; \node at (.75,1.4) {$\color{color2}\mathbf{1}$};
 \node at (1,1.4) {$0$}; \node at (1.25,1.4) {$\color{color2}\mathbf{1}$};
 \node at (1.5,1.4) {$0$}; \node at (1.75,1.4) {$0$};
 \node at (2,1.4) {$0$}; \node at (2.25,1.4) {$0$};
 \node at (2.5,1.4) {$\color{color1}\mathbf{1}$};
   \node at (0,1.05) {$0$}; \node at (.25,1.05) {$\color{color1}\mathbf{1}$};
 \node at (.5,1.05) {$0$}; \node at (.75,1.05) {$0$};
 \node at (1,1.05) {$0$}; \node at (1.25,1.05) {$0$};
 \node at (1.5,1.05) {$\color{color2}\mathbf{1}$}; \node at (1.75,1.05) {$0$};
 \node at (2,1.05) {$\color{color2}\mathbf{1}$}; \node at (2.25,1.05) {$0$};
 \node at (2.5,1.05) {$0$};
 \node at (0,.7) {$0$}; \node at (.25,.7) {$0$};
 \node at (.5,.7) {$\color{color1}\mathbf{1}$}; \node at (.75,.7) {$0$};
 \node at (1,.7) {$\color{color1}\mathbf{1}$}; \node at (1.25,.7) {$0$};
 \node at (1.5,.7) {$0$}; \node at (1.75,.7) {$0$};
 \node at (2,.7) {$0$}; \node at (2.25,.7) {$\color{color2}\mathbf{1}$};
 \node at (2.5,.7) {$0$};
 \node at (0,.35) {$\color{color2}\mathbf{1}$}; \node at (.25,.35) {$0$};
 \node at (.5,.35) {$0$}; \node at (.75,.35) {$0$};
 \node at (1,.35) {$0$}; \node at (1.25,.35) {$\color{color1}\mathbf{1}$};
 \node at (1.5,.35) {$0$}; \node at (1.75,.35) {$\color{color1}\mathbf{1}$};
 \node at (2,.35) {$0$}; \node at (2.25,.35) {$0$};
 \node at (2.5,.35) {$0$};
  \node at (0,0) {$0$}; \node at (.25,0) {$\color{color2}\mathbf{1}$};
 \node at (.5,0) {$0$}; \node at (.75,0) {$\color{color2}\mathbf{1}$};
 \node at (1,0) {$0$}; \node at (1.25,0) {$0$};
 \node at (1.5,0) {$0$}; \node at (1.75,0) {$0$};
 \node at (2,0) {$\color{color1}\mathbf{1}$}; \node at (2.25,0) {$0$};
 \node at (2.5,0) {$\color{color1}\mathbf{1}$};
 \end{scope}
 \begin{scope}[shift={(-2.75,-2.8)}]
 \colorlet{color2}{Red}
\colorlet{color1}{Blue}
\draw [draw=lightgrey,fill=lightgrey] (-.12,.09) rectangle (2.63,2.54);
\foreach \x in {0,...,6} {\node at (-.35,.35*\x) {$\cdots$};}
\foreach \x in {0,...,10} {\node at (.25*\x,-.45) {$\vdots$};}
\node at (-.35,-.45) {$\iddots$};
  \node at (0,2.1) {$0$}; \node at (.25,2.1) {$0$};
 \node at (.5,2.1) {$0$}; \node at (.75,2.1) {$0$};
 \node at (1,2.1) {$\color{color1}\mathbf{1}$}; \node at (1.25,2.1) {$0$};
 \node at (1.5,2.1) {$\color{color1}\mathbf{1}$}; \node at (1.75,2.1) {$0$};
 \node at (2,2.1) {$0$}; \node at (2.25,2.1) {$0$};
 \node at (2.5,2.1) {$0$};
   \node at (0,1.75) {$\color{color2}\mathbf{1}$}; \node at (.25,1.75) {$0$};
 \node at (.5,1.75) {$\color{color2}\mathbf{1}$}; \node at (.75,1.75) {$0$};
 \node at (1,1.75) {$0$}; \node at (1.25,1.75) {$0$};
 \node at (1.5,1.75) {$0$}; \node at (1.75,1.75) {$\color{color1}\mathbf{1}$};
 \node at (2,1.75) {$0$}; \node at (2.25,1.75) {$\color{color1}\mathbf{1}$};
 \node at (2.5,1.75) {$0$};
   \node at (0,1.4) {$0$}; \node at (.25,1.4) {$0$};
 \node at (.5,1.4) {$0$}; \node at (.75,1.4) {$\color{color2}\mathbf{1}$};
 \node at (1,1.4) {$0$}; \node at (1.25,1.4) {$\color{color2}\mathbf{1}$};
 \node at (1.5,1.4) {$0$}; \node at (1.75,1.4) {$0$};
 \node at (2,1.4) {$0$}; \node at (2.25,1.4) {$0$};
 \node at (2.5,1.4) {$\color{color1}\mathbf{1}$};
   \node at (0,1.05) {$0$}; \node at (.25,1.05) {$\color{color1}\mathbf{1}$};
 \node at (.5,1.05) {$0$}; \node at (.75,1.05) {$0$};
 \node at (1,1.05) {$0$}; \node at (1.25,1.05) {$0$};
 \node at (1.5,1.05) {$\color{color2}\mathbf{1}$}; \node at (1.75,1.05) {$0$};
 \node at (2,1.05) {$\color{color2}\mathbf{1}$}; \node at (2.25,1.05) {$0$};
 \node at (2.5,1.05) {$0$};
 \node at (0,.7) {$0$}; \node at (.25,.7) {$0$};
 \node at (.5,.7) {$\color{color1}\mathbf{1}$}; \node at (.75,.7) {$0$};
 \node at (1,.7) {$\color{color1}\mathbf{1}$}; \node at (1.25,.7) {$0$};
 \node at (1.5,.7) {$0$}; \node at (1.75,.7) {$0$};
 \node at (2,.7) {$0$}; \node at (2.25,.7) {$\color{color2}\mathbf{1}$};
 \node at (2.5,.7) {$0$};
 \node at (0,.35) {$\color{color2}\mathbf{1}$}; \node at (.25,.35) {$0$};
 \node at (.5,.35) {$0$}; \node at (.75,.35) {$0$};
 \node at (1,.35) {$0$}; \node at (1.25,.35) {$\color{color1}\mathbf{1}$};
 \node at (1.5,.35) {$0$}; \node at (1.75,.35) {$\color{color1}\mathbf{1}$};
 \node at (2,.35) {$0$}; \node at (2.25,.35) {$0$};
 \node at (2.5,.35) {$0$};
  \node at (0,0) {$0$}; \node at (.25,0) {$\color{color2}\mathbf{1}$};
 \node at (.5,0) {$0$}; \node at (.75,0) {$\color{color2}\mathbf{1}$};
 \node at (1,0) {$0$}; \node at (1.25,0) {$0$};
 \node at (1.5,0) {$0$}; \node at (1.75,0) {$0$};
 \node at (2,0) {$\color{color1}\mathbf{1}$}; \node at (2.25,0) {$0$};
 \node at (2.5,0) {$\color{color1}\mathbf{1}$};
 \end{scope}
\begin{scope}[shift={(0,-2.45)}]
\colorlet{color2}{Red}
\colorlet{color1}{Blue}
\foreach \x in {0,...,10} {\node at (.25*\x,-.45) {$\vdots$};}
  \node at (0,2.1) {$0$}; \node at (.25,2.1) {$0$};
 \node at (.5,2.1) {$0$}; \node at (.75,2.1) {$0$};
 \node at (1,2.1) {$\color{color1}\mathbf{1}$}; \node at (1.25,2.1) {$0$};
 \node at (1.5,2.1) {$\color{color1}\mathbf{1}$}; \node at (1.75,2.1) {$0$};
 \node at (2,2.1) {$0$}; \node at (2.25,2.1) {$0$};
 \node at (2.5,2.1) {$0$};
   \node at (0,1.75) {$\color{color2}\mathbf{1}$}; \node at (.25,1.75) {$0$};
 \node at (.5,1.75) {$\color{color2}\mathbf{1}$}; \node at (.75,1.75) {$0$};
 \node at (1,1.75) {$0$}; \node at (1.25,1.75) {$0$};
 \node at (1.5,1.75) {$0$}; \node at (1.75,1.75) {$\color{color1}\mathbf{1}$};
 \node at (2,1.75) {$0$}; \node at (2.25,1.75) {$\color{color1}\mathbf{1}$};
 \node at (2.5,1.75) {$0$};
   \node at (0,1.4) {$0$}; \node at (.25,1.4) {$0$};
 \node at (.5,1.4) {$0$}; \node at (.75,1.4) {$\color{color2}\mathbf{1}$};
 \node at (1,1.4) {$0$}; \node at (1.25,1.4) {$\color{color2}\mathbf{1}$};
 \node at (1.5,1.4) {$0$}; \node at (1.75,1.4) {$0$};
 \node at (2,1.4) {$0$}; \node at (2.25,1.4) {$0$};
 \node at (2.5,1.4) {$\color{color1}\mathbf{1}$};
   \node at (0,1.05) {$0$}; \node at (.25,1.05) {$\color{color1}\mathbf{1}$};
 \node at (.5,1.05) {$0$}; \node at (.75,1.05) {$0$};
 \node at (1,1.05) {$0$}; \node at (1.25,1.05) {$0$};
 \node at (1.5,1.05) {$\color{color2}\mathbf{1}$}; \node at (1.75,1.05) {$0$};
 \node at (2,1.05) {$\color{color2}\mathbf{1}$}; \node at (2.25,1.05) {$0$};
 \node at (2.5,1.05) {$0$};
 \node at (0,.7) {$0$}; \node at (.25,.7) {$0$};
 \node at (.5,.7) {$\color{color1}\mathbf{1}$}; \node at (.75,.7) {$0$};
 \node at (1,.7) {$\color{color1}\mathbf{1}$}; \node at (1.25,.7) {$0$};
 \node at (1.5,.7) {$0$}; \node at (1.75,.7) {$0$};
 \node at (2,.7) {$0$}; \node at (2.25,.7) {$\color{color2}\mathbf{1}$};
 \node at (2.5,.7) {$0$};
 \node at (0,.35) {$\color{color2}\mathbf{1}$}; \node at (.25,.35) {$0$};
 \node at (.5,.35) {$0$}; \node at (.75,.35) {$0$};
 \node at (1,.35) {$0$}; \node at (1.25,.35) {$\color{color1}\mathbf{1}$};
 \node at (1.5,.35) {$0$}; \node at (1.75,.35) {$\color{color1}\mathbf{1}$};
 \node at (2,.35) {$0$}; \node at (2.25,.35) {$0$};
 \node at (2.5,.35) {$0$};
  \node at (0,0) {$0$}; \node at (.25,0) {$\color{color2}\mathbf{1}$};
 \node at (.5,0) {$0$}; \node at (.75,0) {$\color{color2}\mathbf{1}$};
 \node at (1,0) {$0$}; \node at (1.25,0) {$0$};
 \node at (1.5,0) {$0$}; \node at (1.75,0) {$0$};
 \node at (2,0) {$\color{color1}\mathbf{1}$}; \node at (2.25,0) {$0$};
 \node at (2.5,0) {$\color{color1}\mathbf{1}$};
 \end{scope}
 \begin{scope}[shift={(2.75,-2.1)}]
 \colorlet{color2}{Red}
\colorlet{color1}{Blue}
\draw [draw=lightgrey,fill=lightgrey] (-.12,.09) rectangle (2.63,2.54);
\foreach \x in {0,...,6} {\node at (2.95,.35*\x) {$\cdots$};}
\foreach \x in {0,...,10} {\node at (.25*\x,-.45) {$\vdots$};}
\node at (2.95,-.45) {$\ddots$};
  \node at (0,2.1) {$0$}; \node at (.25,2.1) {$0$};
 \node at (.5,2.1) {$0$}; \node at (.75,2.1) {$0$};
 \node at (1,2.1) {$\color{color1}\mathbf{1}$}; \node at (1.25,2.1) {$0$};
 \node at (1.5,2.1) {$\color{color1}\mathbf{1}$}; \node at (1.75,2.1) {$0$};
 \node at (2,2.1) {$0$}; \node at (2.25,2.1) {$0$};
 \node at (2.5,2.1) {$0$};
   \node at (0,1.75) {$\color{color2}\mathbf{1}$}; \node at (.25,1.75) {$0$};
 \node at (.5,1.75) {$\color{color2}\mathbf{1}$}; \node at (.75,1.75) {$0$};
 \node at (1,1.75) {$0$}; \node at (1.25,1.75) {$0$};
 \node at (1.5,1.75) {$0$}; \node at (1.75,1.75) {$\color{color1}\mathbf{1}$};
 \node at (2,1.75) {$0$}; \node at (2.25,1.75) {$\color{color1}\mathbf{1}$};
 \node at (2.5,1.75) {$0$};
   \node at (0,1.4) {$0$}; \node at (.25,1.4) {$0$};
 \node at (.5,1.4) {$0$}; \node at (.75,1.4) {$\color{color2}\mathbf{1}$};
 \node at (1,1.4) {$0$}; \node at (1.25,1.4) {$\color{color2}\mathbf{1}$};
 \node at (1.5,1.4) {$0$}; \node at (1.75,1.4) {$0$};
 \node at (2,1.4) {$0$}; \node at (2.25,1.4) {$0$};
 \node at (2.5,1.4) {$\color{color1}\mathbf{1}$};
   \node at (0,1.05) {$0$}; \node at (.25,1.05) {$\color{color1}\mathbf{1}$};
 \node at (.5,1.05) {$0$}; \node at (.75,1.05) {$0$};
 \node at (1,1.05) {$0$}; \node at (1.25,1.05) {$0$};
 \node at (1.5,1.05) {$\color{color2}\mathbf{1}$}; \node at (1.75,1.05) {$0$};
 \node at (2,1.05) {$\color{color2}\mathbf{1}$}; \node at (2.25,1.05) {$0$};
 \node at (2.5,1.05) {$0$};
 \node at (0,.7) {$0$}; \node at (.25,.7) {$0$};
 \node at (.5,.7) {$\color{color1}\mathbf{1}$}; \node at (.75,.7) {$0$};
 \node at (1,.7) {$\color{color1}\mathbf{1}$}; \node at (1.25,.7) {$0$};
 \node at (1.5,.7) {$0$}; \node at (1.75,.7) {$0$};
 \node at (2,.7) {$0$}; \node at (2.25,.7) {$\color{color2}\mathbf{1}$};
 \node at (2.5,.7) {$0$};
 \node at (0,.35) {$\color{color2}\mathbf{1}$}; \node at (.25,.35) {$0$};
 \node at (.5,.35) {$0$}; \node at (.75,.35) {$0$};
 \node at (1,.35) {$0$}; \node at (1.25,.35) {$\color{color1}\mathbf{1}$};
 \node at (1.5,.35) {$0$}; \node at (1.75,.35) {$\color{color1}\mathbf{1}$};
 \node at (2,.35) {$0$}; \node at (2.25,.35) {$0$};
 \node at (2.5,.35) {$0$};
  \node at (0,0) {$0$}; \node at (.25,0) {$\color{color2}\mathbf{1}$};
 \node at (.5,0) {$0$}; \node at (.75,0) {$\color{color2}\mathbf{1}$};
 \node at (1,0) {$0$}; \node at (1.25,0) {$0$};
 \node at (1.5,0) {$0$}; \node at (1.75,0) {$0$};
 \node at (2,0) {$\color{color1}\mathbf{1}$}; \node at (2.25,0) {$0$};
 \node at (2.5,0) {$\color{color1}\mathbf{1}$};
 \end{scope}
 \end{tikzpicture}
 \caption{The universal scroll $\widehat\cals$ of our running example $\cals=\Scroll(x)$ from Figure~\ref{fig:example}. Each $7\times 11$ block highlighted by the checkerboard shading corresponds to an identical copy of a $\tau$-orbit. The offset of the blocks arises because the covering map in Eq.~\eqref{eqn:cd-universal} is not just the standard  ``modulo $n$'' reduction, since in the scroll, the end of a row wraps around to the beginning of the next one.}\label{fig:universal-scroll}
\end{figure}
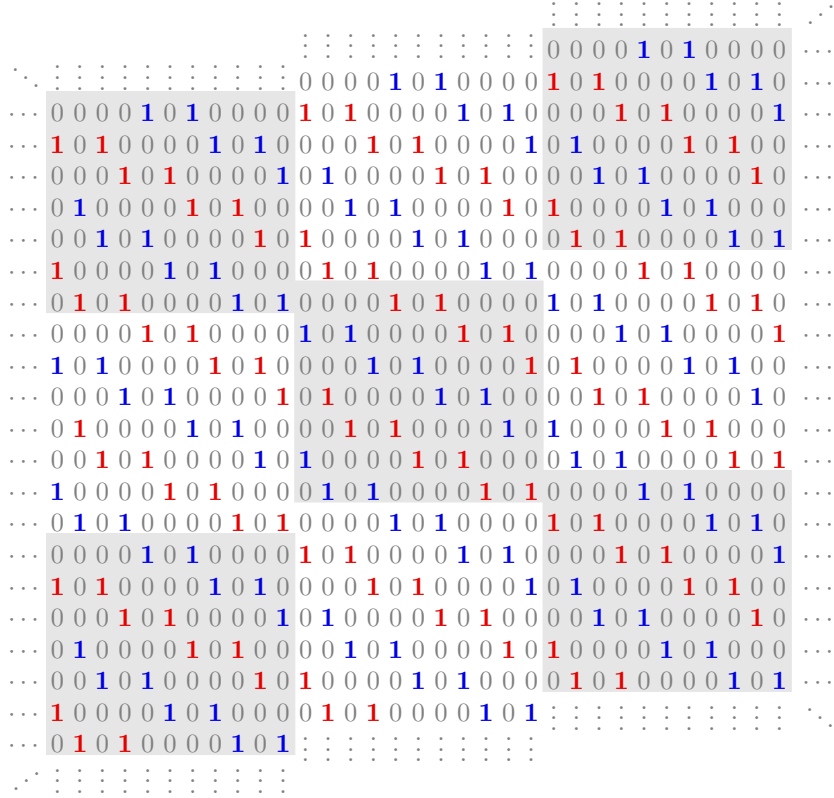

In both a scroll and a ticker tape, positions that have a value of $1$ are said to be \dfn{live}. Formally, the sets of live entries, in both formats, are
\[
\Live(\cals)=\big\{(i,j)\in\zz\times\zz_n\mid X_{i,j}=1\big\},\qquad
\Live(\calx)=\big\{k\in\zz\mid X_k=1\big\}.
\]
We can also define the set of live entries in the universal scroll as
\[
\Live(\widehat\cals)=\big\{(i,j)\in\zz\times\zz\mid X_{i,j}=1\big\}=\q^{-1}(\Live(\cals)).
\]

\subsection{The successor and co-successor functions}\label{subsec:successor}

In this subsection, we will explore the patterns of the live entries within the scrolls. Given an arbitrary live entry $(i,j)$,  we can draw some easy conclusions about its surrounding entries. Recall that since scrolls are naturally embedded on a cylinder, we can always speak of the entry immediately to the left or to the right, even if we are in the first or the last column, by setting $X_{i,k+n}=X_{i+1,k}$. 

It is clear that for any live entry, the four entries that are immediately adjacent to it must be 0. Similarly, the diagonal entries in the upper-right and lower-left directions also must be 0. 

\begin{lemma}\label{lem:nbhd}
If $(i,j)\in\Live(\cals)$, then 
\[
X_{i-1,j}=X_{i-1,j+1}=X_{i,j-1}=X_{i,j+1}
=X_{i+1,j-1}=X_{i+1,j}=0.
\]
\end{lemma}

It is also elementary to see that exactly one of $X_{i,j+2}$ and $X_{i+1,j+1}$ must be $1$; we will call the location of whichever entry is live the \emph{successor} of $(i,j)$. Similarly, exactly one of $X_{i+2,j-2}$ and $X_{i+2,j-1}$ must be $1$; we will call the location of the live entry the \emph{co-successor} of $(i,j)$. The formal statement of this is given in Lemma~\ref{lem:successor-co-successor}, and a visual interpretation is shown in Figure~\ref{fig:successor-co-successor}. Notice how these two images also show the nearby surrounding entries that must be $0$, as guaranteed by Lemma~\ref{lem:nbhd}.

\begin{lemma}\label{lem:successor-co-successor} 
If $(i,j)\in\Live(\cals)$ then
\[
X_{i,j+2}+X_{i+1,j+1}=1,\quad\text{and}\quad X_{i+2,j-2}+X_{i+2,j-1}=1.
\]
\end{lemma}

\begin{figure}[!ht]
\begin{tikzpicture}
\setlength{\tabcolsep}{2.5pt}
\renewcommand{\arraystretch}{.4}
\begin{scope}[shift={(0,0)},scale=.4]
\tikzstyle{every node}=[font=\normalsize,anchor=south]
    \node at (1.75,4.75) {$\cdots$};
    \node at (3,4.75) {$0$}; 
    \node at (4,4.75) {$0$};
    \node at (5.5,4.75) {$\cdots$};
    \node at (2,3.5) {$0$};
    \node at (3,3.5) {$\color{Red}\mathbf{1}$}; 
    \node at (4,3.5) {$0$};
    \node at (5,3.5) {$\color{Red}\mathbf{a}$};
    \node at (6.5,3.5) {$\cdots$};
    \node at (.75,2.25) {$\cdots$};
    \node at (2,2.25) {$0$};
    \node at (3,2.25) {$0$}; 
    \node at (4,2.25) {$\color{Red}\overline{\mathbf{a}}$};
    \node at (5,2.25) {$0$};
    \node at (6.5,2.25) {$\cdots$};
    \node at (3.5,.5) {``\emph{successor of $(i,j)$}''};
\end{scope}
\begin{scope}[shift={(6,.45)},scale=.4]
\tikzstyle{every node}=[font=\normalsize,anchor=south]
    \node at (1.75,4.75) {$\cdots$};
    \node at (3,4.75) {$0$}; 
    \node at (4,4.75) {$0$};
    \node at (5.5,4.75) {$\cdots$};
    \node at (.75,4.75) {$\cdots$};
    \node at (4,3.5) {$0$};
    \node at (3,3.5) {$\color{orange}\mathbf{1}$}; 
    \node at (5.5,3.5) {$\cdots$};
    \node at (.75,3.5) {$\cdots$};
    \node at (3,2.25) {$0$}; 
    \node at (2,3.5) {$0$};
    \node at (2,2.25) {$0$};
    \node at (4.5,2.25) {$\cdots$};
    \node at (.75,2.25) {$\cdots$};
    \node at (-.25,1) {$\cdots$};
    \node at (2,1) {$\color{orange}\overline{\mathbf{b}}$};
    \node at (1,1) {$\color{orange}\mathbf{b}$};
    \node at (3.5,1) {$\cdots$};
    \node at (3.5,-.75) {``\emph{co-successor of $(i,j)$}''};
\end{scope}
\end{tikzpicture}
\caption{A picture illustrating Lemmas~\ref{lem:nbhd} and~\ref{lem:successor-co-successor}, where $\overline{a}=1+a$ and $\overline{b}=1+b$ denote complementary entries in $\ff_2=\{0,1\}$.}\label{fig:successor-co-successor}
\end{figure}

It will be convenient to think of the successor and co-successor as functions on the live entries. 

\begin{definition}\label{def:successor-co-successor}
Given a scroll $\cals$, the \dfn{successor} is the function $s\colon\Live(\cals)\longto\Live(\cals)$ that sends $(i,j)$ to the unique element of
$\big\{(i,j+2),(i+1,j+1)\big\}\cap \Live(\cals)$.
The \dfn{co-successor} is the function $c\colon\Live(\cals)\longto\Live(\cals)$ that sends  $(i,j)$ to the unique element of $\big\{(i+2,j-2),(i+2,j-1)\big\}\cap\Live(\cals)$.
\end{definition}

Naturally, Definition~\ref{def:successor-co-successor} is easily translated into ticker tape notation, where the subscripts are indices rather than ordered pairs. Specifically, if $X_k=1$, then the successor and co-successor functions $s,c\colon\Live(\calx)\to\Live(\calx)$ are defined by sending $k$ to the unique element of
\[
\big\{k+2,k+n+1\big\}\cap\Live(\calx)
\quad\text{and}\quad
\big\{k+2n-2,k+2n-1\big\}\cap\Live(\calx),
\]
respectively. 

\begin{lemma}\label{lem:bijections}
The successor and co-successor functions are bijections on $\Live(\cals)$.
\end{lemma}

\begin{proof}
If $(i,j)\in\Live(\cals)$ were to have two $s$-preimages, then they would have to be $(i-1,j-1)$ and $(i,j-2)$. However, $X_{i-1,j-1}=1$ would force $X_{i,j-2}=0$ by Lemma~\ref{lem:nbhd}.

Similarly, if $(i,j)\in\Live(\cals)$ were to have two $c$-preimages, then they would have to be $(i-2,j+1)$ and $(i-2,j+2)$, and this is clearly not allowed because they are adjacent. Figure~\ref{fig:bijective} provides a visual for why these two scenarios are impossible.
\end{proof}

\begin{figure}[!ht]
\begin{tikzpicture}
\setlength{\tabcolsep}{2.5pt}
\renewcommand{\arraystretch}{.4}
\begin{scope}[shift={(0,0)},scale=.4]
\tikzstyle{every node}=[font=\normalsize,anchor=south]
    \node at (.75,4.75) {$\cdots$};
    \node at (2,4.75) {$\0$}; 
    \node at (3,4.75) {$\mathbf{1}$}; 
    \node at (4,4.75) {$\0$};
    \node at (5,4.75) {$\0$};
    \node at (6.5,4.75) {$\cdots$};
    \node at (.75,3.5) {$\cdots$};
    \node at (2,3.5) {$\mathbf{1}$};
    \node at (3,3.5) {$\0$}; 
    \node at (4,3.5) {$\color{Red}\mathbf{1}$};
    \node at (5,3.5) {$\0$};
    \node at (6.5,3.5) {$\cdots$};
    \node at (1.75,2.25) {$\cdots$};
    \node at (3,2.25) {$\0$}; 
    \node at (4,2.25) {$\0$};
    \node at (5.5,2.25) {$\cdots$};
\end{scope}
\begin{scope}[shift={(5,0)},scale=.4]
\tikzstyle{every node}=[font=\normalsize,anchor=south]
    \node at (3.75,6) {$\cdots$};
    \node at (5,6) {$\mathbf{1}$};
    \node at (6,6) {$\mathbf{1}$};
    \node at (7.5,6) {$\cdots$};
    \node at (2.75,4.75) {$\cdots$};
    \node at (4,4.75) {$\0$};
    \node at (5,4.75) {$\0$};
    \node at (6.5,4.75) {$\cdots$};
    \node at (1.75,3.5) {$\cdots$};
    \node at (3,3.5) {$\0$}; 
    \node at (4,3.5) {$\color{Red}\mathbf{1}$};
    \node at (5,3.5) {$\0$};
    \node at (6.5,3.5) {$\cdots$};
    \node at (1.75,2.25) {$\cdots$};
    \node at (3,2.25) {$\0$}; 
    \node at (4,2.25) {$\0$};
    \node at (5.5,2.25) {$\cdots$};
\end{scope}
\end{tikzpicture}
\caption{If the successor (respectively, co-successor) function were not bijective, then the impossible configuration on the left (respectively, right) would occur in the scroll.}\label{fig:bijective}
\end{figure}

Since the successor and co-successor functions are bijections on $\Live(\cals)$, each of them defines an equivalence relation. The color scheme back in Figure~\ref{fig:example} highlights the equivalence classes, which is why we often prefer to use scrolls over ticker tapes. We will explore these notions more in the next subsection. 

A fundamental property of the successor and co-successor functions is the simple observation that they commute. 

\begin{proposition}\label{prop:successor-of-co-successor}
The successor of the co-successor is the co-successor of the successor. That is, for any live entry $(i,j)$, we have $s(c(i,j))=c(s(i,j))$. 
\end{proposition}

\begin{proof}
 Starting with a live entry ${\color{Red}(i,j)}$, the two possibilities for its co-successor {\color{Blue}$c(i,j)$} are shown in Figure~\ref{fig:successor-of-co-successor}.  In each case, the successor of the co-successor is the position of either $b$ or $\overline{b}$, whichever is live. The successor of $(i,j)$ is the position of either $a$ or $\overline{a}$, whichever is live. It is easy to check that in both cases above, $a=1$ forces $b=1$, and $\overline{a}=1$ forces $\overline{b}=1$. In both cases, the position of whichever $b$ or $\overline{b}$ is live, by definition, is the co-successor of the successor. 
\end{proof}
 
\begin{figure}[!ht]
\begin{tikzpicture}
\tikzstyle{every node}=[font=\normalsize,anchor=south]
\begin{scope}[shift={(0,0)},scale=.4]
    \node at (1.75,4.75) {$\cdots$};
    \node at (3,4.75) {$\0$}; 
    \node at (4,4.75) {$\0$};
    \node at (5.5,4.75) {$\cdots$};
    \node at (.75,3.5) {$\cdots$};
    \node at (2,3.5) {$\0$};
    \node at (3,3.5) {$\color{Red}\mathbf{1}$};
    \node at (4,3.5) {$\0$};
    \node at (5,3.5) {$\color{Red}\mathbf{a}$};
    \node at (6.5,3.5) {$\cdots$};
    \node at (-.25,2.25) {$\cdots$};
    \node at (2,2.25) {$\0$};
    \node at (3,2.25) {$\0$}; 
    \node at (4,2.25) {$\color{Red}\mathbf{\overline{a}}$};
    \node at (5,2.25) {$\0$};
    \node at (6.5,2.25) {$\cdots$};
    \node at (-1.25,1) {$\cdots$};
    \node at (2,1) {$\0$};
    \node at (0,1) {$\0$};
    \node at (0,-.25) {$\0$};
    \node at (1,-.25) {$\0$};
    \node at (1,2.25){$\0$}; 
    \node at (1,1){$\color{Blue}\mathbf{1}$}; 
    \node at (3,1) {$\color{Blue}\mathbf{b}$};
    \node at (4.5,1) {$\cdots$};
    \node at (-1.25,-.25) {$\cdots$};
    \node at (2,-.25) {$\color{Blue}\mathbf{\overline{b}}$};
    \node at (3,-.25) {$\0$};
    \node at (4.5,-.25) {$\cdots$};
\end{scope}
\begin{scope}[shift={(6,0)},scale=.4]
    \node at (1.75,4.75) {$\cdots$};
    \node at (3,4.75) {$\0$}; 
    \node at (4,4.75) {$\0$};
    \node at (5.5,4.75) {$\cdots$};
    \node at (0.75,3.5) {$\cdots$};
    \node at (4,3.5) {$\0$};
    \node at (3,3.5) {$\color{Red}\mathbf{1}$}; 
    \node at (2,3.5) {$\0$};
    \node at (5,3.5) {$\color{Red}\mathbf{a}$};
    \node at (6.5,3.5) {$\cdots$};
    \node at (0.75,2.25) {$\cdots$};
    \node at (3,2.25) {$\0$}; 
    \node at (2,2.25) {$\0$};
    \node at (4,2.25) {$\color{Red}\mathbf{\overline{a}}$};
    \node at (5,2.25) {$\0$};
    \node at (6.5,2.25) {$\cdots$};
    \node at (-.25,1) {$\cdots$};
    \node at (1,1) {$\0$};
    \node at (2,1) {$\color{Blue}\mathbf{1}$};
    \node at (3,1) {$\0$};
    \node at (4,1) {$\color{Blue}\mathbf{b}$};
    \node at (5.5,1) {$\cdots$};
    \node at (-.25,-.25) {$\cdots$};
    \node at (2,-.25) {$\0$};
    \node at (1,-.25) {$\0$};
    \node at (3,-.25) {$\color{Blue}\mathbf{\overline{b}}$};
    \node at (4,-.25) {$\0$};
    \node at (5.5,-.25) {$\cdots$};
\end{scope}
\end{tikzpicture}
\caption{A picture of the argument in Proposition~\ref{prop:successor-of-co-successor} of why the successor and co-successor functions commute.}\label{fig:successor-of-co-successor} 
 \end{figure}

At times, it will be more convenient to work in the universal scroll $\widehat{\cals}$ than in $\cals$ itself. The successor and co-successor functions lift to the commuting \dfn{universal successor} and \dfn{universal co-successor} functions $\widehat{s},\widehat{c}\colon\Live(\widehat\cals)\to \Live(\widehat\cals)$, which are defined by replacing each $\cals$ with $\widehat\cals$ in Definition~\ref{def:successor-co-successor}.

\subsection{Snakes and co-snakes}\label{subsec:snakes}

Recall that by Lemma~\ref{lem:bijections}, the successor and co-successor functions are bijections on $\Live(\cals)$. Thus, they define a group, which we will denote by $G(\cals)=\<s,c\>$ or by $G(\calx)=\<s,c\>$, depending on whether we are using the notation of scrolls or ticker tapes. Clearly, the generators have infinite order, and by Proposition~\ref{prop:successor-of-co-successor}, this group is abelian. We will call $G(\cals)$ the \dfn{snake group}. Our general approach for understanding the structure of the scrolls, and hence understanding toggling independent sets of the cycle graph, is to study this group and its actions. In this section, we will investigate the snake group's relations and derive a presentation.

Given a group $G$ acting on a set $X$, we say $X$ is a \dfn{torsor} for $G$ (or a \dfn{$G$-torsor}) if the action is simply transitive, i.e., it is both transitive and free. When this is the case, there is a bijection between $G$ and $X$, and for a fixed generating set $S\subseteq G$, this action defines a Cayley graph structure on $X$. Specifically, the (left) Cayley graph for $G=\<S\mid R\>$ has vertex set $X$, and for each $x\in X$ and generator $g\in S$, there is a directed edge $x\to g.x$, annotated with $g$ (often by color). In our setting, there is a canonical action of the snake group $G(\cals)$ on the set $\Live(\cals)$ of live entries, and we will prove that this makes $\Live(\cals)$ into a $G(\cals)$-torsor. 

At times, it will be helpful to lift up to the universal scroll and work with the action of the \dfn{affine snake group} $G(\widehat\cals):=\left<\widehat{s},\widehat{c}\right>$
on $\Live(\widehat\cals)$. The actions of these two snake groups on the corresponding sets of live entries are described by the following commutative diagram, which relates the successor functions $s$ and $\widehat{s}$. Here, $\q$ is the quotient map from Eq.~\eqref{eqn:cd-universal}:
\[
\xymatrix{
\Live(\widehat\cals)\ar[d]_{\q}\ar[r]^{\widehat{s}} & \Live(\widehat\cals)\ar[d]^\q \\
\Live(\cals)\ar[r]^{s} & \Live(\cals)}\hspace{15mm}
\xymatrix{
(i+k,j+kn)\ar@{|->}[d]_{\q}\ar@{|->}^{\widehat{s}}[r] & \widehat{s}(i+k,j+kn) \ar^\q@{|->}[d] \\
(i,j)\ar^{s}@{|->}[r] & s(i,j)}
\]
Naturally, there is an analogous diagram relating the co-successor functions $c$ and $\widehat{c}$. 
  
\begin{definition}\hypersetup{hidelinks}
Given a live entry $(i,j)$ in a scroll, the \dfn{snake} and the \dfn{co-snake} containing it are the following subsets of $\zz\times\zz_n$:
\[
\Snake(i,j)=\big\{s^k(i,j)\mid k\in\zz\big\},\qquad\qquad
\Cosnake(i,j)=\big\{c^k(i,j)\mid k\in\zz\big\}.
\]
The \dfn{affine snake} and \dfn{affine co-snake} are defined similarly, but for $\widehat{s}$ and $\widehat{c}$ in the universal scroll. We will denote these by $\Affsnake(i,j)$ 
and $\Affcosnake(i,j)$, respectively\href{http://youtube.com/watch?v=amYzBQMT4VI}{.}
\end{definition} 

Returning to our running example, the two snakes in the scroll shown in Figure~\ref{fig:example} are highlighted by color in the table on the left. The six co-snakes are not as visually prominent in the scroll, but the live entries in the table on the right are colored to distinguish them. There are always infinitely many affine snakes and co-snakes. In the example in Figure~\ref{fig:universal-scroll}, each affine snake is colored according to the snake to which the quotient map $\q$ (from Eq.~\eqref{eqn:cd-universal}) sends it. 
\begin{remark}
The term \emph{snake} is borrowed from a paper by the third author and Roby  about toggling independent sets of a path graph~\cite{joseph2018toggling}. It was chosen because of the visual interpretation resulting from iterating the successor function from a given entry in a scroll. In \cite{haddadan2021some}, Haddadan studied snakes in \emph{tuple boards} to analyze the dynamics of \emph{comotion} on order ideals, which is also defined via toggles.
\end{remark}

To understand the action of the snake group on the live entries in the scroll, it is easiest to consider the action of the affine snake group in the universal scroll and then project downwards.\footnote{We will continue to write snake groups multiplicatively due to their definitions in terms of function composition, and because most snakes are not adders.}

\begin{lemma}
The affine snake group is free abelian, i.e., it has presentation 
\[
G(\widehat\cals)=\left<\widehat{s},\widehat{c}\mid \widehat{s}\,\widehat{c}=\widehat{c}\,\widehat{s}\,\right>.
\]
\end{lemma}

\begin{proof}
It suffices to show that the action of the abelian group $G(\widehat\cals)$ on $\Live(\widehat\cals)\subset\zz\times\zz$ is free. Consider an element $\widehat{s}^k\widehat{c}^\ell$ that fixes $(i,j)$. Since the action of $\widehat{s}$ increases the second coordinate and the action of $\widehat{c}$ decreases it, $k$ and $\ell$ cannot have opposite signs. Without loss of generality, assume $k,\ell\geq 0$. Since $c$ increases the first coordinate, we have $\ell=0$. Then $\widehat{s}^k(i,j)=(i,j)$, so $k=0$. 
\end{proof}

The next lemma tells us that the live entries in the universal scroll form a $G(\widehat\cals)$-torsor.

\begin{lemma}\label{lem:affine-transitivity}
The affine snake group acts simply transitively on $\Live(\widehat\cals)$. 
\end{lemma}

\begin{proof}
Since $G(\widehat{S})$ is free abelian, it suffices to show that the action is transitive. Consider the affine snake containing $(i,j)\in\Live(\widehat\cals)$. There is another affine snake below it containing $\widehat{c}(i,j)$ that differs by a translation of $(-1,2)$ or $(-2,2)$. Similarly, there is one above it containing $\widehat{c}^{-1}(i,j)$, which differs by a translation of $(1,-2)$ or $(2,-2)$. Clearly, there is no room for live entries between any two such consecutive snakes. In particular, this means that we can get from any live entry $(i,j)$ to another $(i',j')$ in $\Live(\widehat\cals)$ by first applying $\widehat{c}^\ell$ for some $\ell\in\zz$, to traverse from $\Affsnake(i,j)$ to $\Affsnake(i',j')$, and then applying $\widehat{s}^k$ for some appropriate $k\in\zz$ to move within the affine snake. 
\end{proof}

Since $\Live(\widehat\cals)$ is a $G(\widehat\cals)$-torsor, the affine snakes are in bijection with the cosets of $\left<\widehat{s}\right>$, and the affine co-snakes are in bijection with the cosets of $\left<\widehat{c}\right>$. Moreover, there are bijections between the elements of these cosets and those in the (co)-snakes. The quotient $\q\colon\Live(\widehat{\cals})\to\Live(\cals)$ is a topological covering map, so it induces a group homomorphism $\q^*\colon G(\widehat{\cals})\to G(\cals)$ with $\q^*(\widehat{s})=s$ and $\q^*(\widehat{c})=c$. The snake group is the quotient
\[
G(\cals)\cong G(\widehat{\cals})/\ker\q^*,
\]
and it acts simply transitively on $\Live(\widehat{\cals})/\ker\q$, which can be canonically identified with $\Live(\cals)$.

\begin{proposition}\label{prop:snake-present}
The set $\Live(\cals)$ is a torsor for the snake group, which has presentation
\[
G(\cals)=\big<s,c\mid sc=cs,\;s^\beta=c^\alpha\big>,
\]
where $\cals$ has $\alpha$ snakes and $\beta$ co-snakes. 
\end{proposition}

\begin{proof}
We have already established the first statement. It follows that there is a bijective correspondence between snakes and cosets of $\<s\>$ and between co-snakes and cosets of $\<c\>$. That is,
\[
\alpha=[G(\cals):\<s\>]\quad\text{and}\quad
\beta=[G(\cals):\<c\>]
\]
are the smallest positive integers for which $s^\beta\in\<c\>$ and $c^\alpha\in\<s\>$. It follows that $s^\beta=c^{\pm\alpha}$. To resolve the sign ambiguity, we notice that applying $c$ will always increase the second coordinate of a live entry, while applying $s$ cannot possibly decrease this coordinate. 

Any relation in $G(\cals)$ other than $sc=cs$ arises from an element of $\ker\q^*$, and these all have the form $\widehat{s}^{\,b}\widehat{c}^{\,a}$ for some $a,b\in\zz$. Since both generators $s=\q^*(\widehat{s})$ and $c=\q^*(\widehat{c})$ have infinite order, we may thus assume that $\ker\q^*=\left<\widehat{s}^{\,b}\widehat{c}^{\,a}\right>$ for some $a,b\neq 0$. By minimality of $\alpha$ and $\beta$, we may take $a=-\alpha$ and $b=\beta$.
\end{proof}

\subsection{Slithers and co-slithers}\label{subsec:slithers}

Since $G(\cals)$ endows $\Live(\cals)$ with the structure of a Cayley graph, if we fix $(i,j)\in\Live(\cals)$, then every word in $\{s,s^{-1},c,c^{-1}\}$ corresponds to a \dfn{path} from $(i,j)$. The snakes and co-snakes correspond to the cosets of $\<s\>$ and $\<c\>$, respectively. In particular, this means that all snakes have the same algebraic structure, as do all co-snakes. In this section, we will prove a stronger result: the embeddings of all snakes and co-snakes in a given scroll additionally have the same ``shape.'' As before, we will let $\alpha=[G(\cals):\<s\>]$ be the number of snakes and $\beta=[G(\cals):\<c\>]$ be the number of co-snakes. Though we will work with scrolls, all of our definitions and results can also be translated into ticker tape notation. 

From a fixed $(i,j)\in\Live(\cals)$, consider the next live entry reached when applying the successor or co-successor function. There are two cases for each, as was shown back in Figure~\ref{fig:successor-co-successor}. We will annotate a step of $s(i,j)=(i+1,j+1)$ by ``$D$'' for ``diagonal'' and a step of $s(i,j)=(i,j+2)$ by ``$\E$'' for ``east.''\footnote{Earlier conference papers involving this work, such as AUTOMATA \cite{david2021toggling} and FPSAC \cite{defant2022torsors} used `$2$' instead of `$E$.'} Similarly, we will annotate a step of $c(i,j)=(i+2,j-1)$ by ``$S$'' for ``short,'' and $c(i,j)=(i+2,j-2)$ by ``$L$'' for ``long.'' Allowing inverses, it is straightforward to annotate any path in the Cayley graph of $G(\cals)$. We will call this the \dfn{shape} of a path. If a path has length $1$, then we will refer to it as a \dfn{step} and refer to its shape as its \dfn{type}. For example, the step from $(i,j)$ to $s(i,j)$ is either of type $D$ or of type $\E$. At times, it will be convenient to speak of a $D$-step or $\E$-step (these are ``\dfn{$s$-steps}''), or of an $S$-step or $L$-step (these are ``\dfn{$c$-steps}''). An \dfn{$s$-path} is a sequence of $s$-steps (inverses allowed); a \dfn{$c$-path} is defined similarly.

Since the snake group is abelian, we have $s(c(i,j))=c(s(i,j))$ for all $(i,j)\in\Live(\cals)$. Thus, if we start from $(i,j)$, then applying an $s$-step and then a $c$-step results in the same endpoint as applying a $c$-step and then an $s$-step. A simple but useful observation is that we also take the same \emph{types} of steps along these two paths, but in the opposite order. Geometrically, this means that paths formed by applying $s^{-1}c^{-1}sc$ (and hence $c^{-1}s^{-1}cs$) always trace out parallelograms, and any scroll is tiled by these parallelograms. There are four such parallelograms, as shown in Figure~\ref{fig:parallelogram}. Specifically, in each of the four neighborhoods, there is a parallelogram formed by both $1$s and both $a$'s, and another one formed by both $1$s and both $\overline{a}$'s. This double-counts each parallelogram, leaving four distinct 1s. In other words, every possible scroll is described by a periodic tiling of parallelograms on a cylinder. However, there are restrictions to which of these tilings are possible, which we will explore further in Section~\ref{sec:classification}.

\begin{lemma}[Parallelogram lemma]\label{lem:parallelogram}
Starting from any $(i,j)\in\Live(\cals)$, the paths $(i,j)\to c(i,j)\to s(c(i,j))$ and $(i,j)\to s(i,j)\to c(s(i,j))$ have the same types of steps but in the opposite order. 
\end{lemma}

\begin{proof}
It suffices to check all possible cases, which are shown in Figure~\ref{fig:parallelogram}. The two diagrams on the left show the possible parallelograms formed by applying $s^{-1}c^{-1}sc$ to ${\color{Red}(i,j)}$, starting with an $L$ and an $S$, respectively. The diagrams on the right show the possible parallelograms formed by applying $c^{-1}s^{-1}cs$ to ${\color{Green}(i,j)}$, starting with a $\E$ and a $D$, respectively.
\end{proof}

\begin{figure}[!ht]
\begin{center}
\begin{tikzpicture}
\tikzstyle{every node}=[font=\small,anchor=south]
\begin{scope}[shift={(0,0)},scale=.36]
    \node at (1.75,4.75) {$\cdots$};
    \node at (3,4.75) {$\0$}; 
    \node at (4,4.75) {$\0$};
    \node at (5.5,4.75) {$\cdots$};
    \node at (.75,3.5) {$\cdots$};
    \node at (2,3.5) {$\0$};
    \node at (3,3.5) {$\color{Red}\mathbf{1}$};
    \node at (4,3.5) {$\0$};
    \node at (5,3.5) {$\color{Red}\mathbf{a}$};
    \node at (6.5,3.5) {$\cdots$};
    \node at (-.25,2.25) {$\cdots$};
    \node at (1,2.25) {$\0$}; 
    \node at (2,2.25) {$\0$};
    \node at (3,2.25) {$\0$}; 
    \node at (4,2.25) {$\color{red}\overline{\mathbf{a}}$};
    \node at (5,2.25) {$\0$};
    \node at (6.5,2.25) {$\cdots$};
    \node at (-1.25,1) {$\cdots$};
    \node at (0,1) {$\0$};
    \node at (1,1){$\color{blue}\mathbf{1}$}; 
    \node at (2,1) {$\0$};
    \node at (3,1) {$\color{blue}\mathbf{a}$};
    \node at (4,1) {$\0$};
    \node at (5.5,1) {$\cdots$};
    \node at (0,-.25) {$\0$};
    \node at (1,-.25) {$\0$};
    \node at (-1.25,-.25) {$\cdots$};
    \node at (2,-.25) {$\color{blue}\overline{\mathbf{a}}$};
    \node at (3,-.25) {$\0$};
    \node at (4.5,-.25) {$\cdots$};
\end{scope}
\begin{scope}[shift={(3.75,0)},scale=.36]
    \node at (1.75,4.75) {$\cdots$};
    \node at (3,4.75) {$\0$}; 
    \node at (4,4.75) {$\0$};
    \node at (5.5,4.75) {$\cdots$};
    \node at (0.75,3.5) {$\cdots$};
    \node at (4,3.5) {$\0$};
    \node at (3,3.5) {$\color{Red}\mathbf{1}$}; 
    \node at (2,3.5) {$\0$};
    \node at (5,3.5) {$\color{Red}\mathbf{a}$};
    \node at (6.5,3.5) {$\cdots$};
    \node at (0.75,2.25) {$\cdots$};
    \node at (3,2.25) {$\0$}; 
    \node at (2,2.25) {$\0$};
    \node at (4,2.25) {$\color{Red}\overline{\mathbf{a}}$};
    \node at (5,2.25) {$\0$};
    \node at (6.5,2.25) {$\cdots$};
    \node at (-.25,1) {$\cdots$};
    \node at (1,1) {$\0$};
    \node at (2,1) {$\color{blue}\mathbf{1}$};
    \node at (3,1) {$\0$};
    \node at (4,1) {$\color{blue}\mathbf{a}$};
    \node at (5.5,1) {$\cdots$};
    \node at (-.25,-.25) {$\cdots$};
    \node at (2,-.25) {$\0$};
    \node at (1,-.25) {$\0$};
    \node at (3,-.25) {$\color{blue}\overline{\mathbf{a}}$};
    \node at (4,-.25) {$\0$};
    \node at (5.5,-.25) {$\cdots$};
\end{scope}
\begin{scope}[shift={(9,0)},scale=.36]
    \node at (1.75,4.75) {$\cdots$};
    \node at (3,4.75) {$\0$}; 
    \node at (4,4.75) {$\0$};
    \node at (5,4.75) {$\0$};
    \node at (6,4.75) {$\0$};
    \node at (7.5,4.75) {$\cdots$};
    \node at (.75,3.5) {$\cdots$};
    \node at (2,3.5) {$\0$};
    \node at (3,3.5) {$\color{Green}\mathbf{1}$};
    \node at (4,3.5) {$\0$};
    \node at (5,3.5) {$\color{orange}\mathbf{1}$};
    \node at (6,3.5) {$\0$};
    \node at (7.5,3.5) {$\cdots$};
    \node at (.75,2.25) {$\cdots$};
    \node at (2,2.25) {$\0$};
    \node at (3,2.25) {$\0$}; 
    \node at (4,2.25) {$\0$};
    \node at (5,2.25) {$\0$};
    \node at (6.5,2.25) {$\cdots$};
    \node at (-.25,1) {$\cdots$};
    \node at (1,1){$\color{Green}\mathbf{a}$}; 
    \node at (2,1) {$\color{Green}\overline{\mathbf{a}}$};
    \node at (3,1) {$\color{orange}\mathbf{a}$};
    \node at (4,1) {$\color{orange}\overline{\mathbf{a}}$};
    \node at (5.5,1) {$\cdots$};
    \node at (1,-.25) {$\0$};
    \node at (-.25,-.25) {$\cdots$};
    \node at (2,-.25) {$\0$};
    \node at (3,-.25) {$\0$};
    \node at (4.5,-.25) {$\cdots$};
\end{scope}
\begin{scope}[shift={(12.75,0)},scale=.36]
    \node at (1.75,4.75) {$\cdots$};
    \node at (3,4.75) {$\0$}; 
    \node at (4,4.75) {$\0$};
    \node at (5.5,4.75) {$\cdots$};
    \node at (0.75,3.5) {$\cdots$};
    \node at (4,3.5) {$\0$};
    \node at (3,3.5) {$\color{Green}\mathbf{1}$}; 
    \node at (2,3.5) {$\0$};
    \node at (5,3.5) {$\0$};
    \node at (6.5,3.5) {$\cdots$};
    \node at (0.75,2.25) {$\cdots$};
    \node at (3,2.25) {$\0$}; 
    \node at (2,2.25) {$\0$};
    \node at (4,2.25) {$\color{orange}\mathbf{1}$};
    \node at (5,2.25) {$\0$};
    \node at (6.5,2.25) {$\cdots$};
    \node at (-.25,1) {$\cdots$};
    \node at (1,1) {$\color{Green}\mathbf{a}$};
    \node at (2,1) {$\color{Green}\overline{\mathbf{a}}$};
    \node at (3,1) {$\0$};
    \node at (4,1) {$\0$};
    \node at (5.5,1) {$\cdots$};
    \node at (-.25,-.25) {$\cdots$};
    \node at (2,-.25) {$\color{orange}\mathbf{a}$};
    \node at (1,-.25) {$\0$};
    \node at (3,-.25) {$\color{orange}\overline{\mathbf{a}}$};
    \node at (4.5,-.25) {$\cdots$};
    \node at (0.75,-1.5) {$\cdots$};
    \node at (2,-1.5) {$\0$};
    \node at (3.5,-1.5) {$\cdots$};
\end{scope}
\end{tikzpicture}
\caption{Scrolls are tiled by up to four types of parallelograms; see the proof of Lemma~\ref{lem:parallelogram}. Each of the four diagrams here depicts two: one for each choice of $a\in\{0,1\}$. Each parallelogram appears exactly twice.}\label{fig:parallelogram}
\end{center}
\end{figure}
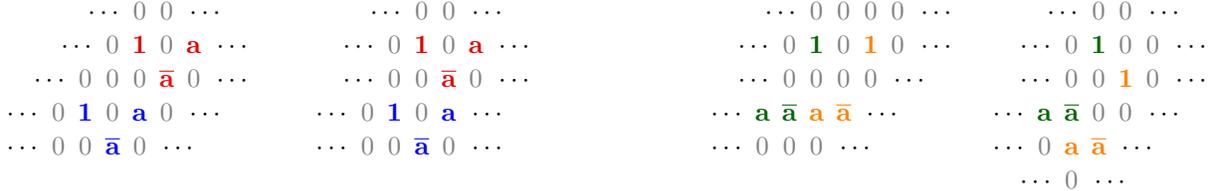

The parallelogram lemma guarantees that the notion of the path shape is well-defined on snakes and co-snakes. In other words, applying the successor (resp., co-successor) function from any two live entries in the same co-snake (resp., snake) yields the same type of step. 

\begin{corollary}\label{cor:path-type}
Suppose $(i',j')\in\Snake(i,j)$ and $(i'',j'')\in\Cosnake(i,j)$. Then the steps $(i,j)\to c(i,j)$ and $(i',j')\to c(i',j')$ have the same type (either both $S$ or both $L$), and the steps $(i,j)\to s(i,j)$ and $(i'',j'')\to s(i'',j'')$ have the same type (either both $D$ or both $\E$). 
\end{corollary}

\begin{proof}
We will show that $(i,j)\to c(i,j)$ and $(i',j')\to c(i',j')$ have the same type; the proof of the analogous statement for $(i,j)\to s(i,j)$ and $(i'',j'')\to s(i'',j'')$ is very similar. Since $\Snake(i,j)$ is the orbit of $(i,j)$ under $s$, it suffices to prove the desired result when $(i',j')=s(i,j)$. However, this is immediate from the parallelogram lemma.  
\end{proof}

The elements of the cyclic quotient group $G(\cals)/\<c\>\cong\zz_\beta$ correspond to the co-snakes in $\cals$. Thus, starting at any live entry $(i,j)$ and iterating the successor function $\beta$ times defines an ordering of the co-snakes. We call the shape of this length-$\beta$ path the \dfn{slither from $(i,j)$}, denoted $\Slither(i,j)$. For ease of notation, we can use exponents to write a slither. In our running example back in Figure~\ref{fig:example}, we have $\Slither(1,1)=
\E D\E D\E D=(\E D)^3$ and $\Slither(3,2)=D\E D\E D\E=(D\E)^3$.

There is a similar construction for co-snakes. The elements of the cyclic quotient group $G(\cals)/\<s\>\cong\zz_\alpha$ are the snakes in $\cals$. Starting at any live entry $(i,j)$ and iterating the co-successor function $\alpha$ times defines an ordering of the snakes. We call the shape of this length-$\alpha$ path the \dfn{co-slither from $(i,j)$}, denoted $\Coslither(i,j)$. In our running example from Figure~\ref{fig:example}, we have $\Coslither(i,j)=SS=S^2$ for every live entry $(i,j)$. The following is immediate from Corollary~\ref{cor:path-type}. 

\begin{lemma}\label{lem:shiftslither}
Let $(i,j) \in \Live(\cals)$ and $(i',j') = s^ac^b(i,j)$ for some $a,b \in \mathbb Z$. 
\begin{itemize}
    \item The slither from $(i',j')$ is the slither from $(i,j)$, cyclically shifted $a$ positions. 
    \item The co-slither from $(i',j')$ is the co-slither from $(i,j)$, cyclically shifted $b$ positions. $\hfill\blacksquare$
\end{itemize}
\end{lemma}

By Lemma~\ref{lem:shiftslither}, it is well-defined to speak of the \dfn{slither of $\cals$}, written $\Slither(\cals)$, as the slither from any live entry of $\cals$ up to cyclic shift. The \dfn{co-slither of $\cals$}, written $\Coslither(\cals)$, is defined similarly.

\subsection{Scales and periods}

Here we formalize the notion of the ``exponent'' that we used to write slithers and co-slithers, and we use this notion to calculate properties of scrolls and the ticker tapes. By Lemma~\ref{lem:shiftslither}, these exponents are inherent properties of a scroll. Not only are they notationally convenient, but they underlie a fundamental property that will be useful to us later in this paper. In Section~\ref{sec:tables}, we will relate these to the periods of both the scrolls and ticker tapes. We will also use a slightly different terminology: the word ``\emph{degree}'' will both suggest the fact that it appears as an exponent and indicate its relation to the degree of a certain covering map that we will introduce in Section~\ref{sec:tables}.

\begin{definition}\label{def:exponent}
The \dfn{degree} of a scroll, denoted $\deg(\cals)$, is the length of $\Slither(\cals)$ divided by its period as a cyclic word. The \dfn{co-degree} of a scroll, denoted $\codeg(\cals)$, is the length of $\Coslither(\cals)$ divided by its period as a cyclic word. 
\end{definition}

The scroll in our running example from Figure~\ref{fig:example} has degree $3$ because its slither is $(D\E)^3$. Note that this must divide $\beta$, the number of co-snakes. Similarly, this scroll has co-degree $2$ because its co-slither is $S^2$. This number divides $\alpha$, the number of snakes.

The definitions of slithers, co-slithers, degrees, and co-degrees can be defined analogously on ticker tapes. We write $\Slither(k)$ (resp.,  $\Coslither(k)$) for the slither (resp., co-slither) from the $k^{\text{th}}$ entry of $\calx$. We also define $\deg(\calx)=\deg(\cals)$ and $\codeg(\calx)=\codeg(\cals)$. While it is usually easy to visualize results in the scroll, sometimes it is notationally simpler to work with ticker tapes. 

Going forward, it will be fruitful to write the slither as $\deg(\calx)$ copies of a word $P$ over $\{D,\E\}$ and the co-slither as $\codeg(\calx)$ copies of a word $Q$ over $\{S,L\}$. Notationally, for any fixed $k \in \Live(\calx)$, we will write
\begin{equation}\label{eqn:P-and-Q}
\Slither(k) = P^{\deg(\calx)},\quad
\beta=|P| \cdot \deg(\calx),\quad\quad
\Coslither(k)=Q^{\codeg(\calx)},\quad \alpha=|Q|\cdot \codeg(\calx).
\end{equation}
We will also want to speak of the number of spaces that we advance in the ticker tape upon applying $P$ and $Q$; we will denote these as
\begin{equation}\label{eqn:p-and-q}
    p=s^{|P|}(k)-k,\quad\text{and}\quad q=c^{|Q|}(k)-k.
\end{equation}
By Lemma~\ref{lem:shiftslither}, these do not depend on $k$, though the actual words $P$ and $Q$ do. Note that $p$ and $q$ are the minimal shifts of the ticker tape that leave invariant the snakes and co-snakes, respectively. Going forward, we will refer to $p$ as the \dfn{snake scale} and $q$ as the \dfn{co-snake scale}.

In our running example from Figure~\ref{fig:example}, if we take $k=12$, which is the live entry $(i,j)=(1,1)$ one row beneath the non-live entry in the upper-left corner, then
\[
\Slither(k)=(\E D)^3,\qquad 
\beta=6,\qquad
P=\E D,\qquad \deg(\calx)=3,\qquad 
p=14,
\]
and
\[
\Coslither(k)=S^2,\qquad 
\alpha=2,\qquad
Q=S,\qquad\codeg(\calx)=2,\qquad q=21.
\]
Starting at an arbitrary live entry $k$, each time we traverse the path corresponding to $P$, we go forward in the ticker tape $p$ entries. For any $r\in\nn$, traversing the path $P^r$ advances us $rp$ entries. Though this is intuitive, we will formalize and prove it below, as it will be a useful technical lemma for a number of results.

\begin{lemma}\label{lem:log-like}
  For any $r \in \zz_{\geq 0}$ and any $k \in \Live(\calx)$, we have
  \[
  s^{r|P|}(k)-k=r\big(s^{|P|}(k)-k\big)=rp\qquad \text{ and }\qquad c^{r|Q|}(k)-k=r\big(c^{|Q|}(k)-k\big)=rq.
  \]
\end{lemma}

\begin{proof}
For the first statement, it suffices to show that $s^{r|P|}(k)=rp+k$, and we will do this by induction. The cases when $r=0$ or $r=1$ are immediate. Assuming the hypothesis holds for $r$, we have
\[
s^{(r+1)|P|}(k)=s^{|P|}(s^{r|P|}(k))=
s^{|P|}(rp+k)=rp+k+p
=(r+1)p+k.
\]
The proof of the second equality is completely analogous; just replace $s$ with $c$, $|P|$ with $|Q|$, and $p$ with $q$.
\end{proof}

Recall that $\alpha$ and $\beta$ are the number of snakes and co-snakes respectively of $S$. 
\begin{lemma}\label{lem:co-prime-new}
For any ticker tape $\calx$, the values of $\deg(\calx)$ and $\codeg(\calx)$ are relatively prime. 
\end{lemma}

\begin{proof}
Let $d=\gcd(\deg(\calx),\codeg(\calx))$. By Lemma~\ref{lem:log-like}, 
\[ 
s^{\beta}(k)-k
=s^{\deg(\calx)|P|}(k)-k
=d(s^{\beta/d}(k)-k).
\]
Since $s^\beta(k)=c^\alpha(k)$, the quantity above is equal to 
\[ 
c^{\alpha}(k)-k
=c^{\codeg(\calx)|Q|}(k)-k
=d(c^{\alpha/d}(k)-k),
\]
and thus $s^{\beta/d}(k) = c^{\alpha/d}(k)$. It now follows from the presentation of the snake group (Proposition~\ref{prop:snake-present}) that $d = 1$. 
\end{proof}

\begin{lemma}\label{lem:lcmdeg}
For any ticker tape $\calx$,
\[
\lcm(p,q)=\deg(\calx)p=\codeg(\calx)q.
\]
\end{lemma}

\begin{proof}
By definition, we have the following chain of equalities. 
\[
\deg(\calx)p=\deg(\calx)\left(s^{|P|}(k)-k\right)
=s^{\beta}(k)-k=c^{\alpha}(k)-k
=\codeg(\calx)\left(c^{|Q|}(k)-k\right)
=\codeg(\calx)q.
\]
From the presentation of the snake group, $\alpha$ and $\beta$ are the smallest positive integers for which the middle equality holds. Thus, $\deg(\calx)p$ is the smallest multiple of $p$ that is a multiple of $q$, and $\codeg(\calx)q$ is the smallest multiple of $q$ that is a multiple of $p$. The lemma follows. 
\end{proof}

The quantity in Lemma~\ref{lem:lcmdeg} is significant enough that we will give it its own name. 

\begin{definition}\label{def:scale} 
The \dfn{scale} of a ticker tape $\calx$ is
\[
\Scale(\calx):=s^\beta(k)-k=c^\alpha(k)-k,
\]
where $\alpha$ is the number of snakes, $\beta$ is the number of co-snakes, and $k$ is an arbitrary integer. If $\cals$ is the corresponding scroll, then we define the \dfn{scale} of $\cals$ to be $\Scale(\cals):=\Scale(\calx)$.
\end{definition}

\begin{remark}
The snake scale $p$ and co-snake scale $q$ of $\calx$ (or $\cals$) are
\[
p=\frac{\Scale(\calx)}{\codeg(\calx)},\quad\text{and}\quad q=\frac{\Scale(\calx)}{\deg(\calx)}.
\]
\end{remark}

Since the scale is the minimal positive integer $\sigma$ for which $k$ and $k+\sigma$ always lie on the same snake and same co-snake, we have the following algebraic interpretation of the scale in terms of cosets.

\begin{cor}\label{cor:scale-lcm}
The scale of a ticker tape is
\[
\Scale(\calx)=\lcm(p,q).
\]
\end{cor}

\begin{proof}
This follows immediately from the proof of Lemma~\ref{lem:lcmdeg}. 
\end{proof}

\begin{definition}
The \dfn{fibers} of a scroll $\cals$ (or ticker tape $\calx$) are the equivalence classes of live entries defined by intersecting the snakes and co-snakes. In both notations, the fiber of a live entry is denoted
\begin{equation}\label{eq:fiber}
\Fiber(i,j)=\Snake(i,j)\cap\Cosnake(i,j),\qquad\Fiber(k)=\Snake(k)\cap\Cosnake(k).
\end{equation}
\end{definition}

Algebraically, the fibers are just the orbits under the action of the cyclic group
\[
\<s\>\cap\<c\>=\big\langle s^\beta\big\rangle=\<c^\alpha\>\leq G(\cals), 
\]
and the scale is the smallest positive integer $\sigma$ for which two live entries in $\calx$ are in the same fiber if and only if they differ by a multiple of $\sigma$ in the ticker tape.

We will now explore the periodicity of the scroll and ticker tape. Though the scroll and ticker tape are both periodic, we measure their periods in different ways.

\begin{definition}\label{def:periods}
The \dfn{period} $T(\cals)$ of a scroll $\cals=(X_{i,j})$ is the smallest $m>0$ such that $X_{i+m,j}=X_{i,j}$ for all $i,j$. 

The \dfn{period} $T(\calx)$ of a ticker tape $\calx=(X_k)$ is the smallest $\ell>0$ such that $X_{k+\ell}=X_k$ for all $k$.
\end{definition}

In our running example from Figure~\ref{fig:example}, the period of the ticker tape and scroll are both $7$. In particular, the ticker tape is generated by the subsequence $1010000$, and the scroll consists of $7$ repeating rows.  For the example back in Figure~\ref{fig:motivating-example}, the period of the ticker tape is 45, while the period of the scroll is 15.

\begin{lemma}\label{lem:when-equiv}
Suppose that $k \in \Live(\calx)$ and $\ell\in \mathbb Z_{\ge 0}$. Then $T(\calx)$ divides $\ell$ if and only if 
$\Slither(k) = \Slither(k+\ell)$ and $\Coslither(k) = \Coslither(k+\ell)$
\end{lemma}

\begin{proof}
The slither and co-slither from any live entry completely determine $\calx$. In particular, when the ticker tape is shifted $\ell$ positions, it remains unchanged. conversely, if $\Slither(k) \not= \Slither(k+\ell)$ or $\Coslither(k) \not= \Coslither(k+\ell)$, then the ticker tape changes after shifting by $\ell$. 
\end{proof}

As a corollary, we get an analogous characterization of the period of a scroll, which will be useful later.

\begin{corollary}\label{cor:when-equiv}
For any $(i,j)\in\Live(\cals)$, the period $T(\cals)$ is the minimal $\ell>0$ such that the following three conditions hold: 
\begin{enumerate}
\item $(i+\ell,j)\in\Live(\cals)$,
\item $\Slither(i,j)=\Slither(i+\ell,j)$,
\item $\Coslither(i,j)=\Coslither(i+\ell,j)$.
\end{enumerate}
\end{corollary}

\begin{theorem}\label{thm:ttperiod}
The period of the ticker tape $\calx$ is
\[
T(\calx)=\gcd(p,q) =\frac{\Scale(\calx)}{\deg(\calx)\codeg(\calx)}.
\]
\end{theorem}

\begin{proof}
For the first equality, it is immediate from the definition of $p$ that $\Slither(k) = \Slither(k+\ell)$ if and only if $\ell$ is a multiple of $p$. Similarly, $\Coslither(k) = \Coslither(k+\ell)$ if and only if $\ell$ is a multiple of $q$. The equality follows from Lemma~\ref{lem:when-equiv}.

For the second equality, by Lemma~\ref{lem:lcmdeg}, as well as the fact that $pq=\lcm(p,q)\gcd(p,q)$, we have
\[
\gcd(p,q)\deg(\calx)\codeg(\calx)
=\gcd(p,q)\frac{\lcm(p,q)}{p}\cdot\frac{\lcm(p,q)}{q}
=\lcm(p,q)=\Scale(\calx).
\]
by Corollary~\ref{cor:scale-lcm}.
\end{proof}

\begin{corollary}\label{cor:scroll-period}
The period of a scroll $\cals$ is 
\[
T(\cals)=\frac{T(\calx)}{\gcd(T(\calx),n)}= \frac{\lcm(T(\calx),n)}{n}
=\frac{\Scale(\calx)}{\deg(\calx)\codeg(\calx)\gcd(T(\calx),n)}.
\]
\end{corollary}

\begin{proof}
It follows from the definition that the period $T(\cals)$ is the minimum positive integer $m$ such that $n$ divides $mT(\calx)$. Therefore, $T(\cals)$ must be $T(\calx)/\gcd(T(\calx),n)$. The result now follows from Theorem~\ref{thm:ttperiod}.
\end{proof}

\section{Combinatorial characterization and enumeration of dynamics}\label{sec:classification}

In this section, we will characterize when a given pair of a potential slither (a sequence of $D$s and $\E$s) and co-slither (a sequence of $S$s and $L$s) defines a ticker tape. This will allow us to characterize which infinite sequences can arise as ticker tapes and to enumerate them.

\subsection{Computing slithers and subslithers}\label{subsec:classification}

Define a \dfn{substring} of a ticker tape to be a finite sequence $X_i,\ldots,X_{i+r}$ of consecutive entries.
We will consider a substring of a scroll to be any substring in the corresponding ticker tape. It is clear that a scroll or ticker tape can be reconstructed from any of its length-$n$ substrings. In this section, we will give an algorithm to directly calculate the slither and co-slither based on the gaps between live entries of a length-$n$ substring that begins with a $1$. Each gap corresponds to a substring of the slither, which we will call a \emph{subslither}. The slither is simply the concatenation of subslithers, though with the last step omitted. After establishing this construction, we will use it to combinatorially characterize all possible ticker tapes for a given $n$. 

Since our construction depends on gaps between live entries, it will be necessary to speak of the live entries that appear immediately before and after a given live entry. We will formally define this using ticker tape notation, but it is easy to translate it back into the language of scrolls, if desired. 

\begin{definition}\label{defn:prevnext}
Given a live entry $k\in\Live(\calx)$, its \dfn{previous live entry} $k^-$ and \dfn{next live entry} $k^+$ are
\[
k^-=\max\big\{j<k\mid j\in\Live(\calx)\big\},\qquad
k^+=\min\big\{\ell>k\mid \ell\in\Live(\calx)\big\}.
\]
In scroll notation, we will write these as $(i,j)^-$ and $(i,j)^+$. Two live entries are \dfn{consecutive} if one of them is the next live entry of the other. A \dfn{$0$-block} is a maximal substring of $0$s between consecutive live entries.
\end{definition}

Clearly, any two consecutive live entries are separated by a $0$-block of some length $z\in\{1,\dots,n\}$, and we will canonically assign a sequence of $D$s and $\E$s to each such block. The simplest case is when this $0$-block has length $z=1$, which occurs when $(i,j)^+=(i,j+2)=s(i,j)$. In this case, we just take the path that is a single $s$-step, i.e., a length-1 sequence $\E$.

If the $0$-block has size $z\in\{2,\dots,n\}$, then our path must begin with a $D$. When this happens, we will iteratively apply the successor function until we reach $c\big((i,j)^+\big)$. Recall that the \emph{shape} of this path, starting at $(i,j)$ and ending at $c\big((i,j)^+\big)$, is the resulting sequence of $D$s and $\E$s. The following observation is straightforward but useful, and it is also necessary for the definition of a subslither.

\begin{lemma}\label{lem:row-cosnakes}
If $k,\ell\in\Live(\calx)$ and $|k-\ell|<n$, then $\Cosnake(k)\neq\Cosnake(\ell)$. In particular, any two live entries on the same row of a scroll are contained in different co-snakes. $\hfill\Box$
\end{lemma}

\begin{definition}
  Let $(i,j)\in\Live(\cals)$ be followed by a length-$z$ 0-block. If $z=1$, then the \dfn{subslither} from $(i,j)$ is $\E$; otherwise, it is the shape of the minimal $s$-path from $(i,j)$ to $c\big((i,j)^+\big)$.
\end{definition}

Suppose we start at any live entry $(i,j)$. The slither describes the minimal path of $D$s and $\E$s that returns us to the same co-snake. This path touches every co-snake exactly once. Therefore, the subslither from $(i,j)$ is an initial sequence of the slither from $(i,j)$. It begins with an $\E$ if and only if $(i,j+2)\in\Live(\cals)$, in which case that length-$1$ word is the entire subslither. Otherwise, the subslither must begin with a $D$. To understand where it terminates, we may, without loss of generality, assume that $(i,j)=(i,1)$. There are now two subcases. In the first, the next live entry is $(i,j)^+=(i+1,j+1)$, which occurs precisely when $(i,j)$ is followed by a $0$-block of length $n$. When this happens, the entries in row $i+1$ alternate $0,1,0,1,\dots$, and then all but possibly the last entry in row $i+2$ are 0. Two examples of opposite parity are shown in Figure~\ref{fig:subslithers-case1}.

\begin{figure}[!ht]
\begin{tikzpicture}
\tikzstyle{every node}=[font=\normalsize,anchor=south,color=gray]
\begin{scope}[shift={(0,0)},scale=.5]
    \node at (1,2) {$\color{Red}\mathbf{1}$}; 
    \node at (2,2) {$0$};
    \node at (3,2) {$0$};
    \node at (4,2) {$0$};
    \node at (5,2) {$0$};
    \node at (6,2) {$0$};
    \node at (7,2) {$0$};
    \node at (8,2) {$0$};
    \node at (9,2) {$0$};
    \node at (10,2) {$0$};
    \node at (11,2) {$0$};
    \node at (12,2) {$0$};
    \node at (13,2) {$0$};
    \node at (1,1) {$0$}; 
    \node at (2,1) {$\color{Blue}\mathbf{1}$};
    \node at (3,1) {$0$};
    \node at (4,1) {$\color{black}\mathbf{1}$};
    \node at (5,1) {$0$};
    \node at (6,1) {$\color{black}\mathbf{1}$};
    \node at (7,1) {$0$};
    \node at (8,1) {$\color{black}\mathbf{1}$};
    \node at (9,1) {$0$};
    \node at (10,1) {$\color{black}\mathbf{1}$};
    \node at (11,1) {$0$};
    \node at (12,1) {$\color{Red}\mathbf{1}$};
    \node at (13,1) {$0$};
    \node at (1,0) {$0$}; 
    \node at (2,0) {$0$};
    \node at (3,0) {$0$};
    \node at (4,0) {$0$};
    \node at (5,0) {$0$};
    \node at (6,0) {$0$};
    \node at (7,0) {$0$};
    \node at (8,0) {$0$};
    \node at (9,0) {$0$};
    \node at (10,0) {$0$};
    \node at (11,0) {$0$};
    \node at (12,0) {$0$};
    \node at (13,0) {$\color{Blue}\mathbf{1}$};
    \node at (1,-1) {$0$}; 
    \node at (2,-1) {$\color{black}\mathbf{1}$};
    \node at (3,-1) {$0$};
    \node at (4,-1) {$\color{black}\mathbf{1}$};
    \node at (5,-1) {$0$};
    \node at (6,-1) {$\color{black}\mathbf{1}$};
    \node at (7,-1) {$0$};
    \node at (8,-1) {$\color{black}\mathbf{1}$};
    \node at (9,-1) {$0$};
    \node at (10,-1) {$\color{black}\mathbf{1}$};
    \node at (11,-1) {$0$};
    \node at (12,-1) {$0$};
    \node at (13,-1) {$0$};
\end{scope}
\begin{scope}[shift={(8,0)},scale=.5]
    \node at (1,2) {$\color{Red}\mathbf{1}$}; 
    \node at (2,2) {$0$};
    \node at (3,2) {$0$};
    \node at (4,2) {$0$};
    \node at (5,2) {$0$};
    \node at (6,2) {$0$};
    \node at (7,2) {$0$};
    \node at (8,2) {$0$};
    \node at (9,2) {$0$};
    \node at (10,2) {$0$};
    \node at (11,2) {$0$};
    \node at (12,2) {$0$};    
    \node at (1,1) {$0$}; 
    \node at (2,1) {$\color{Blue}\mathbf{1}$};
    \node at (3,1) {$0$};
    \node at (4,1) {$\color{black}\mathbf{1}$};
    \node at (5,1) {$0$};
    \node at (6,1) {$\color{black}\mathbf{1}$};
    \node at (7,1) {$0$};
    \node at (8,1) {$\color{black}\mathbf{1}$};
    \node at (9,1) {$0$};
    \node at (10,1) {$\color{black}\mathbf{1}$};
    \node at (11,1) {$0$};
    \node at (12,1) {$\color{Red}\mathbf{1}$};
    \node at (1,0) {$0$}; 
    \node at (2,0) {$0$};
    \node at (3,0) {$0$};
    \node at (4,0) {$0$};
    \node at (5,0) {$0$};
    \node at (6,0) {$0$};
    \node at (7,0) {$0$};
    \node at (8,0) {$0$};
    \node at (9,0) {$0$};
    \node at (10,0) {$0$};
    \node at (11,0) {$0$};
    \node at (12,0) {$0$};
    \node at (1,-1) {$\color{Blue}\mathbf{1}$};
    \node at (2,-1) {$0$};
    \node at (3,-1) {$\color{black}\mathbf{1}$};
    \node at (4,-1) {$0$};
    \node at (5,-1) {$\color{black}\mathbf{1}$};
    \node at (6,-1) {$0$};
    \node at (7,-1) {$\color{black}\mathbf{1}$};
    \node at (8,-1) {$0$};
    \node at (9,-1) {$\color{black}\mathbf{1}$};
    \node at (10,-1) {$0$};
    \node at (11,-1) {$\color{black}\mathbf{1}$};
    \node at (12,-1) {$0$};
\end{scope}
\end{tikzpicture}
\caption{Starting from the {\color{Red}live entry} $(i,j)$ in the upper-left, the next live entry $(i,j)^+$ is on $\Snake(i,j)$ if and only if $(i,j)$ is followed by exactly $n$ 0s. In this case, the successor of the last live entry in the second row is $c\big((i,j)^+\big)$, so the subslither from $(i,j)$ is $D\E^{\left\lfloor n/2-1 \right\rfloor}D$.} 
\label{fig:subslithers-case1}
\end{figure}

In both subcases shown in Figure~\ref{fig:subslithers-case1}, the last live entry in row $i+1$ is the co-successor of $(i,j)=(i,1)$; this is $(i+1,n-1)$ if $n$ is odd and is $(i+1,n)$ if $n$ is even. Because the snake group is abelian, the next live entry is $c(i+1,j+1)=c((i,j)^+)$. By construction, this is where the subslither from $(i,j)$ stops: at $(i+2,n)$ if $n$ is odd or $(i+3,1)$ if $n$ is even.
It is straightforward to see that the subslither from $(i,j)$ is thus $D\E^{\left\lfloor n/2-1 \right\rfloor}D$.

So far, we have covered the two extreme cases: if the $0$-block that follows $(i,j)$ has (minimal) length $z=1$, then the subslither from $(i,j)$ is $\E$. If the $0$-block has (maximal) length $z=n$, then the subslither is $D\E^{\left\lfloor n/2-1 \right\rfloor}D$. The final case is when the $0$-block has length $z\in\{2,\dots,n-1\}$. An example of this appears in Figure~\ref{fig:subslithers-case2}. We claim that in this case, the subslither is also $D\E^{\left\lfloor z/2-1 \right\rfloor}D$, which happens to match the case of $z=n$.

\begin{lemma}\label{lem:subslithers}
The subslither from $(i,j)\in\Live(\cals)$ is
\[
\begin{cases}
\E & \qquad \mbox{if } z=1 \\
D\E^{\left\lfloor z/2 \right\rfloor-1} D & \qquad \mbox{if } z\in\{2,\dots,n\},
\end{cases}
\]
where $z$ is the length of the $0$-block that follows $(i,j)$.
\end{lemma}

\begin{proof}
We have already verified the cases when $z=1$ and when $z=n$. Now suppose $z\geq 2$. It is elementary to show that the subslither must be of the form $D\E^rD$ for some integer $r\geq 0$. Note that $(i,j)^+=(i,j+z+1)$. After starting at $(i,j)$ and traversing the subslither, we reach the live entry $(i+2,j+2r+2)$, which must also be $c((i,j)^+)$. Now, $c((i,j)^+)$ is either $(i+2,j+z)$ or $(i+2,j+z-1)$ (depending on whether the step from $(i,j)^+=(i,j+z+1)$ to its co-successor is of type $S$ or $L$). It follows that $j+2r+2$ is either $j+z$ or $j+z-1$, so $r=\left\lfloor z/2\right\rfloor-1$.
\end{proof}

\begin{figure}[!ht]
\begin{tikzpicture}
\tikzstyle{every node}=[font=\normalsize,anchor=south,color=gray]
\begin{scope}[shift={(0,0)},scale=.5]
    \node at (1,2) {$\color{Red}\mathbf{1}$}; 
    \node at (2,2) {$0$};
    \node at (3,2) {$0$};
    \node at (4,2) {$0$};
    \node at (5,2) {$0$};
    \node at (6,2) {$0$};
    \node at (7,2) {$0$};
    \node at (8,2) {$\color{Blue}\mathbf{1}$};
    \node at (9,2) {$0$};
    \node at (10,2) {$\color{Blue}\mathbf{1}$};
    \node at (11,2) {$0$};
    \node at (12,2) {$0$};
    \node at (13,2) {$\color{orange}\mathbf{1}$};
    \node at (14,2) {$0$};
    \node at (15,2) {$0$};
    \node at (16,2) {$0$};
    \node at (17,2) {$0$};
    \node at (18,2) {$0$};
    \node at (19,2) {$\color{green}\mathbf{1}$};
    \node at (20,2) {$0$};
    \node at (21,2) {$0$};
    \node at (22,2) {$0$};
    \node at (23,2) {$0$};
    \node at (24,2) {$0$};    
    \node at (1,1) {$0$}; 
    \node at (2,1) {$\color{Red}\mathbf{1}$};
    \node at (3,1) {$0$};
    \node at (4,1) {$\color{Red}\mathbf{1}$};
    \node at (5,1) {$0$};
    \node at (6,1) {$\color{Red}\mathbf{1}$};
    \node at (7,1) {$0$};
    \node at (8,1) {$0$};
    \node at (9,1) {$0$};
    \node at (10,1) {$0$};
    \node at (11,1) {$\color{Blue}\mathbf{1}$};
    \node at (12,1) {$0$};
    \node at (13,1) {$0$};
    \node at (14,1) {$\color{orange}\mathbf{1}$};
    \node at (15,1) {$0$};
    \node at (16,1) {$\color{orange}\mathbf{1}$};
    \node at (17,1) {$0$};
    \node at (18,1) {$0$}; 
    \node at (19,1) {$0$};
    \node at (20,1) {$\color{green}\mathbf{1}$};
    \node at (21,1) {$0$};
    \node at (22,1) {$\color{green}\mathbf{1}$};
    \node at (23,1) {$0$};
    \node at (24,1) {$\color{green}\mathbf{1}$};
    \node at (1,0) {$0$}; 
    \node at (2,0) {$0$};
    \node at (3,0) {$0$};
    \node at (4,0) {$0$};
    \node at (5,0) {$0$};
    \node at (6,0) {$0$};
    \node at (7,0) {$\color{Red}\mathbf{1}$};
    \node at (8,0) {$0$};
    \node at (9,0) {$1$};
    \node at (10,0) {$0$};
    \node at (11,0) {$0$};
    \node at (12,0) {$\color{Blue}\mathbf{1}$};
    \node at (13,0) {$0$};
    \node at (14,0) {$0$};
    \node at (15,0) {$0$};
    \node at (16,0) {$0$};
    \node at (17,0) {$\color{orange}\mathbf{1}$};
    \node at (18,0) {$0$};
    \node at (19,0) {$0$};
    \node at (20,0) {$0$};
    \node at (21,0) {$0$};
    \node at (22,0) {$0$};
    \node at (23,0) {$0$};
    \node at (24,0) {$0$};
\end{scope}
\end{tikzpicture}
\caption{Starting from the {\color{Red}live entry} in the upper-left, the steps {\color{Red}$D\E\E D$} reach the co-snake containing its {\color{Blue}next live entry}. Then, starting at that {\color{Blue}live entry}, we use a single {\color{Blue}$\E$} to reach the co-snake of its {\color{Blue}next live entry}. Then a {\color{Blue} $DD$} takes us to the co-snake of the {\color{orange}next live entry}.
From here, the steps {\color{orange}$D\E D$} reach the co-snake containing the {\color{green}next live entry}. Finally, the steps {\color{green}$D\E\E$} get us back to our {\color{Red}starting co-snake}. The slither of the scroll is simply the concatenation of these steps $D\E\E D\E DDD\E DD\E \E$. The co-slither ${\color{Red}S}{\color{green}L}{\color{orange}S}{\color{Blue}S}$ is obtained from the co-successor functions of live entries in the first row, but taken in opposite cyclic order (right to left) and ignoring each live entry whose previous live entry is two spaces to its left.} 
\label{fig:subslithers-case2}
\end{figure}

In constructing the slither from $(i,j)=(i,1)$ from subslithers, there are two cases to handle separately: (i) traversing between consecutive live entries within a row, and then (ii) leaving the final live entry in the row, in which case we prematurely reach 
$c\big((i,j)^+\big)$ before finishing the subslither and our algorithm will terminate. The subslither between consecutive live entries is completely determined by the sizes of the $0$-blocks.

Clearly, we can construct the slither of a scroll by starting at any live entry $(i,j)$ and successively computing subslithers; the only question is when to stop. Assume once again that $(i,j)=(i,1)$ is in the first column, and let $(i,j')$ be the last live entry in row $i$. It is straightforward to see that $j'<n$ and that the subslither from $(i,j')$ must start with a $D$. This step takes us to row $i+1$, and the subslither continues with $\E$s until either the entry $(i+1,n-1)$ or $(i+1,n)$ is reached. Whichever of these is live is the co-successor of $(i,1)$, so this takes us back to $c\big((i,j)^+\big)$  before we finish the subslither from $(i,j')$. Notice that an $L$-step from $(i,1)$ would take us to $(i+1,n-1)$, and an $S$-step would take us to $(i+1,n)$. This is stated formally as the following lemma.

\begin{lemma}\label{lem:partial-subslither}
Suppose $(i,1)\in\Live(\cals)$ and the last live entry $(i,j')$ in row $i$ is followed by exactly $z$ consecutive $0$s. If we start at $(i,j')$ and apply the initial segment of the subslither consisting of $D$ followed by $\left\lfloor\frac{z-1}{2} \right\rfloor$ instances of $\E$, then we end up at $c(i,1)$. Moreover, the step $(i,1)\to c(i,1)$ is of type $S$ if $z$ is odd and is of type $L$ if $z$ is even. $\hfill\Box$
\end{lemma}

Still assuming that $(i,j)=(i,1)$, Lemmas~\ref{lem:subslithers} and~\ref{lem:partial-subslither} completely characterize how to construct the slither of a scroll, given any length-$n$ substring that starts with a $1$. Specifically, we traverse the gaps of 0s between live entries from left to right and append each subslither ($\E$ or $D\E^rD$) as described by Lemma~\ref{lem:subslithers}. Then we append the partial subslither $D\E^r$ per Lemma~\ref{lem:partial-subslither}.

\begin{proposition}\label{prop:getslither}
To construct the slither from a row with $(i,1)\in\Live(\cals)$, traverse the $0$-blocks entirely contained in row $i$ from left to right. Do the following for each $0$-block (where $z$ is the size of the $0$-block):
\begin{itemize}
    \item If $z=1$, append $\E$.
    \item If $z>1$, append $D\E^{\left\lfloor z/2 \right\rfloor-1} D$.
\end{itemize}
Finally, append $D\E^{\left\lfloor(z-1)/2\right\rfloor}$, where $z\geq 1$ is the number of $0$s in row $i$ after the rightmost live entry in row $i$.
\end{proposition}

\begin{proof}
The first step is the result of iteratively appending subslithers, which, by Lemma~\ref{lem:subslithers}, are either $D\E^{\left\lfloor z/2-1 \right\rfloor} D$ or $\E$. The last step of appending $D\E^{\left\lfloor(z-1)/2\right\rfloor}$ is due to Lemma~\ref{lem:partial-subslither}.
\end{proof}

An example of the construction of a subslither from a row beginning with a live entry is shown in Figure~\ref{fig:one-line}. 
Here, we construct the subslither of the scroll shown in Figure~\ref{fig:subslithers-case2} from just the first row, using only the sizes of the $0$-blocks. That figure also provides a summary of how to construct the co-slither, which we will derive next.

\begin{figure}[!ht]
\begin{tikzpicture}
\draw[thick,decorate,decoration={brace,amplitude=6pt}]
        (.25,.25) -- (2.6,.25); 
        \node at (1.425,.75) {\color{Blue}$S$};
\draw[thick,decorate,decoration={brace,amplitude=6pt}]
        (3.4,.25) -- (4.35,.25);
        \node at (3.825,.75) {\color{orange}$S$};
\draw[thick,decorate,decoration={brace,amplitude=6pt}]
        (4.45,.25) -- (6.45,.25);
        \node at (5.45,.75) {\color{green}$L$};
\draw[thick,decorate,decoration={brace,amplitude=6pt}]
        (6.55,.25) -- (8.2,.25);
        \node at (7.45,.75) {\color{Red}$S$};
\node at (0,0) {\color{Red}\textbf{1}};
\node at (.35,0) {0};
\node at (.7,0) {0};
\node at (1.05,0) {0};
\node at (1.4,0) {0};
\node at (1.75,0) {0};
\node at (2.1,0) {0};
\node at (2.45,0) {\color{Blue}\textbf{1}};
\node at (2.8,0) {0};
\node at (3.15,0) {\color{Blue}\textbf{1}};
\node at (3.5,0) {0};
\node at (3.85,0) {0};
\node at (4.2,0) {\color{orange}\textbf{1}};
\node at (4.55,0) {0};
\node at (4.9,0) {0};
\node at (5.25,0) {0};
\node at (5.6,0) {0};
\node at (5.95,0) {0};
\node at (6.3,0) {\color{green}\textbf{1}};
\node at (6.65,0) {0};
\node at (7,0) {0};
\node at (7.35,0) {0};
\node at (7.7,0) {0};
\node at (8.05,0) {0};
\draw[thick,decorate,decoration={brace,amplitude=6pt,mirror}]
        (-.1,-.25) -- (2.25,-.25); 
        \node at (1.075,-.75) {\color{Red}$D\E\E D$};
\draw[thick,decorate,decoration={brace,amplitude=6pt,mirror}]
        (2.35,-.25) -- (2.95,-.25);
        \node at (2.65,-.75) {\color{Blue}$\E$};
\draw[thick,decorate,decoration={brace,amplitude=6pt,mirror}]
        (3.05,-.25) -- (4,-.25);
        \node at (3.4755,-.75) {\color{Blue}$DD$};
\draw[thick,decorate,decoration={brace,amplitude=6pt,mirror}]
        (4.1,-.25) -- (6.1,-.25);
        \node at (5.1,-.75) {\color{orange}$D\E D$};
\draw[thick,decorate,decoration={brace,amplitude=6pt,mirror}]
        (6.2,-.25) -- (8.2,-.25);
        \node at (7.1,-.75) {\color{green}$D\E\E$};
\node[anchor=west] at (-6,0.0) {\small Independent set (row in $\cals$):};
\node[anchor=west] at (-6,-.75) {\small Slither (read left-to-right):};
\node[anchor=west] at (-6,.65) {\small Co-slither (read right-to-left):};

\end{tikzpicture}
\caption{The construction of the slither and co-slither from Figure~\ref{fig:subslithers-case2}. The slither is constructed by iteratively appending subslithers, as described in Proposition~\ref{prop:getslither}. The co-slither is constructed from the parity of the gaps from right-to-left using Proposition~\ref{prop:getco-slither}.}\label{fig:one-line}
\end{figure}
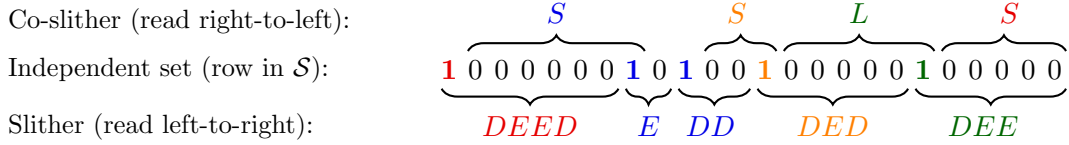

To construct the co-slither, start from $(i,j)=(i,1)$, and apply the co-successor function to reach either $(i+1,n)$ or $(i+1,n-1)$. Applying the inverse-successor function from here takes us to $(i,j')$, the last live entry in row $i$, which lies on $\Snake((i,j'))$. Applying the co-successor function from $(i,j')$ takes us to another snake, and we can apply $s^{-1}$ until we get back to row $i$ and repeat this process until we return to $\Snake(i,j)$. Note that there could be several choices at each step, because some snake might have at least two consecutive live entries in row $i$, like the blue snake does in Figure~\ref{fig:subslithers-case2}. However, this does not matter, because applying the co-successor function from any of these lands us on the same snake.

\begin{proposition}\label{prop:getco-slither}
To construct the co-slither from a row with $(i,1)\in\Live(\cals)$, start with an $L$ if it ends in an even number of $0$s and an $S$ otherwise. Next, for each gap of $z>1$ consecutive $0$s, going from right to left:
\begin{itemize}
    \item if $z$ is even, append $S$;
    \item if $z$ is odd, append $L$.
\end{itemize}
\end{proposition}

\begin{proof}
The first step is due to the characterization of $c(i,1)$ in Lemma~\ref{lem:partial-subslither}. Next, it is straightforward to see (e.g., in Figure~\ref{fig:subslithers-case2}) how for each live entry in row $i$ other than $(i,1)$, the step taking that live entry to its co-successor is of type $S$ if the number of $0$s preceding it is even and is of type $L$ if the number of $0$s preceding it is odd. 
\end{proof}

An example of using Proposition~\ref{prop:getco-slither} to construct the co-slither of the scroll shown in Figure~\ref{fig:subslithers-case2} from just the first line is shown in Figure~\ref{fig:one-line}.

\subsection{Characterizing ticker tapes}
Propositions~\ref{prop:getslither} and ~\ref{prop:getco-slither} can be used to construct all possible pairs of slithers and co-slithers for a given $n$. Given a ticker tape $\calx$, write $\beta_{\E}$ and $\beta_D$ for the number of $\E$s and $D$s in $\Slither(\calx)$ and $\alpha_S$ and $\alpha_L$ for the number of $S$s and $L$s in $\Coslither(\calx)$. Recall that the \emph{scale} of a ticker tape is defined in Definition~\ref{def:scale}. 

\begin{lemma}\label{lem:scale}
The scale of a ticker tape (or scroll) is
\begin{equation*}
\sigma=2\beta_{\E}+(n+1)\beta_D=(2n-1)\alpha_S+(2n-2)\alpha_L. 
\end{equation*}
\end{lemma}

\begin{proof}
From some fixed $k\in\Live(\calx)$, we can compute the scale in two ways: (i) by applying the successor function $\beta=\beta_{\E}+\beta_D$ times, or (ii) applying the co-successor function $\alpha=\alpha_S+\alpha_L$ times. In the first case, we advance $2\beta_{\E}+(n+1)\beta_D$ positions in the ticker tape, and in the latter case, we advance $(2n-1)\alpha_S+(2n-2)\alpha_L$ positions.
\end{proof}

\begin{corollary}\label{cor:2a+3b+4c=n+1}
For any ticker tape (or scroll), we have \[\beta_D = 2(\alpha_S+\alpha_L)-1\quad \text{and}\quad 
2\beta_{\E}+3\alpha_S+4\alpha_L=n+1.
\]
\end{corollary}

\begin{proof}
Without loss of generality, assume that $(0,1)\in\Live(\cals)$. There are $\alpha_S+\alpha_L$ letters in the co-slither: one for each gap of zeros between live entries and one more for the final string of 0s. Each gap contributes exactly two $D$s, except the final string, which contributes one. It follows that the slither contains $2(\alpha_S+\alpha_L)-1$ instances of $D$. The result now follows from substituting $\beta_D=2(\alpha_S+\alpha_L)-1$ into the equation in Lemma~\ref{lem:scale} and simplifying. 
\end{proof}

The equation $2\beta_{\E}+3\alpha_S+4\alpha_L=n+1$ from Corollary~\ref{cor:2a+3b+4c=n+1} gives a necessary condition for the slithers and co-slithers that exist for a given $n$. Theorem~\ref{thm:2a+3b+4c=n+1} will show that this condition is also sufficient. In particular, any solution to this equation with $\beta_{\E}, \alpha_S, \alpha_L\ge 0$ and $\alpha_S+\alpha_L > 0$ gives a set of potential slithers and co-slithers that only differ by rearrangement.\footnote{Note that if we had $\alpha_S=\alpha_L = 0$, then our co-slither would be empty, which is impossible.} From each of these slither and co-slither combinations, we can construct a scroll. Without loss of generality, we will assume that $(0,1)$ is live. 

\begin{definition}\label{def:2a+3b+4c=n+1}
Fix a positive integer $n$. Let $\alpha_S$, $\alpha_L$, $\beta_{\E}$, and $\beta_D=2(\alpha_S+\alpha_L)-1$ be nonnegative integers. Suppose $W_s$ (the ``slither'') is a word over the alphabet $\{D,\E\}$ with $\beta_D$ instances of $D$ and $\beta_{\E}$ instances of $\E$, and suppose $W_c$ (the ``co-slither'') is a word over the alphabet $\{S,L\}$ with $\alpha_S$ instances of $S$ and $\alpha_L$ instances of $L$. We say the pair $(W_s,W_c)$ is \dfn{feasible} if $2\beta_{\E}+3\alpha_S+4\alpha_L=n+1$.  
\end{definition}

Built into the definition of a feasible pair is the assumption that $\alpha_S+\alpha_L>0$, and hence that $\beta_D>0$. In particular, both words must be nonempty.

\begin{figure}
\begin{center}
\begin{tabular}{ c c c c c c}
 $\beta_{\E}$  &  $\alpha_S$  & $\alpha_L$&  $\beta_D =2(\alpha_S + \alpha_L)-1$ & Slither & Co-slither\\ 
 \hline
       5 & 0 & 1 & 1 & $\E\E\E\E\E D$ & $L$\\
       3 & 0 & 2 & 3 & $\E\E\E DDD$ & $LL$\\
       3 & 0 & 2 & 3 & $\E\E D\E DD$ & $LL$\\
       3 & 0 & 2 & 3 & $\E\E DD\E D$ & $LL$\\
       3 & 0 & 2 & 3 & $\E D\E D\E D$ & $LL$\\
       1 & 0 & 3 & 5 & $\E DDDDD$ & $LLL$\\
       4 & 2 & 0 & 3 & $\E\E\E\E DDD$ & $SS$\\
       4 & 2 & 0 & 3 & $\E\E\E D\E DD$ & $SS$\\
       4 & 2 & 0 & 3 & $\E\E\E DD\E D$ & $SS$\\
       4 & 2 & 0 & 3 & $\E\E D\E\E DD$ & $SS$\\
       4 & 2 & 0 & 3 & $\E\E D\E D\E D$ & $SS$\\
       2 & 2 & 1 & 5 & $\E\E DDDDD$ & $SSL$\\
       2 & 2 & 1 & 5 & $\E D\E DDDD$ & $SSL$\\
       2 & 2 & 1 & 5 & $\E DD\E DDD$ & $SSL$\\
       0 & 2 & 2 & 7 & $DDDDDDD$ & $SSLL$\\
       0 & 2 & 2 & 7 & $DDDDDDD$ & $SLSL$\\
       1 & 4 & 0 & 7 & $EDDDDDDD$ & $SSSS$\\
\end{tabular}
\caption{This table classifies all possible slithers and co-slithers (up to cyclic shift) for ticker tapes on $n=13$ vertices.  }\label{fig:SlitherCoslitherClassification}
\end{center}
\end{figure}

\begin{thm}\label{thm:2a+3b+4c=n+1}
Every feasible pair defines a ticker tape.
\end{thm}

\begin{proof}
We can explicitly construct the ticker tape using Propositions~\ref{prop:getslither} and~\ref{prop:getco-slither}. Begin with a live entry in $(0,1)$, and read off the slither. If we read an $\E$ and there have been an even number of $D$s so far, we add a single 0 and then a live entry. If we read a $D$ that is not the final $D$, then let $t$ be the number of $\E$s between this $D$ and the next $D$. We add either $2t + 2$ or $2t + 3$ zeros and then another live entry. To determine which, we look at the next entry in the co-slither, reading right to left. If the entry is $S$, we add $2t+2$ zeros, while if it is $L$, we add $2t + 3$ zeros. When we reach the final $D$ in the slither, we simply fill out the rest of the row with zeros. 

By the condition that $\beta_D = 2(\alpha_S + \alpha_L) - 1$, we know that upon reaching the rightmost $D$ in the slither, we also reach the leftmost letter in the co-slither. Let $t$ be the number of $\E$s at the end of the slither. We claim that the number of $0$s at the end of the first row of the scroll is precisely $2t+1$ if the leftmost entry in the co-slither is $S$ and is $2t+2$ if it is $L$. This follows from the fact that $2\beta_{\E} + 3\alpha_S + 4\alpha_L = n+1$. In particular, when writing entries in the first row, we moved right two positions for each $T$, three positions for each $S$, and four positions for each $L$. 

Now that we have constructed the first row, the remainder of the table is fully determined. By construction, the set of live entries forms an independent set, and the theorem follows. 
\end{proof}

By construction, each feasible pair corresponds to a unique ticker tape. The example in Figures~\ref{fig:subslithers-case2} and \ref{fig:one-line} corresponds to the solution 
$\beta_{\E}=6$, $\alpha_S=3$, $\alpha_L=1$ to the equation $2\beta_{\E} + 3\alpha_S + 4\alpha_L=25$.

\begin{remark}\label{rem:genfunc}
Using generating functions, it is straightforward to calculate the number of solutions to the equation $2\beta_E + 3\alpha_S + 4\alpha_L = n + 1$  over the nonnegative integers that satisfy $\alpha_S+\alpha_L>0 $. In particular, this quantity is given by the coefficient of $x^{n+1}$ in the generating function
\[
\frac{1}{1-x^2}\left(\frac{1}{\left(1-x^3\right)\left(1-x^4\right)}-1\right).
\]
Calculating the total number of feasible pairs is more complicated because each solution corresponds to a collection of slithers and co-slithers made up of the same multiset of letters. 
\end{remark}

\begin{example}Figure~\ref{fig:SlitherCoslitherClassification} gives a list of all of the ticker tapes for $n=13$, which we computed using Theorem~\ref{thm:2a+3b+4c=n+1}. There are a total of $17$ ticker tapes (up to cyclic shift). Furthermore, one can see that there are $7$ possible quadruples $(\alpha_S,\alpha_L,\beta_E,\beta_D)$. This is the coefficient of $x^{14}$ in the generating function from Remark~\ref{rem:genfunc}. 
\end{example}

\section{Dynamics and actions on finite quotient spaces}\label{sec:tables}

\subsection{Orbit tables and ouroboroi}\label{subsec:tables}

Thus far, we have viewed the dynamics generated by toggling independent sets using infinite scrolls and ticker tapes. However, sometimes it will be convenient to restrict our attention to a repeating sequence of rows in a scroll. If we identify two identical rows by a quotient map of the scroll (a cylinder) to get a torus, the snakes and co-snakes ``wrap around'' from bottom to top. Inspired by the ancient symbol of a snake swallowing its tail, as shown on the left in Figure~\ref{fig:ouroboros}, we will call such a finite circular snake an \emph{ouroboros}. 

\colorlet{color1}{Red}
\colorlet{color2}{Red}
\colorlet{color3}{orange}
\colorlet{color4}{green!95!black}
\colorlet{color5}{orange}
\colorlet{color6}{green!95!black}
\colorlet{color7}{orange}
\colorlet{color8}{green!95!black}
\begin{figure}[!ht]
\begin{tikzpicture}
\begin{scope}
\setlength{\tabcolsep}{2.5pt}
\renewcommand{\arraystretch}{.4}
\node at (-8.5,-1.85){\includegraphics[width=.45\textwidth]{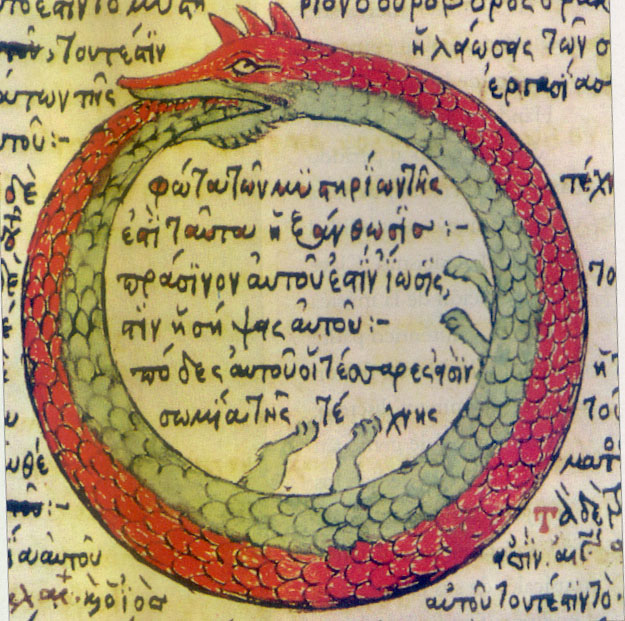}};
\node at (0,0) {
\begin{tabular}{cccccccccccc}
$x$ & $v_1$ & $v_2$ & $v_3$ & $v_4$ & $v_5$ & $v_6$ &
$v_7$ & $v_8$ & $v_9$ & $v_{1\!0}$ & $v_{1\!1}$ \Bstrut
\\\hline
$x^{(0)}$ & \0 & \0 & \0 & \0 & {\bf\color{color1}1} & \0 & {\bf\color{color1}1} & \0 & \0 & \0 & \0 \\
$x^{(1)}$ & {\bf\color{color2}1} & \0 & {\bf\color{color2}1} & \0 & \0 & \0 & \0 & {\bf\color{color1}1} & \0 & {\bf\color{color1}1} & \0 \\
$x^{(2)}$ & \0 & \0 & \0 & {\bf\color{color2}1} & \0 & {\bf\color{color2}1} & \0 & \0 & \0 & \0 & {\bf\color{color1}1} \\
$x^{(3)}$ & \0 & {\bf\color{color1}1} & \0 & \0 & \0 & \0 & {\bf\color{color2}1} & \0 & {\bf\color{color2}1} & \0 & \0 \\
$x^{(4)}$ & \0 & \0 & {\bf\color{color1}1} & \0 & {\bf\color{color1}1} & \0 & \0 & \0 & \0 & {\bf\color{color2}1} & \0 \\
$x^{(5)}$ & {\bf\color{color2}1} & \0 & \0 & \0 & \0 & {\bf\color{color1}1} & \0 & {\bf\color{color1}1} & \0 & \0 & \0 \\
$x^{(6)}$ & \0 & {\bf\color{color2}1} & \0 & {\bf\color{color2}1} & \0 & \0 & \0 & \0 & {\bf\color{color1}1} & \0 & {\bf\color{color1}1}\Bstrut \\ \hline \Tstrut
\end{tabular}};
\node at (0,-4) {
\begin{tabular}{ccccccccccccc}
$x$ & $v_1$ & $v_2$ & $v_3$ & $v_4$ & $v_5$ & $v_6$ &
$v_7$ & $v_8$ & $v_9$ & $v_{1\!0}$ & $v_{1\!1}$ & \Bstrut
\\\hline
$x^{(0)}$ & \0 & \0 & \0 & \0 & {\bf\color{color5}1} & \0 & {\bf\color{color6}1} & \0 & \0 & \0 & \0 \\
$x^{(1)}$ & {\bf\color{color3}1} & \0 & {\bf\color{color4}1} & \0 & \0 & \0 & \0 & {\bf\color{color7}1} & \0 & {\bf\color{color8}1} & \0 \\
$x^{(2)}$ & \0 & \0 & \0 & {\bf\color{color5}1} & \0 & {\bf\color{color6}1} & \0 & \0 & \0 & \0 & {\bf\color{color3}1} \\
$x^{(3)}$ & \0 & {\bf\color{color4}1} & \0 & \0 & \0 & \0 & {\bf\color{color7}1} & \0 & {\bf\color{color8}1} & \0 & \0 \\
$x^{(4)}$ & \0 & \0 & {\bf\color{color5}1} & \0 & {\bf\color{color6}1} & \0 & \0 & \0 & \0 & {\bf\color{color3}1} & \0 \\
$x^{(5)}$ & {\bf\color{color4}1} & \0 & \0 & \0 & \0 & {\bf\color{color7}1} & \0 & {\bf\color{color8}1} & \0 & \0 & \0 \\
$x^{(6)}$ & \0 & {\bf\color{color5}1} & \0 & {\bf\color{color6}1} & \0 & \0 & \0 & \0 & {\bf\color{color3}1} & \0 & {\bf\color{color4}1}\Bstrut \\ \hline \Tstrut
\end{tabular}};
\end{scope}
\end{tikzpicture}
\caption{On the left is a drawing of an ouroboros from a 1478 book on medieval alchemy by Theodoros Pelecanos; this image is from Wikipedia. On the right is the fundamental orbit table $\calt_1$ from our running example in Figure~\ref{fig:example}. When we allow snakes and co-snakes to wrap from bottom to top, the two snakes merge into one ouroboros  with slither $\bar{D}\,\bar{\E}$ (top), and the six co-snakes merge into two co-ouroboroi, with co-slither $\bar{S}$  (bottom).}\label{fig:ouroboros}
\end{figure}

Let $x\in\ff_2^n$, and suppose $r$ is a positive integer such that $x^{(r)}=x^{(0)}=x$. Then $r$ must be a multiple of the period $T(\cals)$ of the scroll defined by $x$ (as defined in Section~\ref{subsec:scrolls}). Hence, $r=\omega T(\cals)$ for some positive integer $\omega$ that we call a \dfn{frequency}. Define the \dfn{$\omega$-fold orbit table} of $x$ to be the $r\times n$ table whose rows are $x^{(0)},\ldots,x^{(r-1)}$. We denote this table by $\calt_\omega=\Table_\omega(x)$ or as $\calt=\Table(x)$, depending on whether $\omega$ is understood.

It will at times be useful to work with a finite version of the ticker tape. If $r=\omega T(\cals)$ as above, then define the \dfn{$\omega$-fold orbit vector} of $x$ to be the length-$rn$ subsequence of the ticker tape that has $x$ as an initial sequence---the result of reading the $\omega$-fold orbit table across each column, downward row-by-row. We denote this as
\begin{equation}\label{eqn:tape}
\calv_\omega=\Vector_\omega(x)=\big(X_{0,1},\dots,X_{0,n},X_{1,1},\dots,X_{1,n},\dots,X_{r-1,1},\dots,X_{r-1,n}\big)\in\ff_2^{rn},
\end{equation}
or $\calv=\Vector(x)$ if $\omega$ is understood or unimportant to the context.
We refer to the $1$-fold orbit table (resp., vector) as the \dfn{fundamental orbit table (resp., vector)}.

If $\calt$ is an orbit table and $\calv$ an orbit vector, we define their sets of live entries to be
\[
\Live(\calt)=\big\{(i,j)\in\zz_r\times\zz_n\mid X_{i,j}=1\big\},\qquad
\Live(\calv)=\big\{k\in\zz_{rn}\mid X_k=1\big\}.
\]
Though it makes no difference either way, we will continue with the convention that the columns are numbered $1,\dots,n$ and the rows are numbered $0,\dots, r-1$. As such, we harmlessly take $\zz_n=\{1,\dots,n\}$ and $\zz_r=\{0,\dots,r-1\}$ in the orbit table and $\zz_{rn}=\{1,\dots,rn\}$ in the orbit vector.

The live entries in an orbit table are simply the images of the live entries in the corresponding scroll under the natural quotient map $\p_\omega\colon\Live(\cals)\to\Live(\calt_\omega)$ that reduces the first coordinate of each entry modulo $r$. Under this map, the successor and co-successor functions descend to bijections on $\Live(\calt_\omega)$ that we call the \dfn{$\omega$-successor function} $\bar{s}_\omega$ and \dfn{$\omega$-co-successor function} $\bar{c}_\omega$. The relationship between the successor and its $\omega$-counterpart is illustrated by the following commutative diagrams. 
\[
\xymatrix{
\Live(\cals)\ar[d]_{\p_\omega}\ar[r]^s & \Live(\cals)\ar[d]^{\p_\omega} \\
\Live(\calt_\omega)\ar[r]^{\bar{s}_\omega} & \Live(\calt_\omega)}\hspace{15mm}
\xymatrix{
(i+kr,j)\ar@{|->}[d]_{\p_\omega}\ar@{|->}^s[r] & s(i+kr,j) \ar^{\p_\omega}@{|->}[d] \\
(i,j)\ar^{\bar{s}_\omega}@{|->}[r] & \bar{s}_\omega(i,j)}
\]
Naturally, there is an analogous diagram relating $c$ and $\bar{c}_\omega$. The functions $\bar{s}_\omega$ and $\bar{c}_\omega$ generate a finite abelian group $G(\calt_\omega):=\left<\bar{s}_\omega,\bar{c}_\omega\right>$ that we call the \dfn{ouroboros group} of $\calt_\omega$, or the \dfn{$\omega$-fold ouroboros group} of $\cals$. Since $\p_\omega$ is a topological covering map, there is an induced homomorphism $\p_\omega^*\colon G(\cals)\to G(\calt_\omega)$ sending $s\mapsto\bar{s}_\omega$ and $c\mapsto\bar{c}_\omega$. 
The ouroboros group is the quotient
\[
G(\calt_\omega)\cong G(\cals)/\ker \p_\omega^*,
\]
and it acts simply transitively on $\Live(\cals)/\ker \p_\omega$, which can be canonically identified with $\Live(\calt_\omega)$. We get a bijective correspondence between the orbits under $\bar{s}_\omega$ and $\bar{c}_\omega$ and the cosets of $\left<\bar{s}_\omega\right>$ and $\left<\bar{c}_\omega\right>$. These are  the images of the snakes and the co-snakes under the quotient map $\p_\omega$, so we call them \emph{ouroboroi} and \emph{co-ouroboroi}, respectively. 

\begin{definition}
Given a live entry $(i,j)$ in an orbit table $\calt_\omega$, the \dfn{ouroboros} and \dfn{co-ouroboros} containing it are the sets 
\[
\Ouroboros_\omega(i,j)=\big\{\bar{s}^{k}_\omega(i,j)\mid k\in\zz\big\},\qquad\qquad 
\Coouroboros_\omega(i,j)=\big\{\bar{c}^{k}_\omega(i,j)\mid k\in\zz\big\}.
\]
\end{definition} 

Throughout the rest of this paper, we will continue to assume that a scroll $\cals$ has $\alpha$ snakes and $\beta$ co-snakes. Recall that we denote the number of $S$s and $L$s in any co-slither by $\alpha_S$ and $\alpha_L$, and the number of $D$s and $\E$s in any slither by $\beta_D$ and $\beta_{\E}$, respectively. Naturally, we have
\[
\alpha=\alpha_S+\alpha_L\quad\text{and}\quad\beta=\beta_D+\beta_{\E}.
\]
We say the $\omega$-fold orbit table $\calt=\calt_\omega$ has $\bar{\alpha}_\omega$ ouroboroi and $\bar{\beta}_\omega$ co-ouroboroi. (If $\omega$ is clear from the context, which it usually will be, then we will typically drop it as a subscript.) Similarly, we will often write $\bar{s}$ and $\bar{c}$ rather than $\bar{s}_\omega$ and $\bar{c}_\omega$ because $\omega$ will usually be unambiguous when we speak of these functions.

\begin{theorem}\label{thm:ouroboros-torsor}
For any orbit table $\calt$, the set $\Live(\calt)$ is a torsor of the ouroboros group, which has presentation
\[
G(\calt)=\left<\bar{s},\bar{c}\;\left|\; \bar{s}\,\bar{c}=\bar{c}\,\bar{s},\; \bar{s}^\beta=\bar{c}^\alpha,\;\bar{s}^{\,\eta/\bar{\alpha}}=\bar{c}^{\,\eta/\bar{\beta}}=1\right.\right>
\cong \zz_{\bar\alpha}\times\zz_{\eta/\bar\alpha}
\cong \zz_{\bar\beta}\times\zz_{\eta/\bar\beta},
\]
where $\eta$ is the number of live entries.
\end{theorem}

\begin{proof}
We have already established the first statement. It follows that there are bijective correspondences between ouroboroi and cosets of $\<\bar{s}\>$ and between co-ouroboroi and cosets of $\<\bar{c}\>$. Since $G(\calt)=\<\bar{s},\bar{c}\>\cong G(\cals)/\ker \p^*$ is a finite abelian group of order $\eta$, the first two relations hold, and we have
\begin{equation}\label{eqn:ouroboroi-cosets}
\bar{\alpha}=[G(\calt):\<\bar{s}\>]=\eta/\bar\beta\quad\text{and}\quad
\bar{\beta}=[G(\calt):\<\bar{c}\>]=\eta/\bar\alpha.
\end{equation}
In other words, $G(\calt)/\<\bar{s}\>\cong\zz_{\bar\alpha}$ and  $G(\calt)/\<\bar{c}\>\cong\zz_{\bar\beta}$. The result now follows. 
\end{proof}

\colorlet{color1}{Red}
\colorlet{color2}{Blue}
\colorlet{color3}{Red}
\colorlet{color4}{Blue}
\colorlet{color5}{purple}
\colorlet{color6}{Green}
\colorlet{color7}{orange}
\colorlet{color8}{amethyst}

\colorlet{color1}{Red}
\colorlet{color2}{Blue}
\colorlet{color3}{orange}
\colorlet{color4}{green!95!black}
\colorlet{color5}{orange}
\colorlet{color6}{green!95!black}
\colorlet{color7}{orange}
\colorlet{color8}{green!95!black}
\begin{figure}[!ht]
\begin{tikzpicture}[scale=.8]
\tikzstyle{every node}=[font=\small,anchor=south]
\begin{scope}
\setlength{\tabcolsep}{2.5pt}
\renewcommand{\arraystretch}{.4}
\node at (0,0) {
\begin{tabular}{cccccccccccc}
$x$ & $v_1$ & $v_2$ & $v_3$ & $v_4$ & $v_5$ & $v_6$ &
$v_7$ & $v_8$ & $v_9$ & $v_{1\!0}$ & $v_{1\!1}$ \Bstrut
\\\hline
$x^{(0)}$ & \0 & \0 & \0 & \0 & {\bf\color{color1}1} & \0 & {\bf\color{color1}1} & \0 & \0 & \0 & \0 \\
$x^{(1)}$ & {\bf\color{color2}1} & \0 & {\bf\color{color2}1} & \0 & \0 & \0 & \0 & {\bf\color{color1}1} & \0 & {\bf\color{color1}1} & \0 \\
$x^{(2)}$ & \0 & \0 & \0 & {\bf\color{color2}1} & \0 & {\bf\color{color2}1} & \0 & \0 & \0 & \0 & {\bf\color{color1}1} \\
$x^{(3)}$ & \0 & {\bf\color{color1}1} & \0 & \0 & \0 & \0 & {\bf\color{color2}1} & \0 & {\bf\color{color2}1} & \0 & \0 \\
$x^{(4)}$ & \0 & \0 & {\bf\color{color1}1} & \0 & {\bf\color{color1}1} & \0 & \0 & \0 & \0 & {\bf\color{color2}1} & \0 \\
$x^{(5)}$ & {\bf\color{color2}1} & \0 & \0 & \0 & \0 & {\bf\color{color1}1} & \0 & {\bf\color{color1}1} & \0 & \0 & \0 \\
$x^{(6)}$ & \0 & {\bf\color{color2}1} & \0 & {\bf\color{color2}1} & \0 & \0 & \0 & \0 & {\bf\color{color1}1} & \0 & {\bf\color{color1}1} \\ 
$x^{(7)}$ & \0 & \0 & \0 & \0 & {\bf\color{color2}1} & \0 & {\bf\color{color2}1} & \0 & \0 & \0 & \0 \\
$x^{(8)}$ & {\bf\color{color1}1} & \0 & {\bf\color{color1}1} & \0 & \0 & \0 & \0 & {\bf\color{color2}1} & \0 & {\bf\color{color2}1} & \0 \\
$x^{(9)}$ & \0 & \0 & \0 & {\bf\color{color1}1} & \0 & {\bf\color{color1}1} & \0 & \0 & \0 & \0 & {\bf\color{color2}1} \\
$x^{(10)}$ & \0 & {\bf\color{color2}1} & \0 & \0 & \0 & \0 & {\bf\color{color1}1} & \0 & {\bf\color{color1}1} & \0 & \0 \\
$x^{(11)}$ & \0 & \0 & {\bf\color{color2}1} & \0 & {\bf\color{color2}1} & \0 & \0 & \0 & \0 & {\bf\color{color1}1} & \0 \\
$x^{(12)}$ & {\bf\color{color1}1} & \0 & \0 & \0 & \0 & {\bf\color{color2}1} & \0 & {\bf\color{color2}1} & \0 & \0 & \0 \\
$x^{(13)}$ & \0 & {\bf\color{color1}1} & \0 & {\bf\color{color1}1} & \0 & \0 & \0 & \0 & {\bf\color{color2}1} & \0 & {\bf\color{color2}1}\Bstrut \\ \hline \Tstrut
\end{tabular}};
\node at (10,0) {
\begin{tabular}{ccccccccccccc}
$x$ & $v_1$ & $v_2$ & $v_3$ & $v_4$ & $v_5$ & $v_6$ &
$v_7$ & $v_8$ & $v_9$ & $v_{1\!0}$ & $v_{1\!1}$ & \Bstrut
\\\hline
$x^{(0)}$ & \0 & \0 & \0 & \0 & {\bf\color{color5}1} & \0 & {\bf\color{color6}1} & \0 & \0 & \0 & \0 \\
$x^{(1)}$ & {\bf\color{color3}1} & \0 & {\bf\color{color4}1} & \0 & \0 & \0 & \0 & {\bf\color{color7}1} & \0 & {\bf\color{color8}1} & \0 \\
$x^{(2)}$ & \0 & \0 & \0 & {\bf\color{color5}1} & \0 & {\bf\color{color6}1} & \0 & \0 & \0 & \0 & {\bf\color{color3}1} \\
$x^{(3)}$ & \0 & {\bf\color{color4}1} & \0 & \0 & \0 & \0 & {\bf\color{color7}1} & \0 & {\bf\color{color8}1} & \0 & \0 \\
$x^{(4)}$ & \0 & \0 & {\bf\color{color5}1} & \0 & {\bf\color{color6}1} & \0 & \0 & \0 & \0 & {\bf\color{color3}1} & \0 \\
$x^{(5)}$ & {\bf\color{color4}1} & \0 & \0 & \0 & \0 & {\bf\color{color7}1} & \0 & {\bf\color{color8}1} & \0 & \0 & \0 \\
$x^{(6)}$ & \0 & {\bf\color{color5}1} & \0 & {\bf\color{color6}1} & \0 & \0 & \0 & \0 & {\bf\color{color3}1} & \0 & {\bf\color{color4}1} \\
$x^{(7)}$ & \0 & \0 & \0 & \0 & {\bf\color{color5}1} & \0 & {\bf\color{color6}1} & \0 & \0 & \0 & \0 \\
$x^{(8)}$ & {\bf\color{color3}1} & \0 & {\bf\color{color4}1} & \0 & \0 & \0 & \0 & {\bf\color{color7}1} & \0 & {\bf\color{color8}1} & \0 \\
$x^{(9)}$ & \0 & \0 & \0 & {\bf\color{color5}1} & \0 & {\bf\color{color6}1} & \0 & \0 & \0 & \0 & {\bf\color{color3}1} \\
$x^{(10)}$ & \0 & {\bf\color{color4}1} & \0 & \0 & \0 & \0 & {\bf\color{color7}1} & \0 & {\bf\color{color8}1} & \0 & \0 \\
$x^{(11)}$ & \0 & \0 & {\bf\color{color5}1} & \0 & {\bf\color{color6}1} & \0 & \0 & \0 & \0 & {\bf\color{color3}1} & \0 \\
$x^{(12)}$ & {\bf\color{color4}1} & \0 & \0 & \0 & \0 & {\bf\color{color7}1} & \0 & {\bf\color{color8}1} & \0 & \0 & \0 \\
$x^{(13)}$ & \0 & {\bf\color{color5}1} & \0 & {\bf\color{color6}1} & \0 & \0 & \0 & \0 & {\bf\color{color3}1} & \0 & {\bf\color{color4}1}\Bstrut \\ \hline \Tstrut
\end{tabular}};
\end{scope}
\end{tikzpicture}
\caption{In the $2$-fold orbit table $\calt_2$ from our running example in Figure~\ref{fig:example}, there are $\bar\alpha_2=2$ ouroboroi with slither $\bar{D}\,\bar{\E}$ and $\bar\beta_2=2$ co-ouroboroi with co-slither $\bar{S}^2$.}\label{fig:2-fold-table}
\end{figure}

\begin{definition}
The \dfn{(co-)ouroboros degree} of an orbit table $\calt_\omega$ is the number of (co-)snakes in the $\p_\omega$-preimage of each (co-)ouroboros. We denote these as 
\[
\deg(\p_\omega):=\frac{[G(\cals):\<s\>]}{[G(\calt_\omega):\<\bar{s}_\omega\>]}=\alpha/\bar\alpha_\omega
\quad\text{and}\quad
\codeg(\p_\omega):=\frac{[G(\cals):\<c\>]}{[G(\calt_\omega):\<\bar{c}_\omega\>]}=\beta/\bar\beta_\omega.
\]
We call $\deg(\p_1)$ and $\codeg(\p_1)$ the \dfn{fundamental ouroboros degree} and \dfn{fundamental co-ouroboros degree}, respectively.
\end{definition}

Returning back to our running example, the fundamental orbit table (i.e., $\omega=1$) is shown in Figure~\ref{fig:ouroboros}. The $\alpha=2$ snakes in $\cals$ merge into $\bar\alpha_1=1$ ouroboros, and the $\beta=6$ co-snakes merge into $\bar\beta_1=2$ co-ouroboroi. Therefore, the fundamental ouroboros and co-ouroboros degrees are
\[
\deg(\p_1)=\alpha/\bar\alpha_1=2/1=2\quad\text{and}\quad
\codeg(\p_1)=\beta/\bar\beta_1=6/2=3.
\]
In contrast, in the $2$-fold orbit table, shown in Figure~\ref{fig:2-fold-table}, the $\alpha=2$ snakes in $\cals$ remain $\bar\alpha_2=2$ separate ouroboroi, and the $\beta=6$ co-snakes become $\bar\beta_2=2$ co-ouroboroi. The ouroboros and co-ouroboros degrees are thus
\[
\deg(\p_2)=\alpha/\bar\alpha_2=2/2=1\quad\text{and}\quad \codeg(\p_2)=\beta/\bar\beta_2=6/2=3. 
\]

Slithers and co-slithers naturally descend to orbit tables via the quotient $\p_\omega\colon\Live(\cals)\to\Live(\calt_\omega)$. 
The slither of $\cals$ is a length-$\beta$ sequence of $D$s and $\E$s, and it defines a cyclic ordering $\<c\>,s\!\<c\>,\dots,s^{\beta-1}\!\<c\>$ of co-snakes. If we apply the quotient map $\p_\omega^*\colon G(\cals)\to G(\calt_\omega)$, then we get a cyclic ordering of the $\bar{\beta}_\omega$ co-ouroboroi. Each co-ouroboros appears in the sequence $\<\bar{c}_\omega\>,\bar{s}_\omega\!\<\bar{c}_\omega\>,\dots,\bar{s}_\omega^{\beta-1}\!\<\bar{c}_\omega\>$ exactly $\codeg(\p_\omega)=\beta/\bar{\beta}_\omega$ times, and this must be a divisor of the co-degree of $\cals$ (the exponent that appears in the slither).
%\footnote{This should explain our choice of using the word ``degree'' rather than ``exponent'' in Definition~\ref{def:exponent}.}
Define the \dfn{slither} of $\calt_\omega$, denoted $\Slither(\calt_\omega)$, to be any length-$\bar{\beta}$ subsequence of a slither of $\cals$ (so $\Slither(\calt_\omega)$ is defined up to cyclic shift). We will also refer to $\Slither(\calt_\omega)$ as the \dfn{$\omega$-slither} of $\cals$.

The preceding notions all have straightforward analogues for co-slithers. More precisely, a co-slither of $\cals$ is a length-$\alpha$ sequence of $S$s and $L$s that defines a cyclic ordering $\<s\>,c\<s\>,\dots,c^{\alpha-1}\<s\>$ of snakes. Via the quotient map $\p_\omega^*\colon G(\cals)\to G(\calt_\omega)$, we get a cyclic ordering of the $\bar{\alpha}_\omega$ ouroboroi. Each ouroboros appears in the sequence 
$\<\bar{s}\>,\bar{c}_\omega\<\bar{s}\>,\dots,\bar{c}_\omega^{\alpha-1}\<\bar{s}\>$ exactly $\deg(\p_\omega)=\alpha/\bar{\alpha}_\omega$ times, and this must be a divisor of the degree of $\cals$ (the exponent that appears in the co-slither). We define the \dfn{co-slither} of $\calt_\omega$, denoted $\Coslither(\calt_\omega)$, to be any length-$\bar{\alpha}$ subsequence (defined up to cyclic shift) of a co-slither of $\cals$; we also call this the \dfn{$\omega$-co-slither} of $\cals$. 
The preceding two paragraphs have established the following. 

\begin{lemma}\label{lem:slither(T)}
For any scroll $\cals$, we have
\[
\big(\Slither(\calt_\omega)\big)^{\codeg(\p_\omega)}=\Slither(\cals)\quad\text{\emph{and}}\quad
\big(\Coslither(\calt_\omega)\big)^{\deg(\p_\omega)}=\Coslither(\cals).
\]
\end{lemma}

We will refer to the (co-)slither of the fundamental orbit table (i.e., $\omega=1$) as the \dfn{fundamental (co-)slither}. To emphasize that we are taking the slither in an orbit table rather than in the scroll, we will sometimes write $\bar{\E}$ and $\bar{D}$ rather than $\E$ and $D$, and we will similarly use $\bar{S}$ and $\bar{T}$ in co-slithers. 

Let us return to our running example from Figure~\ref{fig:example} and its fundamental and $2$-fold orbit tables shown in Figures~\ref{fig:ouroboros} and \ref{fig:2-fold-table}, respectively. The length of the fundamental slither $\bar{D}\,\bar{\E}$ is  $\bar\beta_1=\beta/\codeg(\p_1)$, the number of co-ouroboroi, and the length of the fundamental co-slither $\bar{S}$ is $\bar\alpha_1=\alpha/\deg(\p_1)$, the number of ouroboroi. As guaranteed by Lemma~\ref{lem:slither(T)}, the (co-)slithers of $\cals$ and $\calt$ are related by
\[
(D\E)^{\codeg(\p_1)}=(D\E)^3\quad\text{and}\quad
S^{\deg(\p_1)}=S^2.
\]

The $2$-fold orbit table of our running example, shown in Figure~\ref{fig:2-fold-table}, has two ouroboroi with slither $\bar{D}\,\bar{\E}$ and two co-ouroboroi with co-slither $\bar{S}^2$. The ouroboros degree is thus $\deg(\p_2)=2/2=1$, and the co-ouroboros degree is $\codeg(\p_2)=6/2=3$.
As predicted by Lemma~\ref{lem:slither(T)}, we have
\[
(D\E)^{\codeg(\p_2)}=(D\E)^3\quad\text{and}\quad
\left(S^2\right)^{\deg(\p_2)}=S^2.
\]

\subsection{Swallow and co-swallow functions}\label{subsec:swallow}

In this subsection, we will look at how the snakes and co-snakes merge in various $\omega$-fold orbit tables under the quotient maps $\p_\omega$. This will allow us to derive formulas relating $\bar{\alpha}_\omega$ to $\bar{\alpha}=\bar{\alpha}_1$ and likewise relating $\bar{\beta}_\omega$ to $\bar{\beta}=\bar{\beta}_1$. 

We will write $\ss(\cals)$ for the set of snakes of $\cals$ and $\cc(\cals)$ for the set of co-snakes of $\cals$. Given $\mathfrak{s}\in \ss(\cals)$, define its \dfn{tail} to be the smallest positive coordinate $t$ of the ticker tape that is contained in $\mathfrak{s}$. It is easy to translate this back into orbit table notation if desired, and we will denote this as $\Tail(\mathfrak{s})$, regardless of the setting. Next, for any $\omega$, with $r=\omega T(\cals)$, define the \dfn{$\omega$-head} of $\mathfrak{s}$ to be the largest coordinate $h\in[rn]$ of the orbit vector $\calv_\omega$ contained in $\mathfrak{s}$. Applying the $\omega$-successor function from the $\omega$-head wraps past the end of the orbit vector to back near the beginning, landing on some live entry whose $\p_\omega$-preimage lies in some snake $\mathfrak{s}'\in\ss(\cals)$, possibly $\mathfrak{s}$ itself. This defines a bijection on the snakes of $\cals$, and motivated by the idea of an ouroboros swallowing its tail, we will call this bijection the \dfn{$\omega$-swallow function}:
\[
\Swallow_\omega\colon\ss(\cals)\longto\ss(\cals).
\]
As a permutation, $\Swallow_\omega$ contains $\bar\alpha_\omega$ disjoint cycles, each containing exactly $\deg(\p_\omega)$ snakes. The next result guarantees that if the snakes of $\cals$ are canonically cyclically ordered, then this bijection cyclically shifts this ordering by the same amount.

\begin{prop}\label{prop:swallow-function}
If we cyclically order the snakes in $\cals$ by $\mathfrak{s}_1,\dots,\mathfrak{s}_\alpha$ so that the co-successor of any live entry in $\mathfrak{s}_i$ lies in $\mathfrak{s}_{i+1}$, then for some constant $k$,
\[
\Swallow_\omega(\mathfrak{s}_i)=\mathfrak{s}_{i+k}\qquad
\text{for all $1\leq i\leq\alpha$},
\]
where all indices are taken modulo $\alpha$. 
\end{prop}

\begin{proof}
Since the snakes are cyclically ordered, it suffices to show that this holds for consecutive entries. Specifically, we will show that if $\Swallow(\mathfrak{s}_\alpha)=\mathfrak{s}_k$, then $\Swallow(\mathfrak{s}_1)=\mathfrak{s}_{k+1}$. Denote the $\omega$-head and the tail of $\mathfrak{s}_i$ by $h_i=\Head_\omega(\mathfrak{s}_i)$ and $t_i=\Tail(\mathfrak{s}_i)$, respectively.\footnote{Because we want this argument to work with both orbit table and orbit vector notation, we are not specifying whether $h_i$ and $t_i$ are integers or ordered pairs.}

It suffices to show that if $\bar{s}(h_\alpha)=t_k$, then $\bar{s}(h_1)=t_{k+1}$. Assuming the former,
start at $h_1$, and apply $\bar{s}$ to get to the tail $t_j$ of some snake $\mathfrak{s}_j$. Alternatively, applying $\bar{c}^{-1}$ takes us to some live entry on $\mathfrak{s}_\alpha$. Now, let $i\in\nn$ be the minimal number of $\bar{s}$-steps needed to reach the $\alpha$-head, i.e., $\bar{s}^i\big(\bar{c}^{-1}(h_1)\big)=h_\alpha$. The next $\bar{s}$-step takes us to the tail $t_k\in\mathfrak{s}_k$. Finally, $\bar{c}(t_k)$ lies on snake $\mathfrak{s}_{k+1}$. Because the ouroboros group is abelian, we have
\[
\bar{c}\,\bar{s}^{i+1}\bar{c}^{-1}(h_1)=\bar{s}^{i+1}(h_1)\in\mathfrak{s}_{k+1}.
\]
We also know that $\bar{s}^i\,\bar{s}(h_1)=\bar{s}^i(t_j)$, and further applications of $\bar{s}$ from the tail will remain on the snake $\mathfrak{s}_j$ (until we reach the $\omega$-head $h_j$), so $j=k+1$.
\end{proof}

We can analogously define the tail and $\omega$-head of a co-snake $\mathfrak{c}$, which we will denote by $\Head_\omega(\mathfrak{c})$ and $\Tail(\mathfrak{c})$, respectively. Applying the $\omega$-co-successor function from the $\omega$-head of each co-snake defines a bijection
\[
\Coswallow_\omega\colon\cc(\cals)\longto\cc(\cals)
\]
that we call the \dfn{co-swallow function}. The basic properties of the swallow function carry over to the co-swallow function. For example, as a permutation, $\Coswallow_\omega$ contains $\bar\beta_\omega$ disjoint cycles, each containing exactly $\deg(\p_\omega)$ co-snakes. If the co-snakes are canonically cyclically ordered, then this bijection cyclically shifts this ordering by the same amount. The proof is analogous to that of Proposition~\ref{prop:swallow-function} and will be omitted. 

\begin{prop}\label{prop:coswallow-function}
If we cyclically order the co-snakes in $\cals$ by $\mathfrak{c}_1,\dots,\mathfrak{c}_\beta$ so that the co-successor of any live entry in $\mathfrak{c}_i$ lies in $\mathfrak{c}_{i+1}$, then for some constant $k$,
\[
\Coswallow_\omega(\mathfrak{c}_i)=\mathfrak{c}_{i+k},\qquad
\text{for all $i=1,\dots,\beta$},
\]
where all indices are taken modulo $\beta$.
\end{prop}

Let us return to our running example. Since there are $\beta=6$ co-snakes but only $\alpha=2$ snakes, it is more illustrative to compute the co-swallow functions. The co-snakes are highlighted by color on the left in Figure~\ref{fig:co-swallow}. The entries marked with $\overline{1}$ indicate the tails of the co-snakes; these do not depend on $\omega$. The entries marked with $\underline{1}$ indicate the $1$-heads of the co-snakes; these do depend on the choice of $\omega=1$. 

\colorlet{color1}{Red}
\colorlet{color2}{Blue}
\colorlet{color3}{Red}
\colorlet{color4}{Blue}
\colorlet{color5}{purple}
\colorlet{color6}{Green}
\colorlet{color7}{orange}
\colorlet{color8}{amethyst}
\begin{figure}[!ht]
\begin{tikzpicture}
\setlength{\tabcolsep}{2.5pt}
\renewcommand{\arraystretch}{.4}
\begin{scope}
\tikzstyle{every node}=[font=\small,anchor=south]
\node [anchor=north] at (0,0) {
\begin{tabular}{ccccccccccccc}
& {\color{color3}$\mathfrak{c}_1$} & & {\color{color4}$\mathfrak{c}_2$} & & {\color{color5}$\mathfrak{c}_3$} & & {\color{color6}$\mathfrak{c}_4$} & {\color{color7}$\mathfrak{c}_5$} & & {\color{color8}$\mathfrak{c}_6$} & \\ \\
$x$ & $v_1$ & $v_2$ & $v_3$ & $v_4$ & $v_5$ & $v_6$ &
$v_7$ & $v_8$ & $v_9$ & $v_{1\!0}$ & $v_{1\!1}$ & \Bstrut
\\\hline
$x^{(0)}$ & \0 & \0 & \0 & \0 & {\bf\color{color5}$\overline{1}$} & \0 & {\bf\color{color6}$\overline{1}$} & \0 & \0 & \0 & \0 \\
$x^{(1)}$ & {\bf\color{color3}$\overline{1}$} & \0 & {\bf\color{color4}$\overline{1}$} & \0 & \0 & \0 & \0 & {\bf\color{color7}$\overline{1}$} & \0 & {\bf\color{color8}$\overline{1}$} & \0 \\
$x^{(2)}$ & \0 & \0 & \0 & {\bf\color{color5}1} & \0 & {\bf\color{color6}1} & \0 & \0 & \0 & \0 & {\bf\color{color3}1} \\
$x^{(3)}$ & \0 & {\bf\color{color4}1} & \0 & \0 & \0 & \0 & {\bf\color{color7}1} & \0 & {\bf\color{color8}1} & \0 & \0 \\
$x^{(4)}$ & \0 & \0 & {\bf\color{color5}1} & \0 & {\bf\color{color6}1} & \0 & \0 & \0 & \0 & {\bf\color{color3}1} & \0 \\
$x^{(5)}$ & {\bf\color{color4}\underline{1}} & \0 & \0 & \0 & \0 & {\bf\color{color7}\underline{1}} & \0 & {\bf\color{color8}\underline{1}} & \0 & \0 & \0 \\
$x^{(6)}$ & \0 & {\bf\color{color5}\underline{1}} & \0 & {\bf\color{color6}\underline{1}} & \0 & \0 & \0 & \0 & {\bf\color{color3}\underline{1}} & \0 & {\bf\color{color4}\underline{1}} \\ \vspace{-1mm}
$\vdots$ & $\vdots$ & $\vdots$ & $\vdots$ & $\vdots$ & $\vdots$ & $\vdots$ & $\vdots$
& $\vdots$ & $\vdots$  & $\vdots$ & $\vdots$
\Bstrut \\ \Tstrut
\end{tabular}};
\node at (8,-4.3) {$\Coswallow_1=(\mathfrak{c}_1\;\mathfrak{c}_5\;\mathfrak{c}_3)\,(\mathfrak{c}_2\;\mathfrak{c}_6\;\mathfrak{c}_4)$};
\colorlet{color1}{Red}
\colorlet{color2}{Red}
\colorlet{color3}{orange}
\colorlet{color4}{green!95!black}
\colorlet{color5}{orange}
\colorlet{color6}{green!95!black}
\colorlet{color7}{orange}
\colorlet{color8}{green!95!black}
\node [anchor=north] at (8,0) {
\begin{tabular}{cccccccccccc}
& {\color{color3}$\overline{\mathfrak{c}}_1$} & & {\color{color4}$\overline{\mathfrak{c}}_2$} & & {\color{color5}$\overline{\mathfrak{c}}_3$} & & {\color{color6}$\overline{\mathfrak{c}}_4$} &  {\color{color7}$\overline{\mathfrak{c}}_5$} & &  {\color{color8}$\overline{\mathfrak{c}}_6$} & \\ \\
$x$ & $v_1$ & $v_2$ & $v_3$ & $v_4$ & $v_5$ & $v_6$ &
$v_7$ & $v_8$ & $v_9$ & $v_{1\!0}$ & $v_{1\!1}$ \Bstrut
\\\hline
$x^{(0)}$ & \0 & \0 & \0 & \0 & {\bf\color{color5}$\overline{1}$} & \0 & {\bf\color{color6}$\overline{1}$} & \0 & \0 & \0 & \0 \\
$x^{(1)}$ & {\bf\color{color3}$\overline{1}$} & \0 & {\bf\color{color4}$\overline{1}$} & \0 & \0 & \0 & \0 & {\bf\color{color7}$\overline{1}$} & \0 & {\bf\color{color8}$\overline{1}$} & \0 \\
$x^{(2)}$ & \0 & \0 & \0 & {\bf\color{color5}1} & \0 & {\bf\color{color6}1} & \0 & \0 & \0 & \0 & {\bf\color{color3}1} \\
$x^{(3)}$ & \0 & {\bf\color{color4}1} & \0 & \0 & \0 & \0 & {\bf\color{color7}1} & \0 & {\bf\color{color8}1} & \0 & \0 \\
$x^{(4)}$ & \0 & \0 & {\bf\color{color5}1} & \0 & {\bf\color{color6}1} & \0 & \0 & \0 & \0 & {\bf\color{color3}1} & \0 \\
$x^{(5)}$ & {\bf\color{color4}\underline{1}} & \0 & \0 & \0 & \0 & {\bf\color{color7}\underline{1}} & \0 & {\bf\color{color8}\underline{1}} & \0 & \0 & \0 \\
$x^{(6)}$ & \0 & {\bf\color{color5}\underline{1}} & \0 & {\bf\color{color6}\underline{1}} & \0 & \0 & \0 & \0 & {\bf\color{color3}\underline{1}} & \0 & {\bf\color{color4}1}
\\ \vspace{-0.05 in} \\  \hline \Tstrut
\end{tabular}};
\end{scope}
\end{tikzpicture}
\caption{On the left is the repeating portion of the running example of our scroll. The tails of co-snakes are indicated by $\overline{1}$. On the right is the $\omega$-fold orbit table for $\omega=1$. The $\omega$-heads of the co-snakes are the positions of $\underline{1}$.}
\label{fig:co-swallow}
\end{figure}

\colorlet{color1}{Red}
\colorlet{color2}{Blue}
\colorlet{color3}{Red}
\colorlet{color4}{Blue}
\colorlet{color5}{purple}
\colorlet{color6}{Green}
\colorlet{color7}{orange}
\colorlet{color8}{amethyst}

From Figure~\ref{fig:co-swallow}, the co-swallow function is defined by
\[
\Coswallow_1({\color{color3}\mathfrak{c}_1})={\color{color7}\mathfrak{c}_5},\qquad
\Coswallow_1({\color{color7}\mathfrak{c}_5})={\color{color5}\mathfrak{c}_3},\qquad
\Coswallow_1({\color{color5}\mathfrak{c}_3})={\color{color3}\mathfrak{c}_1}
\]
and also 
\[
\Coswallow_1({\color{color4}\mathfrak{c}_2})={\color{color8}\mathfrak{c}_6},\qquad
\Coswallow_1({\color{color8}\mathfrak{c}_6})={\color{color6}\mathfrak{c}_4},\qquad
\Coswallow_1({\color{color6}\mathfrak{c}_4})={\color{color4}\mathfrak{c}_2}.\qquad
\]

In the language of Proposition~\ref{prop:swallow-function}, the swallow function is the mapping $\mathfrak{s}_i\mapsto\mathfrak{s}_{i+4}$, where the indices are taken modulo $6$. We can describe this by the permutation with cycle decomposition $(264)(153)$. 

By Proposition~\ref{prop:swallow-function}, the swallow function $\Swallow_\omega$ is the permutation sending $i\mapsto i+k$, so each cycle has length $\alpha/\gcd(\alpha,k)$, and there are $\gcd(\alpha,k)$ disjoint cycles. The different cycles correspond to the different ouroboroi.

\subsection{The $\omega$-fold vs.\ fundamental orbit table}

We are now ready to give explicit relationships between the properties of an $\omega$-fold orbit table, such as the (co-)degree and the number of (co-)snakes, and the corresponding properties of the fundamental (i.e., $1$-fold) orbit table. First, recall that $\deg(\cals)$ (resp., $\codeg(\cals)$) is the length of $\Slither(\cals)$ (resp., $\Coslither(\cals)$) divided by its length as a cyclic word, and that the (co-)snake scale is the minimal shift of the ticker tape that leaves the (co-)snakes invariant. These are denoted $p$ and $q$, respectively.

\begin{prop}\label{prop:degscale}
For any scroll $\cals$, 
\[
\deg(\p_1)\,\big|\,\deg(\cals)\quad\text{and}\quad\codeg(\p_1)\,\big|\,\codeg(\cals).
\]
\end{prop}
\begin{proof}
 Without loss of generality, suppose that $(0,1)\in\Live(\cals)$. By Corollary~\ref{cor:when-equiv}, the minimal $k>1$ such that
 \begin{equation}\label{eqn:when-equiv}
(k,1)\in\Live(\cals),\qquad \Slither(k,1)=\Slither(0,1),\quad\text{and}\quad \Coslither(k,1)=\Coslither(0,1)
\end{equation}
is the period $T(\cals)$ of the scroll. By definition, this is also the number of rows in the fundamental orbit table.  

Next, let $k'>0$ be minimal such that the three conditions in Eq.~\eqref{eqn:when-equiv} hold, with the additional condition that $(k',1)\in\Snake(0,1)$. Note that $k'/k$ is just the number of times the ouroboros wraps around from bottom to top before returning to the initial snake, so 
\[
k'=k\deg(\p_1)=\deg(\p_1)T(\cals)=\deg(\p_1)\frac{T(\calx)}{\gcd(T(\calx),n)}=\deg(\p_1)\frac{\lcm(T(\calx),n)}{n}.
\]
Next, we will express $k'$ in terms of $\deg(\cals)$.

Recall that $\scale(\calx)$ is the minimal $\sigma$ such that a live entry $h$ is on the same snake and co-snake as $h+\sigma$. It follows that $p=\scale(\calx)/\codeg(\calx)$ is the minimal $\ell$ such that $h$ is on the same snake as $h+\ell$ and $\Coslither(h) = \Coslither(h+\ell)$. Using the same line of reasoning as in the proof of Corollary~\ref{cor:scroll-period}, we find that 
\[
k'=\frac{1}{n}\lcm\left(\frac{\scale(\calx)}{\codeg(\calx)},n\right)=\frac{1}{n}\lcm\left(\deg(\calx)T(\calx),n\right).
\]
Applying Corollary~\ref{cor:scroll-period} shows that
\[
\deg(\p_1) = \frac{k'}{T(\cals)}
= \frac{\frac{1}{n}\lcm\left(\deg(\calx)T(\calx),n\right)}{\frac{1}{n}\lcm\left(T(\calx),n\right)}
=\deg(\calx)\,\frac{\gcd\left(T(\calx),n\right)}{\gcd\left(\deg(\calx)T(\calx),n\right)}.
\]
We now have the explicit formula
\[
\deg(\cals)=\deg(\calx)=\deg(\p_1)\,\frac{\gcd\left(\deg(\calx)T(\calx),n\right)}{\gcd\left(T(\calx),n\right)},
\]
which establishes that $\deg(\cals)$ divides $\deg(\p_1)$. The proof that $\codeg(\cals)$ divides $\codeg(\p_1)$ is analogous: we simply replace the ``additional condition'' with $(k',1)\in\Cosnake(0,1)$ and swap all instances of $\deg(\calx)$ and $\codeg(\calx)$.
\end{proof}

\begin{corollary}\label{cor:relprime}
The integers $\deg(\p_1)$ and $\codeg(\p_1)$ are relatively prime. 
\end{corollary}

\begin{proof}
Lemma~\ref{lem:co-prime-new} established that $\gcd(\deg(\cals),\codeg(\cals))=1$. The result now follows immediately from Proposition~\ref{prop:degscale}.
\end{proof}

Let $\mathfrak s_1,\dots, \mathfrak s_\alpha$ be the snakes in $\mathcal S$, cyclically ordered so that the co-successor of a live entry of $\mathfrak s_i$ lies in $\mathfrak s_{i+1}$. Suppose $\calt$ is an orbit table associated with $\cals$. For any $\mathfrak s_i$, the final live entry of $\mathfrak s_i$ in $\calt$ has an $\omega$-successor corresponding to an entry on the first row of $\calt$.

\begin{proposition}\label{prop:ouroboroi}
Suppose a scroll $\cals$ has $\alpha$ snakes and $\beta$ co-snakes and that its fundamental orbit table has $\bar\alpha$ ouroboroi and $\bar\beta$ co-ouroboroi. For $\omega>1$, the numbers $\bar\alpha_\omega$ 
 and $\bar\beta_\omega$ of ouroboroi and co-ouroboroi in its $\omega$-fold orbit table satisfy
\[
\bar\alpha_\omega=\bar\alpha\cdot\gcd(\deg(\p_1),\omega)\quad\text{and}\quad
\bar\beta_\omega=\bar\beta\cdot\gcd(\codeg(\p_1),\omega).
\]
\end{proposition}

\begin{proof}
By Lemma~\ref{prop:swallow-function}, there is some constant $k$ such that each snake $\mathfrak s_i$ maps to $\mathfrak s_{i+k}$ when wrapping vertically around the fundamental orbit table. It follows that each snake $\mathfrak s_i$ maps to $\mathfrak s_{i+\omega k}$ when wrapping vertically around the $\omega$-fold orbit table. This means that the number of ouroboroi in the $\omega$-fold orbit table is $\gcd(\omega k,\alpha)$, so 
\begin{align*}
\bar\alpha_\omega&=\gcd(\omega k,\alpha)\\
&=\gcd(\alpha,\gcd(\omega k,\omega\alpha))\\
&=\gcd(\alpha,\omega\cdot \gcd(k,\alpha))\\
&= \gcd(k,\alpha) \cdot \gcd(\alpha/ \gcd(k,\alpha),\omega)\\
&=\bar\alpha\cdot\gcd(\deg(\p_1),\omega). 
\end{align*}

The second equality is analogous. 
\end{proof}

Our running example of a scroll has $\alpha=2$ snakes and $\beta=6$ co-snakes, which we distinguished with different colors. As we took quotients to construct $\omega$-fold orbit tables, these snakes and co-snakes merged into (co)-ouroboroi, so we needed fewer colors to represent them. However, there are special cases when taking the quotient preserves the snakes and co-snakes, and thus the colors as well. Going back to our running example, notice that if $\omega$ is a multiple of $\deg(\p_1)=2$, then there are $\bar\alpha_\omega=\alpha=2$ ouroboroi, and if $\omega$ is a multiple of $\codeg(\p_1)=3$, then there are $\bar\beta_\omega=\beta=6$ co-ouroboroi. Specifically, these quantities are
\[
\bar\alpha_\omega=\begin{cases}2, & \omega\equiv 0\pmod{2} \\ 1, & \omega\equiv 1\pmod{2} \end{cases}\qquad\text{and}\qquad  
\bar\beta_\omega=\begin{cases}6, & \omega\equiv 0\pmod{3} \\ 2, & \omega\equiv 1\pmod{3} \\ 2, & \omega\equiv 2\pmod{3}. \end{cases}
\]

We now consider when the $\omega$-fold orbit table $\calt_\omega$ has the same number of ouroboroi as the scroll has snakes and also has the same number of co-ouroboroi as the scroll has co-snakes.
This happens precisely when $\omega$ is a multiple of $\lcm(\deg(\p_1),\codeg(\p_1))=\deg(\p_1)\cdot\codeg(\p_1)$.
In this case, we say that the orbit table $\calt_\omega=\p_\omega(\cals)$ is \dfn{color-preserving}.

\begin{proposition}
Given an orbit table $\calt_\omega=\p_\omega(\cals)$, the following conditions are equivalent:
\begin{enumerate}
\item $\bar\alpha_\omega=\alpha$ and $\bar\beta_\omega=\beta$;
\item $\Swallow_\omega$ and $\Coswallow_\omega$ are the identity permutation;
\item there is a bijection between the set of (co-)snakes in the scroll and the set of (co-)ouroboroi in the orbit table;
\item the number $rn=\omega nT(\cals)$ of entries in $\calt_\omega$ is a multiple of $\Scale(\calx)$;
\item $\deg(\p_1)\cdot\codeg(\p_1)$ divides $\omega$. 
\end{enumerate}
\end{proposition}

\begin{proof}
The equivalence of conditions (1), (2), (3) is immediate from the above discussion. 

Let $k$ be the smallest positive integer such that $X_k=1$. Let $\mathfrak s$ and $\mathfrak c$ be the snake and the co-snake, respectively, containing $k$. Let $k'$ be the smallest integer such that $k'>rn$ and $k'\equiv k\pmod{\Scale(\calx)}$. Then $k'$ is a live entry in $\Swallow_\omega(\mathfrak s)\cap\Coswallow_\omega(\mathfrak c)$. It follows that $\Scale(\calx)$ divides $rn$ if and only if $\Swallow(\mathfrak s)=\mathfrak s$ and $\Coswallow(\mathfrak c)=\mathfrak c$. But all cycles of $\Swallow(\mathfrak s)$ (respectively, $\Coswallow(\mathfrak c)$) have the same size, so (4) is equivalent to (2).

From Proposition~\ref{prop:ouroboroi}, we see that (1) holds if and only if $\lcm(\deg(\p_1),\codeg(\p_1)) $ divides $\omega$. Since $\deg(\p_1)$ and $\codeg(\p_1)$ are relatively prime by Corollary~\ref{cor:relprime}, conditions (1) and (5) are equivalent.
\end{proof}

\section{Periods of sum vectors}\label{sec:sum-vectors}

In this section, we will introduce the \emph{sum vector} of a scroll and fully classify the possible periods for the sum vector of a scroll on $n$ vertices, when viewed as a cyclic word. 

For an orbit table $\calt$, define the \dfn{sum vector} of $\calt$, written $\vec{\Sigma}(\calt)$, to be the vector in $\mathbb N^n$ given by the column sums of $\calt$. Clearly, for every $\omega\in \mathbb N$, the $\omega$-fold orbit table of a scroll $\cals$ has sum vector $\omega \cdot \vec{\Sigma}(\calt_1)$. Thus, we can define $\vec{\Sigma}(\cals)$ to be the sum vector of any orbit table of $\cals$, and this vector is well-defined up to scalar multiples. Note that $\vec{\Sigma}(\cals)$ indicates the relative frequency of live entries in each column of $\cals$. From the toggling perspective, this means that $\vec{\Sigma}(\cals)$ indicates the relative frequency with which each vertex of $\calc_n$ appears in the independent sets that we obtain from iteratively applying the toggling operation $\tau$. 

In many orbit tables, such as the one in our running example in Figure~\ref{fig:example}, the sum vector is constant, so its period is $1$. However, this is not always the case. For example, the sum vector of the orbit table in Figure~\ref{fig:sum-vector-period-three} has period $3$.

\begin{figure}[!ht]
    \centering
    \setlength{\tabcolsep}{4.5pt}
\renewcommand{\arraystretch}{.8}
\tikzstyle{every node}=[font=\small,anchor=south]
    \begin{tabular}{cccccccccccccc}
$x$ & $v_1$ & $v_2$ & $v_3$ & $v_4$ & $v_5$ & $v_6$ &
$v_7$ & $v_8$ & $v_9$ & $v_{1\!0}$ & $v_{1\!1}$ & $v_{1\!2}$
\\\hline
$x^{(0)}$ & {\bf\color{color3}1} & \0 & {\bf\color{color4}1} & \0 & \0 & \0 & \0 & \0 & {\bf\color{color7}1} & \0 & {\bf\color{color8}1} & \0 \\
$x^{(1)}$ & \0 & \0 & \0 & {\bf\color{color5}1} & \0 & {\bf\color{color6}1} & \0 & \0 & \0 & \0 & \0 & {\bf\color{color3}1} \\
$x^{(2)}$ & \0 & {\bf\color{color4}1} & \0 & \0 & \0 & \0 & {\bf\color{color7}1} & \0 & {\bf\color{color8}1} & \0 & \0 & \0 \\
$x^{(3)}$ & \0 & \0 & {\bf\color{color5}1} & \0 & {\bf\color{color6}1} & \0 & \0 & \0 & \0 & {\bf\color{color3}1} & \0 & {\bf\color{color4}1} \\
$x^{(4)}$ & \0 & \0 & \0 & \0 & \0 & {\bf\color{color7}1} & \0 & {\bf\color{color8}1} & \0 & \0 & \0 & \0 \\ \hline \Tstrut
{\bf Sum:} & \bf{1} & \bf{1} & \bf{2} & \bf{1} & \bf{1} & \bf{2} & \bf{1} & \bf{1} & \bf{2} & \bf{1} & \bf{1} & \bf{2}
\end{tabular}
    \caption{An orbit table whose sum vector $\vec{\Sigma}(\calt_1)=(1,1,2,1,1,2,1,1,2,1,1,2)$ has period 3.}
    \label{fig:sum-vector-period-three}
\end{figure}

The sum vector of an orbit table or scroll is closely related to the scale, the distance between two consecutive entries in any fiber. Specifically, we will want to keep track of the number of \emph{columns} between such entries, and this is just the scale modulo $n$. However, it is also helpful to explicitly define this as an integer in $\{0,\dots,n-1\}$ (i.e., a particular residue modulo $n$).

\begin{definition}\label{def:colscale}
    The \dfn{column scale} of $\cals$, denoted $\colscale(\cals)$, is the value in $\{0,1,\dots,n-1\}$ that is congruent to $\scale(\cals)$ modulo $n$. 
\end{definition}

By definition, if $(i,j)$ and $(i',j')$ are consecutive entries in the same fiber, then $j+\colscale(\cals)\equiv j'\pmod{n}$.
 
\begin{lemma}~\label{lem:colscale}
For any scroll $\mathcal S$, the following equality holds: 
\[
\colscale(\cals)=\beta_D+2\beta_{\E}=n-\alpha_S - 2\alpha_L.
\]
\end{lemma}
\begin{proof}
The first equality follows from Definition~\ref{def:colscale} and Lemma~\ref{lem:scale}, as well as the fact that $\beta_D + 2\beta_{\E} <n$, which follows from Corollary~\ref{cor:2a+3b+4c=n+1}. For the second equality, we again apply Corollary~\ref{cor:2a+3b+4c=n+1} to find that
\begin{align*}
    \colscale(\cals) &= \beta_D+2\beta_{\E} \\
    &= 2\alpha_S + 2\alpha_L + 2\beta_{\E}- 1\\
    & = (2\beta_{\E} + 3\alpha_S + 4\alpha_L) - \alpha_S - 2\alpha_L -1 \\
    & = n+1 - \alpha_S - 2\alpha_L -1\\
    &= n- \alpha_S - 2\alpha_L.\qedhere
\end{align*}
\end{proof}

\begin{lemma}\label{lem:sumvector-scale}
For any scroll $\mathcal S$, the period of $\vec{\Sigma}(\cals)$ must divide $\gcd(n,\Scale(\cals)) = \gcd(n,\colscale(\cals))$. 
\end{lemma}
\begin{proof}
Clearly, the period of $\vec{\Sigma}(\cals)$ divides the period of the ticker tape, which divides $\Scale(\cals)$ by Theorem~\ref{thm:ttperiod}.  Since
$\vec{\Sigma}(\cals)$ has length $n$, its period divides $n$.
Thus, the period of $\vec{\Sigma}(\cals)$ divides $\gcd(n,\Scale(\cals))$. 
\end{proof}

\begin{corollary}\label{cor:sum-vector}
For any scroll $\cals$, the sum vector $\vec{\Sigma}(\cals)$ has odd period. 
\end{corollary}

\begin{proof}
Let $\lambda$ be the period of $\vec{\Sigma}(\cals)$. Thus far, we have established that
\[
\lambda\mid\colscale(\cals)\hspace{.5 cm} \text{ and } \hspace{.5 cm}\colscale(\cals) = \beta_D+2\beta_{\E}
\]
where the divisibility is by Lemma~\ref{lem:sumvector-scale} and the equality by Lemma~\ref{lem:colscale}. The result is immediate from Corollary~\ref{cor:2a+3b+4c=n+1}, which tells us that $\beta_D$ is odd.
\end{proof}

\begin{thm}\label{thm:at least 4r}
If the sum vector of a scroll on $n\geq 4$ vertices has period $\lambda$, then $\lambda \mid n$ and $n \geq 4\lambda$. 
\end{thm}
\begin{proof}
We may assume $\lambda>1$ since the result is trivial otherwise. Since it follows from Lemma~\ref{lem:sumvector-scale} that $\lambda\mid n$, we just need to show that $n \not\in\{\lambda,2\lambda,3\lambda\}$. 

By Lemmas~\ref{lem:colscale} and~\ref{lem:sumvector-scale},
\[
\lambda \mid \colscale(\cals) \hspace{.5 cm} \text{and} \hspace{.5 cm} \colscale(\cals) = n - \alpha_S - 2 \alpha_L = \beta_D + 2\beta_{\E}<n.
\]
It follows that $\lambda\neq n$. Suppose by way of contradiction that $n = 2\lambda$. Then, we must have $\colscale(\cals)=\frac{n}{2}$. In this case, $\alpha_S + 2 \alpha_L = \beta_D + 2\beta_{\E} = \frac n2$. By Corollary~\ref{cor:2a+3b+4c=n+1}, this implies that $\alpha_S + 2 \alpha_L = 2\alpha_S + 2 \alpha_L + 2\beta_{\E} - 1$. It follows that $\beta_{\E} = 0$ and $\alpha_S = 1$. However, when $\beta_{\E} = 0$, the slither is made entirely of $D$s, and each snake has the same number of live entries in each column. This means that the period of $\vec{\Sigma}(\cals)$ must be $1$, which contradicts the assumption that $\lambda>1$. 

Now suppose $n = 3\lambda$. We must have $\alpha_S + 2 \alpha_L \in \big\{\frac n3,\frac{2n}3\big\}$. If $\alpha_S + 2 \alpha_L = \frac {2n}3$, then 
\[
\alpha_S + 2 \alpha_L = 2(\beta_D + 2\beta_{\E}) = 2(2\alpha_S + 2 \alpha_L + 2\beta_{\E} - 1),
\]
which implies that $1=3\alpha_S + 2\alpha_L + 4\beta_{\E}$. This is clearly impossible. On the other hand, if $\alpha_S + 2 \alpha_L = \frac n3$, then $2(\alpha_S + 2 \alpha_L) = 2\alpha_S + 2 \alpha_L + 2\beta_{\E} - 1$. This is impossible because the two sides of the equation have opposite parities. 
\end{proof}

Next, we will prove a partial converse to Theorem~\ref{thm:at least 4r}. In particular, we will show that if $\lambda \mid n$ and $n \ge 4\lambda$, then there must exist some scroll on $n$ vertices whose sum vector has period $\lambda$. Before proving this theorem, we introduce one more useful definition. 

Suppose a snake has a slither of length $\beta$. A \dfn{segment} of the snake is a subset of $\Live(\cals)$ of the form 
\[
\left\{k,s(k),s^2(k),\dots,s^{\beta-1}(k)\right\},
\]
where $k$ is any live entry.

\begin{theorem}\label{thm:snakeconstruction}
For any $\lambda,k \in \mathbb N$ such that $\lambda$ is odd and $k \ge 4$, there exists a scroll $\cals$ on $n = k\lambda$ vertices such that the period of $\vec{\Sigma}(\cals)$ is $\lambda$.
\end{theorem}

\begin{proof}
We prove this result via a construction that depends on the parity of $k$. Since the sum vectors of the $\omega$-fold orbit tables for different choices of $\omega$ always have the same period, we can work with the minimal color-preserving orbit table. 

First, suppose $k$ is even. We will use the slither
\[D^{\lambda-2}\E D^2\E^{\frac{\lambda k}2 - \lambda - 1}\]
and the co-slither
\[SL^{\frac{\lambda-1}2}.\]
Note that this slither and co-slither are a feasible pair, and thus define a ticker tape (and scroll) by Theorem~\ref{thm:2a+3b+4c=n+1}. Furthermore, by a straightforward calculation, we find that the column scale of the associated scroll is $\lambda(k-1)$. Because $\gcd(\lambda(k-1),\lambda k) = \lambda$, each snake can be divided into $n/\lambda = k$ disjoint segments.

We will start by finding the contribution to the sum vector coming from a single snake, and we will then consider the total sum vector. Consider the mod $n$ positions of live entries that form a single snake. For a single segment, we begin with $\lambda-1$ adjacent positions, skip one position, take two more adjacent positions, and then take every other position until we place a 0 in position $\lambda(k-1)$. The next slither has an equivalent effect on the sum, but it is shifted by $\lambda(k-1)$ positions. Figure~\ref{fig:snakeconstruction} gives an example of this construction for $\lambda = 7$ and $k=4$.

\colorlet{color1}{Red}
\colorlet{color2}{Blue}
\colorlet{color3}{orange}
\colorlet{color4}{green!95!black}
\colorlet{color5}{magenta}

\begin{figure}[!ht]
\begin{tikzpicture}
\node at (0,0) {\bf\color{color1}1};
\node at (.35,0) {\bf\color{color1}1};
\node at (.7,0) {\bf\color{color1}1};
\node at (1.05,0) {\bf\color{color1}1};
\node at (1.4,0) {\bf\color{color1}1};
\node at (1.75,0) {\bf\color{color1}1};
\node at (2.1,0) {\0};

\node at (2.8,0) {\bf\color{color1}1};
\node at (3.15,0) {\bf\color{color1}1};
\node at (3.5,0) {\bf\color{color1}1};
\node at (3.85,0) {\0};
\node at (4.2,0) {\bf\color{color1}1};
\node at (4.55,0) {\0};
\node at (4.9,0) {\bf\color{color1}1};

\node at (5.6,0) {\0};
\node at (5.95,0) {\bf\color{color1}1};
\node at (6.3,0) {\0};
\node at (6.65,0) {\bf\color{color1}1};
\node at (7,0) {\0};
\node at (7.35,0) {\bf\color{color1}1};
\node at (7.7,0) {\0};

\node at (8.4,0) {\bf\color{color2}1};
\node at (8.75,0) {\bf\color{color2}1};
\node at (9.1,0) {\bf\color{color2}1};
\node at (9.45,0) {\bf\color{color2}1};
\node at (9.8,0) {\bf\color{color2}1};
\node at (10.15,0) {\bf\color{color2}1};
\node at (10.5,0) {\0};

\begin{scope}[shift={(0,-.5)}]
\node at (0,0) {\bf\color{color2}1};
\node at (.35,0) {\bf\color{color2}1};
\node at (.7,0) {\bf\color{color2}1};
\node at (1.05,0) {\0};
\node at (1.4,0) {\bf\color{color2}1};
\node at (1.75,0) {\0};
\node at (2.1,0) {\bf\color{color2}1};

\node at (2.8,0) {\0};
\node at (3.15,0) {\bf\color{color2}1};
\node at (3.5,0) {\0};
\node at (3.85,0) {\bf\color{color2}1};
\node at (4.2,0) {\0};
\node at (4.55,0) {\bf\color{color2}1};
\node at (4.9,0) {\0};

\node at (5.6,0) {\bf\color{color3}1};
\node at (5.95,0) {\bf\color{color3}1};
\node at (6.3,0) {\bf\color{color3}1};
\node at (6.65,0) {\bf\color{color3}1};
\node at (7,0) {\bf\color{color3}1};
\node at (7.35,0) {\bf\color{color3}1};
\node at (7.7,0) {\0};

\node at (8.4,0) {\bf\color{color3}1};
\node at (8.75,0) {\bf\color{color3}1};
\node at (9.1,0) {\bf\color{color3}1};
\node at (9.45,0) {\0};
\node at (9.8,0) {\bf\color{color3}1};
\node at (10.15,0) {\0};
\node at (10.5,0) {\bf\color{color3}1};
\end{scope}

\begin{scope}[shift={(0,-1)}]
\node at (0,0) {\0};
\node at (.35,0) {\bf\color{color3}1};
\node at (.7,0) {\0};
\node at (1.05,0) {\bf\color{color3}1};
\node at (1.4,0) {\0};
\node at (1.75,0) {\bf\color{color3}1};
\node at (2.1,0) {\0};

\node at (2.8,0) {\bf\color{color4}1};
\node at (3.15,0) {\bf\color{color4}1};
\node at (3.5,0) {\bf\color{color4}1};
\node at (3.85,0) {\bf\color{color4}1};
\node at (4.2,0) {\bf\color{color4}1};
\node at (4.55,0) {\bf\color{color4}1};
\node at (4.9,0) {\0};

\node at (5.6,0) {\bf\color{color4}1};
\node at (5.95,0) {\bf\color{color4}1};
\node at (6.3,0) {\bf\color{color4}1};
\node at (6.65,0) {\0};
\node at (7,0) {\bf\color{color4}1};
\node at (7.35,0) {\0};
\node at (7.7,0) {\bf\color{color4}1};

\node at (8.4,0) {\0};
\node at (8.75,0) {\bf\color{color4}1};
\node at (9.1,0) {\0};
\node at (9.45,0) {\bf\color{color4}1};
\node at (9.8,0) {\0};
\node at (10.15,0) {\bf\color{color4}1};
\node at (10.5,0) {\0};

\end{scope}

\begin{scope}[shift={(0,-1.6)}]

\node at (-.8,0) {\bf {Sum:}};

\node at (0,0) {\bf 2};
\node at (.35,0) {\bf 3};
\node at (.7,0) {\bf 2};
\node at (1.05,0) {\bf 2};
\node at (1.4,0) {\bf 2};
\node at (1.75,0) {\bf 2};
\node at (2.1,0) {\bf 1};

\node at (2.8,0) {\bf 2};
\node at (3.15,0) {\bf 3};
\node at (3.5,0) {\bf 2};
\node at (3.85,0) {\bf 2};
\node at (4.2,0) {\bf 2};
\node at (4.55,0) {\bf 2};
\node at (4.9,0) {\bf 1};

\node at (5.6,0) {\bf 2};
\node at (5.95,0) {\bf 3};
\node at (6.3,0) {\bf 2};
\node at (6.65,0) {\bf 2};
\node at (7,0) {\bf 2};
\node at (7.35,0) {\bf 2};
\node at (7.7,0) {\bf 1};

\node at (8.4,0) {\bf 2};
\node at (8.75,0) {\bf 3};
\node at (9.1,0) {\bf 2};
\node at (9.45,0) {\bf 2};
\node at (9.8,0) {\bf 2};
\node at (10.15,0) {\bf 2};
\node at (10.5,0) {\bf 1};

\end{scope}
\begin{scope}[shift = {(0,-2.5)}]
\node at (-2,0) {\bf {Snake 1 Contrib.:}};

\node at (0,0) {2};
\node at (.35,0) {3};
\node at (.7,0) {2};
\node at (1.05,0) {2};
\node at (1.4,0) {2};
\node at (1.75,0) {2};
\node at (2.1,0) {1};

\node at (2.8,0) {2};
\node at (3.15,0) {3};
\node at (3.5,0) {2};
\node at (3.85,0) {2};
\node at (4.2,0) {2};
\node at (4.55,0) {2};
\node at (4.9,0) {1};

\node at (5.6,0) {2};
\node at (5.95,0) {3};
\node at (6.3,0) {2};
\node at (6.65,0) {2};
\node at (7,0) {2};
\node at (7.35,0) {2};
\node at (7.7,0) {1};

\node at (8.4,0) {2};
\node at (8.75,0) {3};
\node at (9.1,0) {2};
\node at (9.45,0) {2};
\node at (9.8,0) {2};
\node at (10.15,0) {2};
\node at (10.5,0) {1};

\begin{scope}[shift={(0,-.5)}]

\node at (-2,0) {\bf {Snake 2 Contrib.:}};

\node at (0,0) {3};
\node at (.35,0) {2};
\node at (.7,0) {2};
\node at (1.05,0) {2};
\node at (1.4,0) {2};
\node at (1.75,0) {1};
\node at (2.1,0) {2};

\node at (2.8,0) {3};
\node at (3.15,0) {2};
\node at (3.5,0) {2};
\node at (3.85,0) {2};
\node at (4.2,0) {2};
\node at (4.55,0) {1};
\node at (4.9,0) {2};

\node at (5.6,0) {3};
\node at (5.95,0) {2};
\node at (6.3,0) {2};
\node at (6.65,0) {2};
\node at (7,0) {2};
\node at (7.35,0) {1};
\node at (7.7,0) {2};

\node at (8.4,0) {3};
\node at (8.75,0) {2};
\node at (9.1,0) {2};
\node at (9.45,0) {2};
\node at (9.8,0) {2};
\node at (10.15,0) {1};
\node at (10.5,0) {2};

\end{scope}

\begin{scope}[shift={(0,-1)}]

\node at (-2,0) {\bf {Snake 3 Contrib.:}};

\node at (0,0) {2};
\node at (.35,0) {2};
\node at (.7,0) {2};
\node at (1.05,0) {1};
\node at (1.4,0) {2};
\node at (1.75,0) {3};
\node at (2.1,0) {2};

\node at (2.8,0) {2};
\node at (3.15,0) {2};
\node at (3.5,0) {2};
\node at (3.85,0) {1};
\node at (4.2,0) {2};
\node at (4.55,0) {3};
\node at (4.9,0) {2};

\node at (5.6,0) {2};
\node at (5.95,0) {2};
\node at (6.3,0) {2};
\node at (6.65,0) {1};
\node at (7,0) {2};
\node at (7.35,0) {3};
\node at (7.7,0) {2};

\node at (8.4,0) {2};
\node at (8.75,0) {2};
\node at (9.1,0) {2};
\node at (9.45,0) {1};
\node at (9.8,0) {2};
\node at (10.15,0) {3};
\node at (10.5,0) {2};

\end{scope}
\begin{scope}[shift={(0,-1.5)}]

\node at (-2,0) {\bf {Snake 4 Contrib.:}};

\node at (0,0) {2};
\node at (.35,0) {1};
\node at (.7,0) {2};
\node at (1.05,0) {3};
\node at (1.4,0) {2};
\node at (1.75,0) {2};
\node at (2.1,0) {2};

\node at (2.8,0) {2};
\node at (3.15,0) {1};
\node at (3.5,0) {2};
\node at (3.85,0) {3};
\node at (4.2,0) {2};
\node at (4.55,0) {2};
\node at (4.9,0) {2};

\node at (5.6,0) {2};
\node at (5.95,0) {1};
\node at (6.3,0) {2};
\node at (6.65,0) {3};
\node at (7,0) {2};
\node at (7.35,0) {2};
\node at (7.7,0) {2};

\node at (8.4,0) {2};
\node at (8.75,0) {1};
\node at (9.1,0) {2};
\node at (9.45,0) {3};
\node at (9.8,0) {2};
\node at (10.15,0) {2};
\node at (10.5,0) {2};

\end{scope}

\begin{scope}[shift={(0,-2.1)}]

\node at (-1.9,0) {\bf {Overall Sum:}};

\node at (0,0) {\bf 9};
\node at (.35,0) {\bf 8};
\node at (.7,0) {\bf 8};
\node at (1.05,0) {\bf 8};
\node at (1.4,0) {\bf 8};
\node at (1.75,0) {\bf 8};
\node at (2.1,0) {\bf 7};

\node at (2.8,0) {\bf 9};
\node at (3.15,0) {\bf 8};
\node at (3.5,0) {\bf 8};
\node at (3.85,0) {\bf 8};
\node at (4.2,0) {\bf 8};
\node at (4.55,0) {\bf 8};
\node at (4.9,0) {\bf 7};

\node at (5.6,0) {\bf 9};
\node at (5.95,0) {\bf 8};
\node at (6.3,0) {\bf 8};
\node at (6.65,0) {\bf 8};
\node at (7,0) {\bf 8};
\node at (7.35,0) {\bf 8};
\node at (7.7,0) {\bf 7};

\node at (8.4,0) {\bf 9};
\node at (8.75,0) {\bf 8};
\node at (9.1,0) {\bf 8};
\node at (9.45,0) {\bf 8};
\node at (9.8,0) {\bf 8};
\node at (10.15,0) {\bf 8};
\node at (10.5,0) {\bf 7};

\end{scope}
\end{scope}
\end{tikzpicture}
\caption{An example of the construction used for the proof of Theorem~\ref{thm:snakeconstruction} when $k=4$ and $\lambda=7$. The first 3 rows of the figure show the column position of the live entries from a single snake, where different colors represent different segments. Note that when two live entries are adjacent in this presentation, they appear on subsequent rows of the orbit table, but the row position of live entries does not affect the column sums. The fourth row is the column sum of the first 3 rows, which is the contribution to the sum vector from a single snake. In the remaining rows of the figure, we add the contributions to the sum vector from each of the four snakes. This overall sum is the sum vector of the orbit table when $k=4$ and $\lambda = 7$. We show in the proof that this sum vector always has period $\lambda$. }\label{fig:snakeconstruction}
\end{figure}

Notice from the structure of the slither that the contribution to the sum vector from a single snake is periodic with period $\lambda$. In particular, almost all of the first $\lambda$ entries have the same value $x$, except that the second entry is $x+1$ and the $\lambda^{\text{th}}$ entry is $x-1$. This pattern repeats for every set of $\lambda$ columns. 

Now, we need to consider the sum vectors of the other snakes. Notice that for every $S$ in the co-slither, we get a snake whose sum vector is shifted one position left from the previous snake, while for every $L$ in the co-slither, we get a snake whose sum vector is shifted two positions left from the previous snake. For our co-slither of $SL^{(\lambda-1)/2}$, we consider how many times each of the entries $x$, $x-1$ and $x+1$ appear in each column. Through a straightforward computation, we find that $x+1$ appears once in the first column as well as every column (of the first $\lambda$) of even index but does not appear in the other columns. Similarly, $x-1$ appears once in the $\lambda^{\text{th}}$ column as well as once in every even-indexed column, but it does not appear in the other columns. Overall, this means that the initial $\lambda$ entries in the sum vector are $a+1,a,a\ldots,a,a-1$ for some $a$ (see Figure~\ref{fig:snakeconstruction}). It follows that the period of the sum vector is $\lambda$.

In the case where $k$ is odd, we use the slither
\[D^{2\lambda+1}\E^{((k-4)\lambda-1)/2}\]
and the co-slither
\[S^2L^{\lambda-1}.\]
It is again easy to confirm that the slither and co-slither are a feasible pair. Furthermore, by a straightforward calculation, we find that the column scale of the associated scroll is $\lambda(k-2)$. The argument for this case is essentially analogous to the argument we used for the case where $k$ is even. Figure~\ref{fig:snakeconstruction2} gives an example of this construction for $\lambda = 7$ and $k=5$.

\begin{figure}[!ht]
\begin{tikzpicture}
\node at (0,0) {\bf\color{color1}1};
\node at (.35,0) {\bf\color{color1}1};
\node at (.7,0) {\bf\color{color1}1};
\node at (1.05,0) {\bf\color{color1}1};
\node at (1.4,0) {\bf\color{color1}1};
\node at (1.75,0) {\bf\color{color1}1};
\node at (2.1,0) {\bf\color{color1}1};

\node at (2.8,0) {\bf\color{color1}1};
\node at (3.15,0) {\bf\color{color1}1};
\node at (3.5,0) {\bf\color{color1}1};
\node at (3.85,0) {\bf\color{color1}1};
\node at (4.2,0) {\bf\color{color1}1};
\node at (4.55,0) {\bf\color{color1}1};
\node at (4.9,0) {\bf\color{color1}1};

\node at (5.6,0) {\bf\color{color1}1};
\node at (5.95,0) {\bf\color{color1}1};
\node at (6.3,0) {\0};
\node at (6.65,0) {\bf\color{color1}1};
\node at (7,0) {\0};
\node at (7.35,0) {\bf\color{color1}1};
\node at (7.7,0) {\0};

\node at (8.4,0) {\bf\color{color2}1};
\node at (8.75,0) {\bf\color{color2}1};
\node at (9.1,0) {\bf\color{color2}1};
\node at (9.45,0) {\bf\color{color2}1};
\node at (9.8,0) {\bf\color{color2}1};
\node at (10.15,0) {\bf\color{color2}1};
\node at (10.5,0) {\bf\color{color2}1};

\node at (11.2,0) {\bf\color{color2}1};
\node at (11.55,0) {\bf\color{color2}1};
\node at (11.9,0) {\bf\color{color2}1};
\node at (12.25,0) {\bf\color{color2}1};
\node at (12.6,0) {\bf\color{color2}1};
\node at (12.95,0) {\bf\color{color2}1};
\node at (13.3,0) {\bf\color{color2}1};

\begin{scope}[shift={(0,-.5)}]
\node at (0,0) {\bf\color{color2}1};
\node at (.35,0) {\bf\color{color2}1};
\node at (.7,0) {\0};
\node at (1.05,0) {\bf\color{color2}1};
\node at (1.4,0) {\0};
\node at (1.75,0) {\bf\color{color2}1};
\node at (2.1,0) {\0};

\node at (2.8,0) {\bf\color{color3}1};
\node at (3.15,0) {\bf\color{color3}1};
\node at (3.5,0) {\bf\color{color3}1};
\node at (3.85,0) {\bf\color{color3}1};
\node at (4.2,0) {\bf\color{color3}1};
\node at (4.55,0) {\bf\color{color3}1};
\node at (4.9,0) {\bf\color{color3}1};

\node at (5.6,0) {\bf\color{color3}1};
\node at (5.95,0) {\bf\color{color3}1};
\node at (6.3,0) {\bf\color{color3}1};
\node at (6.65,0) {\bf\color{color3}1};
\node at (7,0) {\bf\color{color3}1};
\node at (7.35,0) {\bf\color{color3}1};
\node at (7.7,0) {\bf\color{color3}1};

\node at (8.4,0) {\bf\color{color3}1};
\node at (8.75,0) {\bf\color{color3}1};
\node at (9.1,0) {\0};
\node at (9.45,0) {\bf\color{color3}1};
\node at (9.8,0) {\0};
\node at (10.15,0) {\bf\color{color3}1};
\node at (10.5,0) {\0};

\node at (11.2,0) {\bf\color{color4}1};
\node at (11.55,0) {\bf\color{color4}1};
\node at (11.9,0) {\bf\color{color4}1};
\node at (12.25,0) {\bf\color{color4}1};
\node at (12.6,0) {\bf\color{color4}1};
\node at (12.95,0) {\bf\color{color4}1};
\node at (13.3,0) {\bf\color{color4}1};
\end{scope}

\begin{scope}[shift={(0,-1)}]
\node at (0,0) {\bf\color{color4}1};
\node at (.35,0) {\bf\color{color4}1};
\node at (.7,0) {\bf\color{color4}1};
\node at (1.05,0) {\bf\color{color4}1};
\node at (1.4,0) {\bf\color{color4}1};
\node at (1.75,0) {\bf\color{color4}1};
\node at (2.1,0) {\bf\color{color4}1};

\node at (2.8,0) {\bf\color{color4}1};
\node at (3.15,0) {\bf\color{color4}1};
\node at (3.5,0) {\0};
\node at (3.85,0) {\bf\color{color4}1};
\node at (4.2,0) {\0};
\node at (4.55,0) {\bf\color{color4}1};
\node at (4.9,0) {\0};

\node at (5.6,0) {\bf\color{color5}1};
\node at (5.95,0) {\bf\color{color5}1};
\node at (6.3,0) {\bf\color{color5}1};
\node at (6.65,0) {\bf\color{color5}1};
\node at (7,0) {\bf\color{color5}1};
\node at (7.35,0) {\bf\color{color5}1};
\node at (7.7,0) {\bf\color{color5}1};

\node at (8.4,0) {\bf\color{color5}1};
\node at (8.75,0) {\bf\color{color5}1};
\node at (9.1,0) {\bf\color{color5}1};
\node at (9.45,0) {\bf\color{color5}1};
\node at (9.8,0) {\bf\color{color5}1};
\node at (10.15,0) {\bf\color{color5}1};
\node at (10.5,0) {\bf\color{color5}1};

\node at (11.2,0) {\bf\color{color5}1};
\node at (11.55,0) {\bf\color{color5}1};
\node at (11.9,0) {\0};
\node at (12.25,0) {\bf\color{color5}1};
\node at (12.6,0) {\0};
\node at (12.95,0) {\bf\color{color5}1};
\node at (13.3,0) {\0};
\end{scope}

\begin{scope}[shift={(0,-1.6)}]

\node at (-.8,0) {\bf {Sum:}};

\node at (0,0) {\bf 3};
\node at (.35,0) {\bf 3};
\node at (.7,0) {\bf 2};
\node at (1.05,0) {\bf 3};
\node at (1.4,0) {\bf 2};
\node at (1.75,0) {\bf 3};
\node at (2.1,0) {\bf 2};

\node at (2.8,0) {\bf 3};
\node at (3.15,0) {\bf 3};
\node at (3.5,0) {\bf 2};
\node at (3.85,0) {\bf 3};
\node at (4.2,0) {\bf 2};
\node at (4.55,0) {\bf 3};
\node at (4.9,0) {\bf 2};

\node at (5.6,0) {\bf 3};
\node at (5.95,0) {\bf 3};
\node at (6.3,0) {\bf 2};
\node at (6.65,0) {\bf 3};
\node at (7,0) {\bf 2};
\node at (7.35,0) {\bf 3};
\node at (7.7,0) {\bf 2};

\node at (8.4,0) {\bf 3};
\node at (8.75,0) {\bf 3};
\node at (9.1,0) {\bf 2};
\node at (9.45,0) {\bf 3};
\node at (9.8,0) {\bf 2};
\node at (10.15,0) {\bf 3};
\node at (10.5,0) {\bf 2};

\node at (11.2,0) {\bf 3};
\node at (11.55,0) {\bf 3};
\node at (11.9,0) {\bf 2};
\node at (12.25,0) {\bf 3};
\node at (12.6,0) {\bf 2};
\node at (12.95,0) {\bf 3};
\node at (13.3,0) {\bf 2};

\end{scope}

\end{tikzpicture}
\caption{An example of the construction used for the proof of Theorem~\ref{thm:snakeconstruction} when $k=5$ and $\lambda=7$. The different colors represent different segments. The 1s are in the correct column (but not the correct row) for a single snake so that we can see the impact on the sum vector. }\label{fig:snakeconstruction2}
\end{figure}

Once we have the contribution to the sum vector from a single snake, showing that the total sum vector has period $\lambda$ is even easier than it was in the previous case. In particular, using a co-slither of $S^2L^{r-1}$ is equivalent to taking $\lambda$ sum vectors, each subsequently shifted $2$ positions, and then adding one more sum vector shifted $1$ position. The first $\lambda$ contributions to the sum vectors must add up to a constant sum because $\gcd(2,\lambda) = 1$. Thus, the total sum vector has period $\lambda$ because a single sum vector has period $\lambda$.  
\end{proof}

There is nothing special about the constructions used for Theorem~\ref{thm:snakeconstruction}, and we expect that there are many other classes of ticker tapes satisfying the conditions of the theorem. The challenge was to incorporate just enough asymmetry so as not to shrink the period of the sum vector.  One idea for future research would be to characterize the relative frequencies of different sum vector periods.

\section{Concluding remarks}\label{sec:conclusions}

We began the work described in this paper thinking it was a fun but fairly narrow and self-contained problem about toggling independent sets of cycle graphs. What we encountered was an unexplored mathematical theory that is applicable to other combinatorial actions. The iteration of a fixed Coxeter element $\tau$ in any generalized toggle group defines a finite dynamical system, and the concepts of a scroll, ticker tape, and orbit table all carry over. There is nothing special about these objects on their own, as they are just ways to describe and visualize the dynamics. What makes this problem unique is the fact that there are two commuting bijections, the successor and co-successor functions, that act simply transitively on the live entries. As far as we know, studying actions on the live entries is new, and it endows the orbits of the global action with a signature algebraic structure. 

It is natural to ask which other combinatorial actions also have torsor structures on their orbits. For this to have any chance of working, there needs to be strong structural regularity of the underlying graph. For example, in the current paper, the automorphism group of the graph---the dihedral group of order $2n$---acts transitively on the vertices, and the action commutes with the toggle operation. It is hard to see how the orbits of a general action could admit any meaningful algebraic structure without this. In other words, we should be looking at other vertex-transitive graphs that are in some sense similar to cycle graphs. In ongoing work, we have found torsor structures on orbits from toggling independent sets over distance-$2$ cycle graphs and from toggling other combinatorial objects over cycle graphs. 

In another direction, toggling independent sets can be formalized as one of the 256 elementary cellular automata (CA) rules. Specifically, a CA is a regular grid of cells, where each cell has a Boolean state in $\ff_2=\{0,1\}$, and the state of a vertex $v$ is updated at each timestep based on the states of $v$ and its neighbors. In the $1$-dimensional setting, a ``grid'' is simply a cycle graph or an infinite path graph. In either case, the update function is defined by some $f\colon\ff_2^3\to\ff_2$. Each one is characterized by its \emph{truth table}, a generic example of which is shown in Eq.~\eqref{eq:eca}.

\begin{equation}\label{eq:eca}
\begin{tabular}{c||c|c|c|c|c|c|c|c} 
  $x_{i-1}x_ix_{i+1}$ & 111 & 110 & 101 & 100 & 011 & 010 & 001 & 000 
  \\ \hline 
  $f^{(k)}_i(x_{i-1},x_i,x_{i+1})$ & $b_7$ & $b_6$ & $b_5$ & $b_4$ & $b_3$ & $b_2$ & $b_1$ & $b_0$  \\ 
\end{tabular}
\end{equation}
Since each $b_i$ is in $\ff_2$, there are $2^8=256$ possible functions; each such function is indexed by the integer $k\in\{0,1,\dots,255\}$, which is the integer whose binary representation is $b_7b_6b_5b_4b_3b_2b_1b_0$. The one indexed by $k$ is called the \emph{elementary cellular automata (ECA) rule $k$}. In this setting, the toggle functions introduced in this paper can be realized as ECA rule $1$. Technically, the bits $b_7$, $b_6$, and $b_3$ can be anything because the the substrings $111$, $110$, and $011$ do not appear in the set $\call$ of independent sets of $\calc_n$. However, it is arguably more natural to use ECA rule $1$, also known as the \emph{logical NOR} function. ECA rules are defined on all $2^n$ states, and the periodic points of ECA rule $1$ are precisely the independent sets. 

Casting the work from this paper in the setting of cellular automata poses new questions that would likely not be asked from within the dynamical algebraic combinatorics community. For example, which ECA rules lead to dynamics whose live entries admit a torsor structure? For such a toggle action to be defined, the local functions need to act on the set of periodic points, and this happens for precisely $104$ ECA rules, or $41$ up to the equivalences defined by reflection and inversion. These rules were classified by the third author with McCammond and Mortveit in \cite{macauley2008order}, and their toggle groups were studied in \cite{macauley2011dynamics}, though under the name of ``dynamics groups.'' Preliminary investigation has revealed that most of these $41$ ECA rules do not admit an interesting torsor structure, but largely for mundane reasons. For example, the toggle group is trivial for $26$ of the $41$ rules. For more on the connections between dynamical algebraic combinatorics and cellular automata, the interested reader can consult our paper in the proceedings of the annual AUTOMATA conference \cite{david2021toggling}. The first two thirds of that paper is a survey bringing these fields together, and the last third is an ``extended abstract'' of this current paper. In it, we also propose a number of open-ended problems in both fields, inspired by ideas and themes of the other one. It is our hope that the present work, and the ideas within, catches the interest of researchers with different backgrounds in these fields and beyond.

\section*{Acknowledgements}
We are grateful to the Banff International Research Station and the organizers of the Dynamical Algebraic Combinatorics Workshop, which was held virtually in 2020. We thank Laurent David for several substantial ideas and conversations during the initial stages of this project. We also thank Adam Dzedzej and David Einstein for helpful discussions. The first author was supported by the National Science Foundation under Award No.\ DGE--1656466 and Award No.\ 2201907, by a Fannie and John Hertz Foundation Fellowship, and by a Benjamin Peirce Fellowship at Harvard University. The third author was supported by Simons Foundation Collaboration Grant \#358242.
\bibliography{bibliography}
\bibliographystyle{halpha}
\end{document}